\documentclass[oneside]{amsart}
\usepackage{graphicx} 
\usepackage{amsthm}
\usepackage{amsmath}
\usepackage{amsfonts}
\usepackage{amssymb}
\usepackage{enumitem}
\usepackage[bottom=47mm]{geometry} 
\usepackage[dvipsnames]{xcolor}
\usepackage{bbm}
\usepackage{hyperref}
\usepackage{soul} 

\usepackage{pgfplots,tikz}
\usetikzlibrary{patterns, shapes.geometric, intersections, calc}
\definecolor{imayou}{RGB}{154, 154, 235}
\definecolor{usuai}{RGB}{9, 150, 126}
\definecolor{persred}{RGB}{154,63,63}
\definecolor{sand}{RGB}{201, 177, 60}

\usepackage{mathtools}
\mathtoolsset{showonlyrefs=true}

    \newcommand{\ti}{\tilde}
    \newcommand{\e}{\varepsilon}
    \renewcommand{\d}{\delta}
    \newcommand{\Zb}{\mathbb{Z}}
    \newcommand{\supp}{\text{supp }}
    \newcommand{\nm}[1]{\left|\left| #1 \right|\right|}
    \newcommand{\lp}[2]{\nm{#1}_{L^{#2}}}
    \newcommand{\ltwo}[1]{\lp{#1}{2}}
    \newcommand{\ra}{\rightarrow}
    \newcommand{\p}{\partial}
    \newcommand{\Cs}{C^{\infty}}
    \newcommand{\Cc}{C^{\infty}_c}
    \newcommand{\<}{\left<}
    \renewcommand{\>}{\right>}
    \renewcommand{\Im}{\text{Im}}

    \newcommand{\Rb}{\mathbb{R}}
    \newcommand{\Trf}{T^*\Rb^4}
    \newcommand{\Trfo}{T^*\Rb^4\setminus 0}

    \newcommand{\AF}[1]{\nm{#1}_{AF}}
    \newcommand{\Nb}{\mathbb{N}}
    \newcommand{\Ltxaj}[1]{\nm{#1}_{L^2_t L^2_x(\Rb_+ \times A_j)}}
    \newcommand{\LtxTAj}[1]{\nm{#1}_{L^2_t L^2_x([0,T] \times A_j)}}
    \newcommand{\LEsltx}[1]{\nm{#1}_{LE^*+L^1_t L^2_x}}
    \newcommand{\LET}[1]{\nm{#1}_{LE[0,T]}}
    \newcommand{\LEoT}[1]{\nm{#1}_{LE^1[0,T]}}
    \newcommand{\LEST}[1]{\nm{#1}_{LE^*[0,T]}}
    \newcommand{\LEsltxT}[1]{\nm{#1}_{LE^*+L^1_t L^2_x[0,T]}}
    
        \newcommand{\characteristicsetofP}{\text{Char}({P})}
        \newcommand{\characteristicsetoftildeP}{\text{Char}({\tilde P})}
        \newcommand{\characteristicsetofPplusminus}{\text{Char}({P}^\pm)}
        \newcommand{\characteristicsetofPplus}{\text{Char}({P}^+)}
        \newcommand{\characteristicsetofPminus}{\text{Char}({P}^-)}
        \newcommand{\forwardtrappedset}{\Gamma_{tr}}
        \newcommand{\backwardtrappedset}{\Lambda_{tr}}
        \newcommand{\forwardnontrappedset}{\Gamma_\infty}
        \newcommand{\backwardnontrappedset}{\Lambda_\infty}
        \newcommand{\trappedset}{\Omega^{p}_{\textit{tr}}}
        \newcommand{\nontrappedset}{\Omega^{p}_{\infty}}
    \newcommand{\ppm}{p^{\pm}}
        \newcommand{\plusminusforwardtrappedset}{\Gamma_\textit{tr}^\pm}
        \newcommand{\plusminusbackwardtrappedset}{\Lambda_\textit{tr}^\pm}
        \newcommand{\plusminusforwardnontrappedset}{\Gamma_\infty^\pm}
        \newcommand{\plusminusbackwardnontrappedset}{\Lambda_\infty^\pm}
        \newcommand{\plusminustrappedset}{\Omega_{tr}^\pm}
        \newcommand{\plusminusnontrappedset}{\Omega_{\infty}^\pm}
            \newcommand{\plusforwardtrappedset}{\Gamma_\textit{tr}^+}
            
            \newcommand{\plusforwardnontrappedset}{\Gamma_\infty^+}
            
            \newcommand{\plustrappedset}{\Omega_{tr}^+}
            \newcommand{\plusnontrappedset}{\Omega_{\infty}^+}

            \newcommand{\minustrappedset}{\Omega_{tr}^-}
            \newcommand{\minusnontrappedset}{\Omega_{\infty}^-}
        \newcommand{\homogplusminusforwardtrappedset}{\mathring{\Gamma}_\textit{tr}^\pm}
        \newcommand{\homogplusminusbackwardtrappedset}{\mathring{\Lambda}_\textit{tr}^\pm}

        \newcommand{\homogplusminustrappedset}{\mathring{\Omega}_{tr}^\pm}
        
            \newcommand{\homogplustrappedset}{\mathring{\Omega}_{tr}^+}
            
            \newcommand{\homogminustrappedset}{\mathring{\Omega}_{tr}^-}
            
    \newcommand{\Hpm}{H_{p^{\pm}}}

    \newcommand{\bpm}{b^{\pm}}
    \newcommand{\bpmx}{\bpm(x,\xi)}

    \newcommand{\xpms}{x^{\pm}_s}
    \newcommand{\xipms}{\xi^{\pm}_s}
    \newcommand{\tpms}{t^{\pm}_s}
    \newcommand{\taupms}{\tau^{\pm}_s}
    \newcommand{\bpmxs}{\bpm(\xpms,\xipms)}
    \newcommand{\gpms}{\varphi_s^{\pm}}
    \newcommand{\gpmsd}{\mathring{\varphi}_s^{\pm}}

    \newcommand{\xpmso}{\xpms(\omega)}
    \newcommand{\tpmso}{\tpms(\omega)}
    \newcommand{\xipmso}{\xipms(\omega)}
    \newcommand{\taupmso}{\taupms(\omega)}
    \newcommand{\zpms}{z^{\pm}_s}
    \newcommand{\zetapms}{\zeta^{\pm}_s}
    \newcommand{\zpmso}{\zpms(\omega)}
    \newcommand{\zetapmso}{\zetapms(\omega)}

    \newcommand{\Phipm}{\Phi^{\pm}}
\newcommand{\Cm}{\overline{C}}

\newcommand{\Piz}{\Pi_z}
\newcommand{\Pizeta}{\Pi_{\zeta}}

\newcommand{\pmftrappedcompactset}{\Gamma_R^{\pm}}
\newcommand{\pmbtrappedcompactset}{\Lambda_R^{\pm}}
\newcommand{\pmsemitrappedset}{\mathcal{T}_{\leq R}^{\pm}}
\newcommand{\homogpmsemitrappedset}{\mathring{\mathcal{T}}_{\leq R}^{\pm}}
\newcommand{\pmhomogftrappedcompactset}{\mathring{\Gamma}_R^{\pm}}
\newcommand{\pmhomogbtrappedcompactset}{\mathring{\Lambda}_R^{\pm}}

\newcommand{\semitrappedescapeset}{V_R^{\pm}}
\newcommand{\homogsemitrappedescapeset}{\mathring{V}_R^{\pm}}
\newcommand{\qpm}{q^{\pm}}
\newcommand{\homogpmcompactsemi}{\mathring{\mathcal{T}}^{\pm}_{\leq R, T_2}}

\newcommand{\ac}{\mathcal{A}}
\newcommand{\acpm}{\ac^{\pm}}
\newcommand{\rpm}{r^{\pm}}
\newcommand{\qpmo}{q^{\pm}_1}
\newcommand{\acpmo}{\mathcal{A}_1^{\pm}}
\newcommand{\rpmo}{r^{\pm}_1}

    \newcommand{\qo}{q_{\omega}^{\pm}}
    \newcommand{\aco}{\ac_{\omega}^{\pm}}
    \newcommand{\alphaopm}{\alpha_{\omega}^{\pm}}
    \newcommand{\ro}{r_{\omega}^{\pm}}
    \newcommand{\rhopm}{\rho_{\omega}}
    \newcommand{\opmj}{\omega_j^{\pm}}
    \newcommand{\opmjk}{\omega_{j,k}}

\newcommand{\gopmr}[1]{\varphi_{#1}^{\pm}(\omega)}
\newcommand{\gopms}{\gopmr{s}}
\newcommand{\Topm}{T_{\omega}}
\newcommand{\gopm}{\varphi^{\pm}(\omega)}
\newcommand{\Psiopm}{\Psi_{\omega}}
\newcommand{\phiopm}{\phi_{\omega}}
\newcommand{\chiopm}{\chi_{\omega}}

\newcommand{\Ci}{C^{\infty}}

\newcommand{\qpmin}{\qpm_{in}}
\newcommand{\qpmint}{\widetilde{q}^{\pm}_{in}}
\newcommand{\qpmout}{\qpm_{out}}

\newcommand{\psipm}{\psi^{\pm}}
\newcommand{\phipm}{\varphi^{\pm}}
\newcommand{\Cpm}{C^{\pm}}
\newcommand{\cpm}{c^{\pm}}

\newcommand{\Ypm}{Y^{\pm}}

\newcommand{\qpmlg}{q^{\pm}_{>1}}
\newcommand{\qmplg}{q^{\mp}_{>1}}
\newcommand{\qmlg}{q^{-}_{>1}}
\newcommand{\qplg}{q^{+}_{>1}}
\newcommand{\qpmt}{\qpm_2}
\newcommand{\bmp}{b^{\mp}}

\newcommand{\chigl}{\chi_{>1}}

\newcommand{\chilr}{\chi_{<R}}
\newcommand{\chigr}{\chi_{>R}}
\newcommand{\chir}{\chi_R}
\newcommand{\chilro}{\chi_{<R_0}}
\newcommand{\chigro}{\chi_{>R_0}}
\newcommand{\vgl}{v_{>\lambda}}
\newcommand{\vll}{v_{<\lambda}}
\newcommand{\vggl}{v_{>>\lambda}}

\newcommand{\LE}[1]{\nm{#1}_{LE}}
\newcommand{\LEs}[1]{\nm{#1}_{LE^*}}
\newcommand{\LEo}[1]{\nm{#1}_{LE^1}}
\newcommand{\Ltx}[1]{\nm{#1}_{L^2_t L^2_x}}
\newcommand{\LtxT}[1]{\nm{#1}_{L^2_t L^2_x[0,T]}}
\newcommand{\Ltauxi}[1]{\nm{#1}_{L^2_{\tau}L^2_{\xi}}}
\newcommand{\chiggl}{\chi_{|\xi|+|\tau|>\frac{\lambda}{2}}}
\newcommand{\chill}{\chi_{|\xi|+|\tau|<\frac{\lambda}{2}}}
\newcommand{\chitl}{\chi_{|\tau|>1}}

\newcommand{\vglgl}{v_{>\lambda,>1}}
\newcommand{\vlgl}{v_{>\lambda,<1}}
\newcommand{\vgll}{v_{< \lambda, >\sigma \lambda}}
\newcommand{\vlsl}{v_{<\lambda,<\sigma \lambda}}

\newcommand{\Op}{\text{Op}}
\newcommand{\chihigh}{\chi_{\tau \geq \tau_1}}
\newcommand{\Ophw}{\text{Op}_h^w}
\newcommand{\chilow}{\chi_{\tau \leq \tau_0}}
\newcommand{\chimed}{\chi_{<>\tau}}

\theoremstyle{plain}
\newtheorem{definition}{Definition}[section]
\newtheorem{theorem}[definition]{Theorem}
\newtheorem{lemma}[definition]{Lemma}
\newtheorem{remark}[definition]{Remark}
\newtheorem{proposition}[definition]{Proposition}
\numberwithin{equation}{section}

\title[Local energy decay with time-dependent damping]{Integrated Local Energy Decay for Waves with Time-Dependent Damping}
\author{Perry Kleinhenz and Michael McNulty}
\date{}

\begin{document}

\begin{abstract}
    We prove integrated local energy decay for solutions of the damped wave equation with time-dependent damping satisfying an appropriate generalization of the geometric control condition on asymptotically flat, stationary space-times. We first obtain a high frequency estimate, which we prove via a positive commutator estimate using an escape function explicitly constructed in terms of the damping around individual space-time trajectories. We  combine the high frequency estimate with low and medium frequency results for the undamped problem, then we handle the damping term as a perturbation to obtain local energy decay.
\end{abstract}
\maketitle
\section{Introduction}

    Let $(\Rb^4,g)$ be a Lorentzian manifold with coordinates $(t,x) \in \Rb \times \Rb^3$, where $g$ has signature $(- + + +)$. Let $\Box_g = D_{\alpha}g^{\alpha \beta}D_{\beta}$ for $\alpha, \beta \in \{0,1,2,3\}$. We consider the Cauchy problem for damped wave operators 
    \begin{equation}
        \begin{cases}
            Pu=(\Box_g+iaD_t)u=f \\
            u[0]=(u(0), \p_t u(0)) \in \dot{H}^1(\Rb^3) \times L^2(\Rb^3),
        \end{cases}
    \end{equation}
    where $a \in C^{\infty}(\Rb \times \Rb^3)$ is non-negative and uniformly continuous. Our goal is to show that the energy of solutions, measured with spatial weights, decays quickly enough to be integrable in time. This is \textit{integrated local energy decay}, which we define precisely below. 

    We will focus on the case where the operator $P$ is asymptotically flat. That is, for large values of $|x|$, $g$ is close to the Minkowski metric.
    \begin{definition}\label{d:asymptoticFlat1}
        Let $\p=(\p_t, \nabla)$ be the space-time gradient, and $\<x\>=(1+|x|^2)^{1/2}$ be the Japanese angle bracket of $x$. 
    
        For $j \geq 0$, let $A_j=\{x \in \Rb^3: 2^{j-1} \leq \<x\> \leq 2^{j+1}\}$ and define the family of norms
        \begin{equation}
            \|h\|_k = \sum_{|\alpha| \leq k} \nm{ \<x\>^{|\alpha|} \p^{\alpha} h}_{\ell^1_jL^{\infty}(\Rb \times A_j)},
        \end{equation}
        where $\ell^1_j$ denotes the $\ell^1$ norm over the $j$ index.

        We define the $AF$ norm as 
        \begin{equation}
            \AF{(h,a)} = \|h\|_2 + \|\<x\> a\|_1.
        \end{equation}
        Now, letting $m$ denote the Minkowski metric, we say $P$ is asymptotically flat if 
        \begin{equation}
            \AF{(g-m,a)}<\infty,
        \end{equation} 
        and for multi-indices $\alpha$ there exists $C_{\alpha}>0$, such that 
        \begin{align}
            &\nm{\<x\>^{|\alpha|} \p^{\alpha} g}_{\ell^1_j L^{\infty}(\Rb \times A_j)} \leq C_{\alpha}, & |\alpha| \geq 3,\\
            &\nm{\<x\>^{|\alpha|+1} \p^{\alpha} a}_{\ell^1_j L^{\infty}(\Rb \times A_j)} \leq C_{\alpha},  & |\alpha| \geq 2.
        \end{align}
    \end{definition}
    We also require that the metric $g$ is independent of time.
    \begin{definition}
        We say that $P$ is stationary if the metric $g$ is independent of $t$.
    \end{definition}
    Now, to make our notion of local energy precise, we define local energy norms.
    \begin{definition} We write $L^p_t L^q_x=L^p(\Rb_+, L^q(\Rb^3))$, and define
    \begin{align}
        &\LE{u} = \sup_{j \geq 0} \Ltxaj{\<x\>^{-1/2} u},\\
        &\LEo{u}=\LE{\p u} + \LE{\<x\>^{-1} u},\\
        &\LEs{f}=\sum_{j=0}^{\infty} \Ltxaj{\<x\>^{1/2} f},\\
        &\LEsltx{f}=\inf_{f =f_1+f_2} \left( \LEs{f_1}+\nm{f_2}_{L^1_tL^2_x}\right).
    \end{align}
         We will use the notation $\LET{u}$, $\LEoT{u} $, $ \LEST{u}$, $ \LEsltxT{u}$ to represent these norms with time interval $[0,T]$ instead of $\Rb_+$. We write $LE^*_c$ to refer to the elements of $LE^*$ with compact support. 
    \end{definition}
    Our final preliminary is to state a rough version of our time-dependent geometric control condition. We provide an exact statement of this in Definition \ref{Defn:TGCC}.
    \begin{definition}
        We say that the time-dependent geometric control condition holds if there exists $T_0, \Cm>0$ such that for all trapped space-time trajectories $(t_s,x_s)$ with $|t_s'|=c$ and $T \geq T_0/c$, we have 
        \begin{equation}
            \frac{1}{2T}\int_{-T}^T a(t_s, x_s) ds \geq \Cm.
        \end{equation}
    \end{definition}
    We are now ready to state our main result.
    \begin{theorem}\label{thm:iled}
        Let $P$ be a stationary asymptotically flat damped wave operator with non-negative time-dependent damping satisfying the time-dependent geometric control condition, and suppose $\p_t$ is uniformly time-like and constant time-slices are uniformly space-like. 
        Then local energy decay holds. That is, there exists $C>0$ such that for all $T>0$ and all $u$ with $u[0] \in \dot{H}^1 \times L^2$, we have
        \begin{equation}
            \LEoT{u}+\nm{\p u}_{L^{\infty}_tL^2_x[0,T]} \leq C \left( \ltwo{\p u(0)} + \nm{Pu}_{LE^*+L^1_t L^2_x[0,T]}\right).
        \end{equation}
    \end{theorem}
    If $u$ solves $Pu=0$, this result roughly says that the local energy of $u$, measured via 
            \begin{equation}
                \sup_{j \geq 0} \nm{\<x\>^{-1/2} \p u(t, \cdot)}_{L^2_x(A_j)} + \sup_{j \geq 0} \nm{\<x\>^{-3/2} u(t,\cdot)}_{L^2_x(A_j)}, 
            \end{equation}
            decays quickly enough in $t$ to be $L^2_t$ integrable, with a uniform upper bound given by the initial energy $\ltwo{\p u(0)}$. It is in this sense that the result provides integrated local energy decay. 
    \begin{remark}\normalfont
        \begin{enumerate}
            \item This result generalizes that of \cite[Theorem 1.7]{Kofroth23} by allowing the damping $a$ to depend on time. We also assume only that the damping $a$ is asymptotically flat, rather than compactly supported, although the stationary asymptotically flat case is addressed in \cite{KofrothMagnetic}. Our result also generalizes the integrated version of \cite[Theorem 1.1]{BoucletRoyer2014} from a Euclidean background and stationary damping, to a Lorentzian background and time-dependent damping. 
            \item If there are no trapped space-time trajectories, then the time-dependent geometric control condition is always satisfied. In this case we only require that the damping $a$ is nonnegative and asymptotically flat; it is otherwise free. In this case our result is a partial generalization of \cite[Theorems 2.12(a), 2.15, and 2.16]{MST20} because we do not require our $P$ to be $\e$-slowly varying or $\e$-almost symmetric. That is, we can write
            \begin{align}
            &P=(D_{\alpha} +A_{\alpha} )g^{\alpha \beta} (D_{\beta} + A_{\beta}) +V(t,x),\quad  \text{ where}\\
                 &2g^{\alpha \beta} A_{\alpha} := i a(t,x) \delta^{\beta}_{\;0},
                 \qquad V(t,x) := -\left( g^{\alpha \beta}(D_{\alpha} A_{\beta}) + (D_{\alpha} g^{\alpha \beta}) A_{\beta}+A_{\alpha} g^{\alpha \beta} A_{\beta} \right).
            \end{align}
            Our result gives local energy decay for this $P$ which is non-trapping, asymptotically flat and has stationary $g$. Note the $A_{\alpha}(t,x)$ and $V(t,x)$ terms need not be $\e$-slowly varying or $\e$-almost symmetric, c.f. \cite[Definitions 1.2, 1.3]{MST20}. Because of this \cite{MST20} does not provide integrated local energy decay for such a $P$. Note however the high frequency result \cite[Theorem 2.11]{MST20} applies to this $P$, and so it is our argument in Section \ref{s:localEnergyDecay} that provides this partial generalization.     
            \item Conversely, if $a=0$, then there cannot be any trapping. Then we are in a special case of
            \cite[Theorem 2.12(b)]{MST20}. Using their notation we have $A=V=0$ and so local energy decay occurs \cite[Section 8.3]{MST20}.

            \item We have written the wave operator in divergence form, $\Box_g = D_{\alpha} g^{\alpha \beta} D_{\beta}$, as opposed to Laplace-Beltrami form
            \begin{equation}
                \Box_g^{LB} = |g|^{-1/2} D_{\alpha} |g|^{1/2} g^{\alpha \beta} D_{\beta}, \quad |g| = |\det g^{\alpha \beta}|.
            \end{equation}
            We can transition between these forms by conjugating the wave operator by $|g|^{1/4}$ at the cost of lower order potential terms, see \cite[Section 2.2]{Tataru13} or \cite[Proposition 2.2]{Morgan24}.
        \end{enumerate}
    \end{remark}
    Theorem \ref{thm:iled} follows from what we call a \textit{high-frequency integrated local energy decay} result. 
\begin{theorem}\label{thm:highfreq}
    Let $P$ be a stationary asymptotically flat damped wave operator with non-negative time-dependent damping satisfying the time-dependent geometric control condition, and suppose $\p_t$ is uniformly time-like and constant time-slices are uniformly space-like, then there exists $C>0$ such that for all $T>0$, and $u$ with $u[0] \in \dot{H}^1 \times L^2$
    \begin{equation}
        \LEoT{u} + \nm{\p u}_{L^{\infty}_t L^2_x[0,T]} \leq C \left( \ltwo{\p u(0)}+\LET{\<x\>^{-2} u}  + \LEsltxT{Pu}\right).
    \end{equation}
\end{theorem}
We call this a high-frequency result because when we apply it to $\chi_{\tau \geq \tau_1}^w u$, by applying Plancherel's theorem, we have 
    \begin{align}
        \LE{\<x\>^{-2} \chihigh^w u} &\leq C \nm{\<x\>^{-2} \chihigh \hat{u}(\tau,x)}_{LE_{\tau,x}} \\
        &\leq \frac{C}{\tau_1} \nm{\<x\>^{-2} \tau \chihigh \hat{u}(\tau,x)}_{LE_{\tau,x}} \leq \frac{C}{\tau_1} \LEo{\chihigh^w u}.
    \end{align}
Thus taking $\tau_1$ large enough, we can absorb this term back into the left hand side of the estimate, and be left with a right hand side resembling that of Theorem \ref{thm:iled}. We expand on this in the Outline \ref{outline:ILED} and Section \ref{sec:combineILED}.

\subsection{Additional Preliminaries}

We further define some constants related to the asymptotic flatness of $g$.

\begin{definition}\label{d:asymflat2}
    \begin{enumerate}
        \item Fix a $\textbf{c}\ll1$, and let  $R_0>0$ be such that 
        \begin{equation}
            \nm{(g-m,a)}_{AF \geq R_0} \leq \textbf{c},
        \end{equation}
        where the subscript denotes the restriction of the norm to $\{|x| \geq R_0\}$. Note that for any $\textbf{c}>0$, such an $R_0$ is guaranteed to exist, by the asymptotic flatness of $g$.

        \item Consider a sequence $c_j$ such that for some $C>0$,
        \begin{equation}
            \nm{(g-m,a)}_{AF(A_j)} \leq C c_j, \quad \text{ and } \quad  \sum_j c_j \leq C \textbf{c},
        \end{equation}
        where
        \begin{equation}
            \nm{(h,a)}_{AF(A_j)} = \sum_{|\alpha| \leq 2} \nm{\<x\>^{|\alpha|} \p^{\alpha} h}_{L^{\infty}(A_j)} + \sum_{|\beta| \leq 1} \nm{\<x\>^{1+|\beta|} \p^{\beta}a}_{L^{\infty}(A_j)}.
        \end{equation}
        We may further assume, without loss of generality, that the sequence is slowly varying, that is there exists $\d\ll1$ such that 
        \begin{equation}
            \frac{c_j}{c_k} \leq 2^{\d|k-j|}.
        \end{equation}
        In particular, there exists $C \geq 1$ such that 
        \begin{equation}
            \frac{1}{C} 2^{-\d j} \leq c_j \leq C 2^{-\d j}.
        \end{equation}
    \end{enumerate}
\end{definition}

\subsection{Cutoff Notation}\label{s:cutoffdef}
    Throughout the paper we use the following notation for cutoffs. Let $\chi \in \Cc(\Rb)$ be non-increasing and have $\chi(x) \equiv 1$ for $|x| \leq 1$ and $\chi(x) \equiv 0$ for $|x| \geq 2$. Then for any $R>0$ define
    \begin{align}
        &\chilr(x) = \chi\left(\frac{|x|}{R}\right), \qquad \chigr(x)=1-\chilr(x)\\
        &\chir \in \Cc(\Rb),\quad  0\leq \chir\leq1, \quad \supp\chir \subset \{R/2 \leq |x| \leq R\}.
    \end{align}
    Note we will often use this notation with $\lambda$ or other constants in place of $R$.

    \subsection{Literature Review}
    Local energy decay estimates on Minkowski space-times go back to \cite{Mor66, Mor68, Mor75, MRS77} and similar estimates have been obtained for small perturbations of Minkowski space \cite{KSS02, KPV95, SmithSogge2000, Sterbenz2005, Strauss75, Alinhac2006, MS2006, MS2007, MT2009}. Local energy decay has also been proven for
    asymptotically flat space-times with no trapping \cite{BH2009, MST20}. When there is trapping, local energy decay cannot occur \cite{Ralston69, Sbierski15}. However when trapping is allowed, local energy decay with a loss can still be recovered, \cite{ NZ2009, WZ2011, BCMP17}. Local energy decay with a loss has also been proven on black-hole backgrounds, see \cite{LT20} and the references therein. 

    Energy decay for the damped wave equation via a geometric control condition goes back to \cite{RauchTaylor1974}. See also \cite{Lebeau1996} and \cite{BurqJoly2016}. Using a stationary damping on the trapped set to obtain local energy decay estimates goes back to \cite{AlouiKhenissi2002} on an exterior domain, and \cite{BoucletRoyer2014} when $\Rb^4$ is stationary and asymptotically Euclidean. This was brought to the space-time setting in \cite{Kofroth23} and further generalized to include potential terms in \cite{KofrothMagnetic}, although both still require a stationary damping.

    The geometric control condition was used to establish energy decay for a damping $W(t,x)$ with periodic $t$ dependence in \cite{LRLTT17}. This was generalized to fully time-dependent damping in \cite{Kleinhenz2022a}, \cite{kleinhenz2025sharp}.

    Local energy decay estimates can be used to prove pointwise energy decay results (for example see \cite{MTT12} \cite{Tataru13}, \cite{MTT17}, \cite{Looi22b}, \cite{Looi22c}, \cite{Looi22a}, \cite{Looi23}, \cite{LooiTohaneanu2025}, \cite{MW21}, \cite{Morgan24}) and Strichartz estimates (for example see \cite{MMTT10}, \cite{MT12}, \cite{Tohaneanu12}).    

\subsection{Outline of the Proof}
    In this section, we outline the proof of Theorem \ref{thm:iled} and highlight the key novelties. 
    
\subsubsection{Outline of ILED Proof}\label{outline:ILED}
    By Proposition \ref{prop:iledCaseReduction}, Theorem \ref{thm:iled} holds if there exists $C>0$ such that for all $u$ with zero Cauchy data at $t=0$ and $t=T$ and $Pu\in LE^*_c$, we have
        \begin{equation}\label{eqn:iledReduced-outline}
            \LEoT{u} \leq C \LEST{Pu}.
        \end{equation}
    In Section \ref{s:localEnergyDecay}, we prove \eqref{eqn:iledReduced-outline} by decomposing $u$ into low, medium, and high time-frequencies
        \begin{equation}\label{eq:timefreqdecomp}
            u = \chilow^w u + (1-\chilow^w-\chihigh^w) u + \chihigh^w u,
        \end{equation}
    for $\tau_0,\tau_1>0$, then controlling the $LE^1$ norm of each of the three terms separately.

    To estimate the low time-frequency term, we treat $a \p_t u$ as a perturbation and apply a zero non-resonance condition satisfied by $\Box_g$ \eqref{eq:zernoresonance}, to obtain
        $$
            \LEo{\chilow^w u} \leq C \left(\LEs{Pu} + \LEs{Pu}^{1/2} \LEo{u}^{1/2}\right).
        $$
    
    For the high time-frequency term, we utilize our high frequency estimate Theorem \ref{thm:highfreq}. The proof of the high frequency estimate relies on our novel escape function construction in Section \ref{s:escapeFunction}. Our construction incorporates the time-dependent geometric control condition to overcome new difficulties introduced by the time-dependence of the damping. We outline this in more detail in Sections \ref{outline:highfrequencyestimateproof} and \ref{outline:escapefunction}. Combining Theorem \ref{thm:highfreq} with a semiclassical estimate of $[a,\chihigh^w]$, we obtain
        $$
            \LEo{\chihigh^w u} \leq C \nm{Pu}_{LE^*}+\frac12\LEo u.
        $$
    
    For the medium time-frequency term, we use a Carleman estimate \eqref{eq:carleman} from \cite{MST20} to obtain
        $$
            \LEo{(1-\chilow^w-\chihigh^w) u} \leq C \left( \LEs{Pu} + \LEs{Pu}^{1/2} \LEo{u}^{1/2}\right).
        $$
    The proof of \eqref{eqn:iledReduced-outline} follows from these three estimates, along with applications of the triangle inequality and Young's inequality.

\subsubsection{Outline of the High Frequency Estimate Proof}\label{outline:highfrequencyestimateproof}
    By Proposition \ref{prop:CaseReduction}, the high frequency estimate holds if there exists $C>0$ such that for all $T>0$ and $u$ with spatial support contained in $\{|x|\leq 2R_0\}$, zero Cauchy data at $t=0$ and $t=T$, and $Pu\in LE^*_c$, we have
        \begin{equation}\label{eqn:reducedhighfreqest-outline}
            \nm{u}_{LE^1[0,T]} \leq C \left( \nm{u}_{L^2_t L^2_x[0,T]} + \nm{Pu}_{LE^*[0,T]} \right).
        \end{equation}
    In Section \ref{s:propagation}, we prove \eqref{eqn:reducedhighfreqest-outline} by performing a space-frequency decomposition into low and high frequencies
        \begin{equation}
            u=\chi_{|\xi|< \lambda}^w u+\chi_{|\xi|>\lambda}^w u, 
        \end{equation}
    for $\lambda \geq 1$ and control the $LE^1$ norm of each of the two terms separately. 
    
    To estimate the low space-frequency term, we 
    use a further time-frequency decomposition to obtain
        \begin{equation}
            \LEo{\chi_{|\xi|< \lambda}^w u} \leq C \left( \sigma \lambda \Ltx{u}+ \frac{1}{ \sigma \lambda} \LEs{Pu} + \frac{1}{\sigma} \LEo{u} \right), \label{eqn:highlow-outline}
        \end{equation}
    for an additional parameter $\sigma \geq 1$. 
    
    For the high space-frequency term, we use a positive commutator argument. Our implementation of this argument relies crucially on our escape function construction. Indpendent of our construction, we first note that for any two symbols $q\in S^1(T^*\mathbb R^4)$ and $m\in S^0(T^*\mathbb R^4)$, we can consider the operator $Q:=q^w-\frac i2m^w\in\Psi^1(\mathbb R^4)$ and compute $\Im\<Pu,Qu\>$ two different ways to obtain
        \begin{align*}
            2 \Im&\<Pu,Qu\> + \frac{i \kappa}{2}\<[aD_t, m^w]u,u\>-\kappa \<(D_t a) q^w u,u\> + \frac{i \kappa}{2} \<(D_t a) m^w u,u\> 
            \\
            &= \<i[\Box_g, q^w]u,u\> + \kappa \<(q^w aD_t + aD_t q^w)u,u\> + \frac{1}{2}\<(\Box_g m^w + m^w \Box_g)u,u\>.
        \end{align*}
    For the first line, we obtain the bound
        \begin{align*}
            |2 \Im&\<Pu,Qu\> + \frac{i \kappa}{2}\<[aD_t, m^w]u,u\>-\kappa \<(D_t a) q^w u,u\> + \frac{i \kappa}{2} \<(D_t a) m^w u,u\>|
            \\
            &\leq C \LEs{Pu} \LEo{u} + C(\lambda) \Ltx{u}^2+C \lambda^{-\frac{1}{2}} \LEo{u}^2.
        \end{align*}
    From our escape function construction in Section \ref{s:escapeFunction}, we will have $q$ and $m$ such that
        \begin{align}
        \begin{split}
            \<i[\Box_g, q^w]u,u\> + \kappa \<(q^w aD_t + aD_t q^w)u,u\> + \frac{1}{2}\<(\Box_g m^w + m^w \Box_g)u,u\> \\
            \geq C \nm{\chi_{|\xi|>\lambda}^w u}_{LE^1}^2 - C(\lambda) \Ltx{u}^2-C(\lambda^{-1}+\rho^{-2})\LEo{u}^2. \label{eqn:gardingapplication-outline}
        \end{split}
        \end{align}
        We outline our novel construction of $q$ and $m$ that achieves this inequality in Section \ref{outline:escapefunction}.
    
    The proof of the high frequency estimate \eqref{eqn:reducedhighfreqest-outline} concludes by applying \eqref{eqn:highlow-outline} and \eqref{eqn:gardingapplication-outline} and taking $\lambda,\sigma$ sufficiently large to close the estimate.

\subsubsection{Outline of the Escape Function Construction}\label{outline:escapefunction}
    Crucial to the proof of the high frequency estimate is \eqref{eqn:gardingapplication-outline}. Our construction of symbols $q\in S^1(T^*\mathbb R^4)$ and $m\in S^0(T^*\mathbb R^4)$ --- namely our escape function construction from Section \ref{s:escapeFunction} --- such that \eqref{eqn:gardingapplication-outline} holds is the primary contribution of the present work. In this section, we provide a brief overview of the desired properties of these symbols and describe the key steps in their construction. A more detailed outline is provided in Section \ref{sec:escapeoutline} before we carry out the construction. 
    
    To find $q\in S^1(T^*\mathbb R^4)$ and $m\in S^0(T^*\mathbb R^4)$ which satisfy \eqref{eqn:gardingapplication-outline}, we first note that for any such symbols, we have
        \begin{align}
        \begin{split}
            \< i[\Box_g, q^w]u,u\>& + \kappa \<(q^w aD_t + aD_t q^w)u,u\> + \frac{1}{2} \<(\Box_g m^w + m^w \Box_g)u,u\> \\
            =&\<(H_p q+ 2\kappa \tau a q+mp)^w u,u\>
        \end{split}
        \end{align}
    modulo an error term. To obtain \eqref{eqn:gardingapplication-outline}, we bound $H_p q+ 2\kappa \tau a q+mp$ from below and then apply the Sharp G\r{a}rding inequality (see Proposition \ref{prop:Garding}). Specifically, we construct symbols $q\in S^1(T^*\mathbb R^4)$ and $m\in S^0(T^*\mathbb R^4)$ such that for some $C>0$, and all $\omega =(t,x,\tau,\xi) \in T^*\Rb^4$
        \begin{equation}
            H_p q(\omega)+2\kappa \tau a(\omega) q(\omega)+mp(\omega) \geq C \mathbbm{1}_{|\xi| \geq 1}\mathbbm{1}_{|\tau| \geq 1} \<x\>^{-4} (|\xi|^2 + \tau^2). \label{ol:escape function estimate}
        \end{equation}
    One of the key difficulties we overcome is obtaining a uniform $C$ for all $t$, despite the time-dependence of the damping $a$. 
    
    We begin by constructing $q$ on the characteristic set of $P$. Here, there are two regions to consider: semi-trapped null bicharacteristics and non-trapped null bicharacteristics. 
    
    For the semi-trapped null bicharacteristics, the trapping is an obstruction to integrated local energy decay, and one of our key innovations is how we address this region.
     Here we explicitly construct escape functions in terms of the damping in local coordinates around individual null bicharacteristics. Using the time-dependent geometric control condition we are able to ensure that several key properties hold uniformly in $t$ for these different escape functions. This uniformity in $t$ allows us to cover the semi-trapped set with a locally finite number of these escape functions. 
     
     We combine these local escape functions to obtain a single escape function $q$ that satisfies global in time symbol estimates, and 
        $$
            H_pq+2\kappa\tau aq\geq C \mathbbm{1}_{V},
        $$
    where $V$ is an open set containing the semi-trapped region, and which is invariant in $t$. This is essential in obtaining a constant $C$ uniform in $t$ in \eqref{ol:escape function estimate}.

    For the non-trapped null bicharacteristics, because $V$ is invariant in $t$, we are able to separate this step from the damping and its time-dependence. Thus our construction is an adaptation of \cite{Kofroth23,BoucletRoyer2014,MST20,MMT08}.
    
    We then combine all of these escape functions into a single escape function $q$ defined on the characteristic set of $P$ satisfying
        $$
            H_pq+2\kappa\tau aq\geq C.
        $$
    After that, we construct $m$ on the elliptic set of $P$ in order to extend this estimate (in the sense of \eqref{ol:escape function estimate}) to all of $T^*\mathbb R^4$. At that point, we are able to apply the Sharp G\r{a}rding inequality as previously mentioned and complete the proof of \eqref{eqn:iledReduced-outline}.

\subsubsection{Structure of the Paper}
    The remainder of the paper proceeds as follows. In Section \ref{sec:Hamiltonian Dynamics}, we summarize the Hamiltonian dynamics associated to $P$ and its half-wave factorization. This will allow us to precisely state the time-dependent geometric control condition and its consequences which we will utilize in the escape function construction. In Section \ref{s:escapeFunction}, we carry out our escape function construction. In Section \ref{s:caseReduce}, we reduce the proofs of our main results, Theorem \ref{thm:iled} and Theorem \ref{thm:highfreq}, to the proofs of \eqref{eqn:iledReduced-outline} and \eqref{eqn:reducedhighfreqest-outline} respectively. In Section \ref{s:propagation}, we prove \eqref{eqn:reducedhighfreqest-outline} and, as a consequence, Theorem \ref{thm:highfreq}. Finally, in Section \ref{s:localEnergyDecay}, we prove \eqref{eqn:iledReduced-outline} and, as a consequence, Theorem \ref{thm:iled}.

    \subsection{Acknowledgements}
    The authors would like to thank Andras Vasy, Mihai Tohaneanu, Jared Wunsch, and Willie Wong for helpful conversations. The authors would also like to thank Jason Metcalfe for helpful correspondence. 

    The second author thanks the NSF for partial support under grant DMS-2530465.
    
\section{Hamiltonian Dynamics}\label{sec:Hamiltonian Dynamics}
    In this section, we summarize the relevant Hamiltonian dynamics associated to the operator $P$ and the half-wave factorization of its principal symbol. We use these definitions and basic results to precisely state the time-dependent geometric control condition and two of its consequences, which we will use in the escape function construction. Our general approach follows that of \cite[Section 2.2]{Kofroth23}, although we must work on $\Trf$ rather than $T^*\Rb^3$ to handle the time-dependence of the damping. Also, our geometric control condition and related proofs are necessarily different, and we handle $g^{00}$ differently. See also \cite[Section 8]{BoucletRoyer2014}.

    \subsection{Hamiltonian Flow for the Principal Symbol}    
    The principal symbol of $P$ is
        \begin{equation}
            p(t,x,\tau,\xi)=g^{00}(x)\tau^2+2\tau g^{0j}(x)\xi_j+g^{ij}(x)\xi_i\xi_j,
        \end{equation}
    viewed as a function on $\Trfo$ with $0$ denoting the zero section. This symbol generates the Hamiltonian flow map, $\varphi:\mathbb R\times \Trf\to \Trf$, denoted by 
        \begin{equation}
            \varphi_s(\omega)=\left(t_s(\omega),x_s(\omega),\tau_s(\omega),\xi_s(\omega)\right),
        \end{equation} 
    and defined as the exponential of the Hamilton vector field
    \begin{equation}
        H_p = \p_{\zeta} p  \p_z - \p_z p \p_{\zeta},
    \end{equation}
    where $z=(t,x), \zeta=(\tau,\xi)$. This can be defined as a system of differential equations
    \begin{align}
        \begin{cases}
            \frac{d}{ds}{t}_s=\partial_\tau p\left(\varphi_s(\omega)\right), & \frac{d}{ds}{\tau}_s=-\partial_tp\left(\varphi_s(\omega)\right),
            \\
            \frac{d}{ds}{x}_s=\nabla_\xi p\left(\varphi_s(\omega)\right), & \frac{d}{ds}{\xi}_s=-  \nabla_xp\left(\varphi_s(\omega)\right).
        \end{cases}
    \end{align} 
    The existence and uniqueness of a smooth, globally-defined flow with smooth dependence on the initial data follows from $g$ being smooth and asymptotically flat in addition to $\partial_t$ being a uniformly timelike vector field.
    
    Associated to the flow $\varphi$, we define the \textit{forward} and \textit{backward trapped sets}
        \begin{align}
            &\forwardtrappedset=\left\{\omega\in \Trfo:\sup_{s\geq0}|x_s(\omega)|<\infty\right\}\cap\characteristicsetofP,
            \\
            &\backwardtrappedset=\left\{\omega\in \Trfo:\sup_{s \geq0}|x_{-s}(\omega)|<\infty\right\}\cap\characteristicsetofP
        \end{align}
        respectively. We also define the \textit{forward} and \textit{backward non-trapped sets}
        \begin{align}
            &\forwardnontrappedset=\left\{\omega\in \Trfo:\lim_{s\to\infty}|x_s(\omega)|=\infty\right\}\cap\characteristicsetofP,
            \\
            &\backwardnontrappedset=\left\{\omega\in \Trfo:\lim_{s\to\infty}|x_{-s}(\omega)|=\infty\right\}\cap\characteristicsetofP.
        \end{align}
    The \textit{trapped} and \textit{non-trapped} sets are defined to be
        \begin{equation}
            \trappedset=\forwardtrappedset\cap\backwardtrappedset \qquad \text{ and } \qquad \nontrappedset=\forwardnontrappedset\cap\backwardnontrappedset.
        \end{equation}
    To refer to the coordinates of $\omega$, we write $\omega=(\omega_t, \omega_x, \omega_{\tau}, \omega_{\xi})$. 
    We can now precisely state the time-dependent geometric control condition associated to the damping function $a$.
    \begin{definition}[Time-Dependent Geometric Control Condition (TGCC)] \label{Defn:TGCC}
        We say that time-dependent geometric control holds if there exist $\Cm,T_0>0$ such that for every $\omega \in\trappedset$ and $T\geq \frac{T_0}{\omega_{\tau}}$ 
            \begin{equation}
                \frac{1}{2T}\int^T_{-T}a\left(t_s(\omega), x_s(\omega)\right)ds\geq \Cm
            \end{equation}
    \end{definition}
    \begin{remark}
        \begin{enumerate}
            \item When the damping does not depend on time, the TGCC holds if and only if the $x$-projection of every trapped trajectory eventually enters the damped set $\{a>0\}$, which is \cite[Definition 2.2]{Kofroth23}. For a proof of this statement see Lemma \ref{l:gccequiv}.
            \item Formulating the geometric control condition in this way for stationary damping  goes back to \cite{Lebeau1996}. See also \cite{BurqJoly2016}. This was first applied to time-dependent damping in \cite{Kleinhenz2022a,kleinhenz2025sharp}. See also \cite{LRLTT17}. 
            \item When $a \equiv 0$, this is a qualitative non-trapping assumption. Because $g$ is stationary we are then back in the setting of \cite{MST20}, see their Definition 1.4 and subsequent remark.
            \item As we will show in Lemma \ref{l:escapeFlow}, $\trappedset \subset \{|x| \leq R_0\}$. Because of this, the TGCC can be satisfied by damping which are non-trivial only in a compact spatial region, for example $a(t,x)=(1+\<t\>^{-1}) \chi_{<2R_0}(|x|)$ satisfies the TGCC.
        \end{enumerate}

    \end{remark}
    \begin{remark}
        Note we do not make the somewhat standard simplifying assumption that $g^{00}=-1$, as it cannot be done without loss of generality in our setup. Recall that we assumed $\p_t$ is uniformly time-like and constant time slices are uniformly space-like. As a result, there exists $C>0$ such that $g^{00} \leq -C$. The standard argument, for example see \cite[Section 3]{MT12}, is to divide by $g^{00}$. That is let $(-g^{00})^{-1} P :=P_1$, and $g_1^{\alpha \beta} = (-g^{00})^{-1} g^{\alpha\beta}, a_1 = (-g^{00})^{-1} a$. Then $g_1^{00}=-1$ and we have 
    \begin{align}
        P_1 &= D_{\alpha} (-g^{00})^{-1} g^{\alpha \beta} D_{\beta} + [(-g^{00})^{-1}, D_{\alpha}]g^{\alpha \beta} D_{\beta} + (-g^{00})^{-1} a \p_t\\
        &= D_{\alpha} g_1^{\alpha \beta} D_{\beta} - (D_{\alpha} (g^{00})^{-1}) g^{\alpha \beta} D_{\beta} + a_1 \p_t.
    \end{align}
    Notice that we are left with a lower order error term which cannot be written as part of the divergence form. Because we require our operator to be exactly of the form $\Box_g+a\p_t$ elsewhere, we are not able to absorb these lower order terms as in \cite[Section 4]{MST20} and \cite[Section 3]{MT12}. However, we still take advantage of $g^{00} \geq -C$ and effectively simplify to $g^{00}=-1$ via our half-wave decomposition in Section \ref{s:half-wave}.
    \end{remark}

    \subsection{Behavior of the Flow Under Rescaling}
        In the construction of the escape function, it is useful to replace the damping $a$, by a multiple $\kappa a$. To accomplish this, we take advantage of a scaling property of $P$

        Given a solution $u$ of $Pu=f$ and $\kappa \geq 1$, define 
            \begin{align}
                \tilde v(t,x)&=\kappa^{-2}u(\kappa t,\kappa x), \\
                \tilde g^{\alpha\beta}(x)&=g^{\alpha\beta}(\kappa x),\\
                \tilde a(t,x)&=a(\kappa t,\kappa x),\\
                \tilde f(t,x)&=f(\kappa t,\kappa x),\\
                \ti{p}(t,x,\xi,\tau)&=p( \kappa t, \kappa x, \tau,  \xi),
            \end{align}
        and
            \begin{equation}
                \tilde P=D_\alpha\tilde g^{\alpha\beta}D_\beta+i\kappa\tilde aD_t.
            \end{equation}
        A direct calculation verifies that $\tilde v$ solves $\tilde P\tilde v=\tilde f$. We have the following proposition concerning the behavior of the time-dependent geometric control condition under rescaling.

        \begin{proposition}
            Suppose Definition \ref{Defn:TGCC} holds and recall $T_0, \Cm$ from there. For the same $T_0, \Cm>0$ and any $\kappa \geq 1$, Definition \ref{Defn:TGCC} holds with respect to the Hamiltonian flow generated by the principal symbol of $\tilde P$ with the damping $a$ replaced by $ \tilde a$.
        \end{proposition}
        \begin{proof}
            A direct calculation shows that a scaled version of the Hamiltonian flow generated by $\tilde p$, which we will call $\tilde{\varphi}_s$, with initial data $(t_0,x_0,\tau_0,\xi_0)$ solves the same system of ordinary differential equations as the Hamiltonian flow generated by $p$ with initial data $(\kappa t_0,\kappa x_0,\kappa \tau_0,\kappa \xi_0)$. In particular, we have that 
                \begin{equation}
                 \varphi_s(\kappa t_0,\kappa x_0,\kappa \tau_0,\kappa \xi_0)=\kappa\tilde \varphi_s(t_0,x_0,\tau_0,\xi_0).\label{eq:dampingflowrescale}
                \end{equation}
            Let $\tilde \omega=( \ti{t},\ti{x}, \ti{\tau}, \ti{\xi}) \in\Omega^{\tilde p}_{tr}$. Since $\tilde\omega\in\characteristicsetoftildeP$, if we compute directly and let $\omega=\kappa \ti{\omega}$ we have
            \begin{align}
                0&=\kappa^2 \ti{p}(\ti{\omega})=\kappa^2 p(\kappa \ti{t}, \kappa \ti{x}, \ti{\tau}, \ti{\xi}) \\
                &= \kappa^2( g^{00}(\kappa \ti{x}) \ti{\tau}^2 + 2 \ti{\tau} g^{0j}(\kappa \ti x) \ti{\xi}_j + g^{ij}(\kappa \ti x) \ti{\xi}_i \ti{\xi}_j) \\
                &=p(\kappa \ti{t}, \kappa \ti{x}, \kappa \ti{\tau}, \kappa \ti{\xi}) = p (\omega).
            \end{align}
            That is $\omega \in \characteristicsetofP$. Furthermore,
                \begin{equation}
                \sup_{s\in\mathbb R}|x_s(\omega)|=\kappa\sup_{s\in\mathbb R}|\tilde x_s(\tilde\omega)|<\infty,
                \end{equation}
            and so $\omega\in\trappedset$. Now by \eqref{eq:dampingflowrescale} and Definition \ref{Defn:TGCC}, for any $T\geq \frac{T_0}{\ti{\tau}} = \kappa \frac{T_0}{\omega_{\tau}} \geq \frac{T_0}{\omega_{\tau}}$,
                \begin{align*}
                    \frac{1}{2T}\int_{-T}^T  \tilde a\left(\tilde t_s(\ti \omega), \tilde x_s(\ti{\omega})\right)ds= \frac{1}{2T}\int_{-T}^T  a\left(t_s( \omega), x_s(\omega)\right)ds\geq \Cm.
                \end{align*}
                Since this holds for any $\ti{\omega} \in \Omega^{\ti{p}}_{tr}$, this is exactly Definition \ref{Defn:TGCC} holding for the Hamiltonian flow generated by $\ti{p}$ with $a$ replaced by $\ti{a}$.
        \end{proof}
        Using this proposition, without loss of generality we replace $a$ by $\kappa a$ for some large $\kappa \geq 1$. We fix the value of $\kappa$ during our escape function construction, specifically in the proof of Lemma \ref{l:combineEscapeFunc}.
    \subsection{The Half-Wave Decomposition}\label{s:half-wave}
        When working with null-bicharacteristics, it will be convenient to avoid the cross-terms involving both $\tau$ and $\xi$ in the principal symbol $p$. This can be done by factoring the principal symbol as follows
            \begin{equation}
                p(t,x,\tau,\xi)=g^{00}(x)\left(\tau-b^+(x,\xi)\right)\left(\tau-b^-(x,\xi)\right), \label{Eqn:HalfWaveFactorization}
            \end{equation}
        where
            \begin{equation}\label{eq:bdef}
                b^\pm(x,\xi)=\left(\frac{g^{0j}(x)}{-g^{00}(x)} \xi_j\pm\sqrt{\left(\frac{g^{0j}(x)}{-g^{00}(x)}  \xi_j\right)^2+\frac{g^{ij}(x)}{-g^{00}(x)}\xi_i\xi_j} \right).
            \end{equation}
        In particular, note that $b^\pm(x,\xi)$ are both homogeneous of degree one in the variable $\xi$. 
        
        Furthermore, $\bpm$ are signed and satisfy symbol estimates.
        \begin{lemma}
            For any $(x,\xi) \in T^* \Rb^3 \backslash 0$, we have 
            \begin{equation}
                b^+(x,\xi) >  0 > b^-(x,\xi).
            \end{equation}
            Additionally $\bpm \in S^1(\Trfo)$, where $S^1$ is defined in Definition \ref{def:Symbol}.
        \end{lemma}
        \begin{proof}
        Recall that since $\p_t$ is uniformly time-like and constant time-slices are uniformly space-like, $g^{00} \leq -C$. 
            By ellipticity of $g^{ij}$ we have 
            \begin{equation}
                \sqrt{\left(\frac{g^{0j}}{-g^{00}}  \xi_j\right)^2+\frac{g^{ij}}{-g^{00}}\xi_i\xi_j} > \left|\frac{g^{0j}}{-g^{00}} \xi_j\right|.
            \end{equation}
            Thus
            \begin{equation}
                b^+ > \frac{g^{0j}}{-g^{00}} \xi_j + \left|\frac{g^{0j}}{-g^{00}} \xi_j\right| \geq 0, \qquad b^- < \frac{g^{0j}}{-g^{00}} \xi_j -\left|\frac{g^{0j}}{-g^{00}} \xi_j \right|\leq 0.
            \end{equation}
            Asymptotic flatness of $g$, and $-g^{00} \geq C$ show that $\bpm$ satisfies the symbol estimates in Definition \ref{def:Symbol}.
        \end{proof}

        We define the \textit{half-wave symbols} as
            \begin{equation}\label{eq:halfWave}
                p^{\pm}(t,x,\tau,\xi)=\tau-b^{\pm}(x,\xi).
            \end{equation}
        According to \eqref{Eqn:HalfWaveFactorization}, since $g^{00} \neq 0$, $p=0$ if and only if either $p^+=0$ or $p^-=0$. In particular, the signs of $b^{\pm}$ guarantee that there does not exist $\omega\in \Trfo$ such that $p^+(\omega)=p^-(\omega)=0$.

        We associate to $p^{\pm}$ their Hamiltonian flow maps $\varphi^\pm:\Rb \times \Trf\to \Trf$ denoted by 
            \begin{equation}
                \varphi_s^\pm(\omega)=\left(t_s^\pm(\omega),x_s^\pm(\omega),\tau_s^\pm(\omega),\xi_s^\pm(\omega)\right),
            \end{equation}
        and defined as the exponential of the Hamilton vector field 
        \begin{equation}
            \Hpm = \p_{\zeta} p^{\pm} \p_z - \p_z p^{\pm} \p_{\zeta},
        \end{equation}
        or defined via the system of differential equations 
        \begin{equation}
            \label{eq:flowdef}
            \begin{cases}
                \frac{d}{ds}{t}_s^\pm=\partial_\tau p^\pm\left(\varphi_s^\pm(\omega)\right), & \frac{d}{ds}{\tau}_s^\pm=-\partial_tp^\pm\left(\varphi_s^\pm(\omega)\right),
                \\
                \frac{d}{ds}{x}_s^\pm=\nabla_\xi p^\pm\left(\varphi_s^\pm(\omega)\right), & \frac{d}{ds}{\xi}_s^\pm=-  \nabla_xp^\pm\left(\varphi_s^\pm(\omega)\right).
            \end{cases}
        \end{equation}
        Moreover, we define the forward and backward trapped sets associated to the half-wave flows $\varphi^\pm$ as
            \begin{align}
                &\plusminusforwardtrappedset=\left\{\omega\in \Trfo:\sup_{s\geq0}|x_s^\pm(\omega)|<\infty\right\}\cap\characteristicsetofPplusminus
                \\
                &\plusminusbackwardtrappedset=\left\{\omega\in \Trfo:\sup_{s \geq0}|x_{-s}^\pm(\omega)|<\infty\right\}\cap\characteristicsetofPplusminus,
            \end{align}
       where $P^\pm$ are the Weyl quantizations of $p^\pm$, see Definition \ref{def:Pseudo}. We similarly define the forward and backward non-trapped sets associated to $\varphi^{\pm}$ as 
            \begin{align}
                &\plusminusforwardnontrappedset=\left\{\omega\in \Trfo:\lim_{s\to\infty}|x_s^\pm(\omega)|=\infty\right\}\cap\characteristicsetofPplusminus
                \\
                &\plusminusbackwardnontrappedset=\left\{\omega\in \Trfo:\lim_{s\to\infty}|x_{-s}^\pm(\omega)|=\infty\right\}\cap\characteristicsetofPplusminus.
            \end{align}
        The corresponding trapped and non-trapped sets are
            \begin{equation}
                \plusminustrappedset=\plusminusforwardtrappedset\cap\plusminusbackwardtrappedset \qquad \text{ and } \qquad \plusminusnontrappedset=\plusminusforwardnontrappedset\cap\plusminusbackwardnontrappedset.
            \end{equation}

        The decomposition into $p^+$ and $p^-$ is convenient because there are no cross terms involving both $\tau$ and $\xi$ in the Hamilton flow, and null bicharacteristics of $p$ correspond to null bicharacteristics of $p^{\pm}$ as described in the following proposition.
        \begin{lemma}\label{l:CorrespondenceOfFlows}
            Every null bicharacteristic for the flow generated by $p$ is a reparameterization of a null bicharacteristic for the flow generated by either $p^+$ or $p^-$. The converse is also true.
        \end{lemma}
        Before proving this we mention that as a consequence, we have the following relationship between the trapped and non-trapped sets of $p$ and $p^\pm$
            \begin{equation}
                \trappedset=\plustrappedset\cup\minustrappedset \qquad \text{ and } \qquad \nontrappedset=\plusnontrappedset\cup\minusnontrappedset.
            \end{equation}
        \begin{proof}
            For the forward direction, let $\omega=(t_0,x_0,\tau_0,\xi_0)\in\characteristicsetofP$. Then either $p^+(\omega)=0$ or $p^-(\omega)=0$ but not both simultaneously. Without loss of generality, take $p^+(\omega)=0$.

            Consider the bicharacteristic flows $\varphi_{(\cdot)}(\omega):\mathbb R\to \Trf$ and $\varphi_{(\cdot)}^+(\omega):\mathbb R\to \Trf$ generated by $p$ and $p^+$ respectively, and starting at $\omega$.
            By definition, these flow maps solve the system of differential equations
                \begin{equation}\label{eq:PflowODE}
                    \begin{cases}
                        \frac{d}{ds}t_s(\omega)=\partial_\tau p\left(\varphi_s(\omega)\right), & \frac{d}{ds}\tau_s(\omega)=-\partial_tp\left(\varphi_s(\omega)\right),
                        \\
                        \frac{d}{ds} x_s(\omega)=\nabla_\xi p\left(\varphi_s(\omega)\right), & \frac{d}{ds}\xi_s(\omega)=-  \nabla_xp\left(\varphi_s(\omega)\right),
                    \end{cases}
                \end{equation}
            and
                \begin{equation}\label{eq:PplusflowODE}
                    \begin{cases}
                        \frac{d}{ds'}t_{s'}^+(\omega)=\partial_\tau p^+\left(\varphi_{s'}^+(\omega)\right), & \frac{d}{ds'}\tau_{s'}^+(\omega)=-\partial_tp^+\left(\varphi_{s'}^+(\omega)\right),
                        \\
                        \frac{d}{ds'} x_{s'}^+(\omega)=\nabla_\xi p^+\left(\varphi_{s'}^+(\omega)\right), & \frac{d}{ds'}\xi_{s'}^+(\omega)=-  \nabla_xp^+\left(\varphi_{s'}^+(\omega)\right).
                    \end{cases}
                \end{equation}
            Using that $p=g^{00} p^+p^-$, and that null bicharacteristics of $P$ are contained in $\characteristicsetofP$, we have 
            \begin{align}
                \frac{d}{ds} \varphi_s(\omega) = H_p = H_{g^{00}p^+ p^-} = p^+ p^- H_{g^{00}} + g^{00} H_{p^+ p^-} = 0 + g^{00} H_{p^+ p^-},
            \end{align}
            since $p=0$ exactly when $p^+ p^-=0$ because $g^{00} \neq 0$.
            
            Using this, along with $p^\pm=\tau-b^\pm$, we can rewrite \eqref{eq:PflowODE} as
                \begin{equation}\label{eq:PflowODE2}
                    \begin{cases}
                        \frac{d}{ds}t_s(\omega)=g^{00}\left(\varphi_s(\omega)\right) (p^+\left(\varphi_s(\omega)\right)+p^-\left(\varphi_s(\omega)\right)),
                        \\
                        \frac{d}{ds} x_s(\omega)=g^{00}\left(\varphi_s(\omega)\right)(p^+\left(\varphi_s(\omega)\right)\nabla_\xi p^-\left(\varphi_s(\omega)\right)-p^+\left(\varphi_s(\omega)\right)\nabla_\xi p^+\left(\varphi_s(\omega)\right)), 
                        \\
                        \frac{d}{ds}\tau_s(\omega)=0,
                        \\
                        \frac{d}{ds}\xi_s(\omega)=-g^{00}\left(\varphi_s(\omega)\right)(p^+\left(\varphi_s(\omega)\right)\nabla_x p^-\left(\varphi_s(\omega)\right)+p^-\left(\varphi_s(\omega)\right)\nabla_xp^+\left(\varphi_s(\omega)\right)).
                    \end{cases}
                \end{equation}
            and we can rewrite \eqref{eq:PplusflowODE} as
                \begin{equation}
                    \begin{cases}
                        \frac{d}{ds'}t_{s'}^+(\omega)=1, & \frac{d}{ds'}\tau_{s'}^+(\omega)=0,
                        \\
                        \frac{d}{ds'} x_{s'}^+(\omega)=\nabla_\xi p^+\left(\varphi_{s'}^+(\omega)\right), & \frac{d}{ds'}\xi_{s'}^+(\omega)=-  \nabla_xp^+\left(\varphi_{s'}^+(\omega)\right),
                    \end{cases} \label{eq:pluscharflow}
                \end{equation}
            with $\omega$ still as the initial condition.

            Now, noting that $p^+(\omega)=0$, we claim that for all $s\in\mathbb R$, $p^+\left(\varphi_s(\omega)\right)=0$. To see this, suppose by way of contradiction that there exists $s_0\in\mathbb R$ such that $p^+\left(\varphi_{s_0}(\omega)\right)\neq0$. Since $p\left(\varphi_{s_0}(\omega)\right)=0$, it follows that $p^-\left(\varphi_{s_0}(\omega)\right)=0$. Using the explicit form of $p^-$, this implies that $\tau_{s_0}(\omega)=b^-(\varphi_{s_0}(\omega))<0$. However, \eqref{eq:PflowODE2} implies that for all $s\in\mathbb R$, $\tau_s(\omega)=\tau_0$. Since $p^+(\omega)=0$, we must also have that $\tau_0=b^+(\omega)>0$. Thus, we have a contradiction. As a consequence, \eqref{eq:PflowODE2} simplifies to
                \begin{equation} 
                    \begin{cases}
                        \frac{d}{ds}t_s(\omega)=g^{00}\left(\varphi_s(\omega)\right) p^-\left(\varphi_s(\omega)\right),
                        \\
                        \frac{d}{ds} x_s(\omega)=g^{00}\left(\varphi_s(\omega)\right) p^-\left(\varphi_s(\omega)\right)\nabla_\xi p^+\left(\varphi_s(\omega)\right), 
                        \\
                        \frac{d}{ds}\tau_s(\omega)=0,
                        \\
                        \frac{d}{ds}\xi_s(\omega)=-g^{00}\left(\varphi_s(\omega)\right) p^-\left(\varphi_s(\omega)\right)\nabla_xp^+\left(\varphi_s(\omega)\right).
                    \end{cases} \label{eq:charflow}
                \end{equation}
                
            We now reparameterize \eqref{eq:charflow}. Define $s(r)$ via
                \begin{equation}
                    \frac{d}{dr}s(r)=\frac{1}{g^{00}(\varphi_{s(r)}(\omega)) p^-\left(\varphi_{s(r)}(\omega)\right)},\quad s(0)=0.
                \end{equation}
            By the chain rule, we have
                \begin{equation}
                    \begin{cases}
                        \frac{d}{dr}t_{s(r)}(\omega)=1,& \frac{d}{dr}\tau_{s(r)}(\omega)=0,
                        \\
                        \frac{d}{dr} x_{s(r)}(\omega)=\nabla_\xi p^+\left(\varphi_{s(r)}(\omega)\right), &\frac{d}{dr}\xi_{s(r)}(\omega)=-\nabla_xp^+\left(\varphi_{s(r)}(\omega)\right)
                    \end{cases}
                \end{equation}
            with $\omega$ as the initial condition. Note that this is the same set of equations that $\varphi^+_s$ satisfies in \eqref{eq:pluscharflow} and $\varphi_0(\omega)=\varphi_0^+(\omega)=\omega$. Therefore by uniqueness of solutions to systems of ordinary differential equations, for all $r\in\mathbb R$, 
            \begin{equation}
                \varphi_{s(r)}(\omega)=\varphi^+_r(\omega).
            \end{equation}

            The converse proceeds similarly. Let $\omega=(t_0,x_0,\tau_0,\xi_0)\in\characteristicsetofPplusminus$. Then $\omega\in\characteristicsetofP$. We now reparameterize \eqref{eq:PplusflowODE}. Define $s'(r')$ via
                \begin{equation}
                    \frac{d}{dr'}s'(r')=g^{00}(\varphi_{s'(r')}^{\pm}(\omega))p^{\mp}\left(\varphi_{s'(r')}^\pm(\omega)\right),\quad s'(0)=0.
                \end{equation}
            By the chain rule, we get
                \begin{equation}
                    \begin{cases}
                        \frac{d}{dr'}t_{s'(r')}^\pm(\omega)=g^{00}\left(\varphi_{s'(r')}^\pm(\omega)\right) p^\mp\left(\varphi_{s'(r')}^\pm(\omega)\right),
                        \\
                        \frac{d}{dr'} x_{s'(r')}^\pm(\omega)= g^{00}\left(\varphi_{s'(r')}^\pm(\omega)\right) p^\mp\left(\varphi_{s'(r')}^\pm(\omega)\right)\nabla_\xi p^\pm\left(\varphi_{s'(r')}^\pm(\omega)\right),
                        \\
                        \frac{d}{dr'}\tau_{s'(r')}^\pm(\omega)=0,
                        \\
                        \frac{d}{dr'}\xi_{s'(r')}^\pm(\omega)=- g^{00}\left(\varphi_{s'(r')}^\pm(\omega)\right) p^\mp\left(\varphi_{s'(r')}^\pm(\omega)\right)\nabla_xp^\pm\left(\varphi_{s'(r')}^\pm(\omega)\right),
                    \end{cases}
                \end{equation}
             with $\omega$ as the initial condition. Note this is the same set of equations that $\varphi$ satisfies in \eqref{eq:charflow}, and $\varphi_0(\omega)=\varphi_0^{\pm}(\omega)=\omega$. Therefore by uniqueness of solutions to systems of ordinary differential equations, for all $r'\in\mathbb R$, 
                $$
                    \varphi^\pm_{s'(r')}(\omega)=\varphi_{r'}(\omega).
                $$
        \end{proof}
        In the construction of the escape function, it will be useful to first construct functions for space-time frequencies at a fixed scale. 
        We then extend the initial construction to all frequency scales (away from 0) via homogeneity. To facilitate this argument, we record the behavior of the half-wave flow under such a rescaling. 
        \begin{proposition}\label{Propn:RescalingPropn}
            For any $(t,x,\tau,\xi)\in\Trf$ and $\lambda>0$, the Hamiltonian flows generated by $p^\pm$ satisfy the following scaling relations
                \begin{equation}
                    \begin{cases}
                        t_s^\pm(t,x,\tau,\xi)=t_s^\pm(t, x, \lambda \tau, \lambda \xi)
                        \\
                        x_s^\pm(t,x,\tau,\xi)=x_s^\pm(t, x, \lambda \tau, \lambda \xi)
                         \\
                        \lambda\tau_s^\pm(t,x,\tau,\xi)=\tau_s^\pm(t, x, \lambda \tau, \lambda \xi)
                        \\
                        \lambda\xi_s^\pm(t,x,\tau,\xi)=\xi_s^\pm(t, x, \lambda \tau, \lambda \xi)
                    \end{cases}
                \end{equation}
        \end{proposition}
        \begin{proof}
            For notational convenience, define the following functions
                \begin{equation}
                    \begin{cases}
                        
                        t_{s,\lambda}^\pm(t,x,\tau,\xi)=t_s^\pm(t, x, \lambda \tau, \lambda \xi)
                        \\
                        x_{s,\lambda}^\pm(t,x,\tau,\xi)=x_s^\pm(t, x, \lambda \tau, \lambda \xi)                    
                        \\
                        \tau_{s,\lambda}^\pm(t,x,\tau,\xi)=\tau_s^\pm(t, x, \lambda \tau, \lambda \xi)
                        \\
                        \xi_{s,\lambda}^\pm(t,x,\tau,\xi)=\xi_s^\pm(t, x, \lambda \tau, \lambda \xi).
                    \end{cases}
                \end{equation}
            Recall that $b^\pm$, and thus $p^{\pm}$, is homogeneous in $\xi$, while $\nabla_{\xi} p^{\pm}$, and $\nabla_x p^{\pm}$ do not depend on $\tau$. Therefore we have
                \begin{equation}
                    \begin{cases}
                        \frac{d}{ds}t_s^\pm=\partial_\tau p^{\pm}\left(t_s^\pm,x_s^\pm, \tau_s^\pm, \xi_s^\pm\right)=1
                        \\
                        \frac{d}{ds} x_s^{\pm} = \nabla_{\xi} p^{\pm}(t_s^\pm,x_s^\pm, \tau_s^\pm, \xi_s^\pm) = \nabla_{\xi} p^{\pm}(t_s^{\pm}, x_s^{\pm}, \lambda \tau_s^{\pm}, \lambda \xi_s^{\pm})\\
                        \frac{d}{ds}\left(\lambda\tau_s^\pm\right)=-\lambda\partial_t p^{\pm}\left(t_s^\pm,x_s^\pm, \tau_s^\pm, \xi_s^\pm\right)=0
                        \\
                        \frac{d}{ds} (\lambda \xi_s^{\pm}) = -\lambda \nabla_x p^{\pm}(t_s^{\pm}, x_s^{\pm},\tau_s^{\pm}, \xi_s^{\pm}) =-\nabla_x p^{\pm}(t_s^{\pm}, x_s^{\pm},\lambda \tau_s^{\pm}, \lambda \xi_s^{\pm})\\
                        \left(t_s^\pm,x_s^{\pm}, \lambda\tau_s^\pm, \lambda \xi_s^{\pm} \right)|_{s=0}=\left(t,x,\lambda\tau, \lambda \xi \right).
                    \end{cases}
                \end{equation}
            Similarly,
                \begin{equation}
                    \begin{cases}
                        \frac{d}{ds}t_{s,\lambda}^\pm=\partial_\tau p^{\pm}(t_{s,\lambda}^\pm,x_{s,\lambda}^\pm,\tau_{s,\lambda}^\pm,\xi_{s,\lambda}^\pm)=1
                        \\
                        \frac{d}{ds} x_{s,\lambda}^{\pm} = \nabla_{\xi} p^{\pm}(t_{s,\lambda}^{\pm}, x_{s,\lambda}^{\pm}, \tau_{s,\lambda}^{\pm}, \xi_{s,\lambda}^{\pm})\\
                        \frac{d}{ds}\tau_{s,\lambda}^\pm=-\partial_t p^{\pm}(t_{s,\lambda}^\pm,x_{s,\lambda}^\pm,\tau_{s,\lambda}^\pm,\xi_{s,\lambda}^\pm)=0\\
                        \frac{d}{ds} \xi_{s,\lambda}^{\pm} = -\nabla_x p^{\pm}(t_{s,\lambda}^{\pm}, x_{s,\lambda}^{\pm}, \tau_{s,\lambda}^{\pm}, \xi_{s,\lambda}^{\pm})
                        \\
                        (t_{s,\lambda}^\pm, x_{s,\lambda}^{\pm}, \tau_{s,\lambda}^\pm, \xi_{s,\lambda}^{\pm} )|_{s=0}=\left(t,x,\lambda\tau, \lambda \xi\right).
                    \end{cases}
                \end{equation}
            These systems are the same, so by uniqueness of solutions to systems of ordinary differential equations, 
            \begin{equation}
                (t_s^{\pm}, x_s^{\pm}, \lambda \tau_s^{\pm}, \lambda \xi_s^{\pm}) = (t^{\pm}_{s,\lambda}, x^{\pm}_{s,\lambda}, \tau_{s,\lambda}^{\pm}, \xi_{s,\lambda}^{\pm}).
            \end{equation}
        \end{proof}

        We will make extensive use of a particular rescaling adapted to the Hamiltonian flows generated by $p^\pm$. Define $\Phi^{\pm}:\Trf\to\Trf$  by
                \begin{equation}
                    \Phi^\pm(t,x,\tau,\xi)=\left(t,x,\frac{\tau}{\bpmx},\frac{\xi}{b^\pm(x,\xi)}\right).
                \end{equation}

            We note that because $p^{\pm}=0=\tau-\bpm$ along its null bicharacteristics, and $\tau$ is a constant along these null bicharacteristics, $\bpm$ is constant along them as well. Furthermore there exist $0<c<C$ such that 
                \begin{equation}\label{eq:xibhomog}
                    c\leq \left|\frac{\xi}{b^\pm(x,\xi)}\right| \leq C.
                \end{equation}
            The proof of this follows from the definition of $\bpm$ and asymptotic flatness of $g$. For details see \cite[Proposition 2.8]{Kofroth23}. With this map, we define the following rescaled, forward/backward trapped sets associated to the half-wave flows $\varphi^\pm$ as
                \begin{align}
                    &\homogplusminusforwardtrappedset=\plusminusforwardtrappedset  \cap\Phi^{\pm}(\Trfo),
                    \\
                    &\homogplusminusbackwardtrappedset=\plusminusbackwardtrappedset\cap\Phi^{\pm}(\Trfo).
                \end{align}
            Observe that $\homogplusminusforwardtrappedset,\homogplusminusbackwardtrappedset$ are invariant under the corresponding Hamiltonian flows generated by $p^\pm$. This follows from $b^\pm$ being constant along these flows combined with Proposition \ref{Propn:RescalingPropn}. Note also, there exists $0<c<C$ such that for any $\omega \in \homogplusminusforwardtrappedset \cup \homogplusminusbackwardtrappedset$ we have 
            \begin{equation}
                 c \leq |\omega_{\xi}| \leq C, \qquad \text{ and } \qquad \omega_{\tau}=1.
            \end{equation}
            The first follows immediately from \eqref{eq:xibhomog}. To see the second, note that since $\omega \in \characteristicsetofPplusminus$,
    $\omega_{\tau}= \bpm(\omega_x, \omega_{\xi})$ and for some $(\omega_x, \omega_t, \tau_0, \xi_0) \in \homogplusminusforwardtrappedset \cup \homogplusminusbackwardtrappedset$ we have $(\omega_{\tau}, \omega_{\xi}) = \left( \frac{\tau_0}{\bpm(\omega_x, \xi_0)}, \frac{\xi_0}{\bpm(\omega_x, \xi_0)}\right)$. By the $1$-homogeneity of $\bpm$ we have 
    \begin{equation}
        \omega_{\tau}= \bpm(\omega_x, \omega_{\xi}) = \bpm \left( \omega_x, \frac{\xi_0}{\bpm(\omega_x,\xi_0)} \right) = \frac{\bpm(\omega_x, \xi_0)}{\bpm(\omega_x,\xi_0)}=1.
    \end{equation}

    \subsection{Decomposition of the Characteristic Set}
        In this section, we show how to decompose $\characteristicsetofPplusminus$ using the trapped and non-trapped sets. 
        First, we show that in the asymptotically flat region, null bicharacteristics escape to infinity like straight lines. A consequence of this is that fully trapped trajectories never enter the asymptotically flat region. 

        \begin{lemma}\label{l:escapeFlow}
            Let $R> R_0$. 
            \begin{enumerate}
                \item If for some $\omega \in \characteristicsetofPplus$ and $s'>0$,
                \begin{equation}
                    \left|x_{\pm s'}^{+}(\omega)\right|\geq\max\{R,|x_0(\omega)|\},
                \end{equation}
            then 
                \begin{equation}
                    \left|x_{\pm s}^{+}(\omega)\right|\geq\max\{R,|x_0(\omega)|\},
                \end{equation}
            for all $s\geq s'$ and
                \begin{equation}
                    \lim_{s\to\infty}\left|x_{\pm s}^{+}(\omega)\right|=\infty.
                \end{equation}
            Analogous statements hold with $x_s^-$, resp. $x_s$, replacing $x_s^{+}$, and $\characteristicsetofPminus$, resp. $\characteristicsetofP$, replacing $\characteristicsetofPplus$.

            \item Furthermore, for any $\omega \in \plusminustrappedset$
            \begin{equation}
                |x_s^{\pm}(\omega)| \leq R_0,
            \end{equation}
            for all $s \in \Rb$. An analogous statement holds with $x_s$ replacing $x_s^{\pm}$ and $\trappedset$ replacing $\plusminustrappedset$.
        
            \end{enumerate}
        \end{lemma}
        \begin{proof}
            1) Since $x_s$ is a reparametrization of $x_s^+$ or $x_s^-$, without loss of generality, we work with $x^+$ and $+s'$. We will show that if the magnitude of the position of the bicharacteristic flow is large enough, then the magnitude's first derivative is non-negative and strictly increasing. 
            
            By Lemma \ref{l:magnitudeDeriv} there exists $c>0$ such that for all $\omega \in \characteristicsetofPplus$, so long as $|x_s^+(\omega)|>R_0$ we have
            \begin{equation}
                \frac{\partial^2}{\partial s^2}\left|x^+_s(\omega)\right|^2\geq c.
            \end{equation}
            We now claim that there exists $s''\in(0,s')$ such that
            \begin{equation}
                \left|x^+_{s''}(\omega)\right|^2>R_0^2
                \quad\text{and}\quad
                \left(\frac{\partial}{\partial s}\left|x^+_s(\omega)\right|^2\right)\Bigg|_{s=s''}\geq0.
            \end{equation}
            We prove this in two cases. First, suppose that for all $s\in[0,s')$, $\left|x^+_{s}(\omega)\right|^2>R_0^2$. Then by the Mean Value Theorem, there exists $s''\in(0,s')$ such that
            \begin{equation}
                \left(\frac{\partial}{\partial s}\left|x^+_s(\omega)\right|^2\right)\Bigg|_{s=s''}=\frac{\left|x^+_{s'}(\omega)\right|^2-\left|x^+_{0}(\omega)\right|^2}{s'}\geq0.
            \end{equation}
            Where the final inequality holds by our assumption on $x_{s'}^{\pm}(\omega)$. 

            Second, suppose there exists $s^*\in[0,s')$ such that $\left|x^+_{s^*}(\omega)\right|^2\leq R_0^2$. Define $\alpha=\sup\{s\in[0,s'):\left|x^+_{s}(\omega)\right|^2\leq R_0^2\}$. 
            Since $|x_{s'}^+(\omega)|^2 \geq R^2 >R_0^2$, and by the continuity of the flow, we have $\alpha<s'$ and $|x_{\alpha}^+(\omega)|^2=R_0^2$. By the Mean Value Theorem, there exists $s''\in(\alpha,s')$ such that
                \begin{equation}
                    \left|x^+_{s''}(\omega)\right|^2>R_0^2\quad\text{and}\quad\left(\frac{\partial}{\partial s}\left|x^+_s(\omega)\right|^2\right)\Bigg|_{s=s''}=\frac{\left|x^+_{s'}(\omega)\right|^2-\left|x^+_{\alpha}(\omega)\right|^2}{s'-\alpha}>0.
                \end{equation}
            Where the final inequality holds again by our assumption on $x_{s'}^+(\omega).$
            
            In either case, we have found $s''\in(0,s')$ such that
                \begin{align}
                     \left|x^+_{s''}(\omega)\right|^2 > R_0^2, \qquad  \left(\frac{\partial}{\partial s}\left|x^+_s(\omega)\right|^2\right)\Bigg|_{s=s''} \geq 0 ,
                     \qquad\text{and} \quad \left(\frac{\partial^2}{\partial s^2}\left|x^+_s(\omega)\right|^2\right)\Bigg|_{s=s''}>0.
                \end{align}
            Therefore 
            \begin{align}
                |x_s^{+}(\omega)|^2 >R_0^2, \quad \frac{\p}{\p s} |x_s^+(\omega)|^2 >0,\quad\text{ and } \quad \frac{\p^2}{\p s^2} |x_s^+(\omega)|^2 >0 \quad \text{ for all } s > s''.
            \end{align}
            This gives the desired conclusion.

        2) Consider $\omega_0 \in \characteristicsetofPplusminus$, such that for some $s' \in \Rb$, $|x_{s'}^{\pm}(\omega_0)| >R_0$. By continuity of the flow in $s$, there exists $R_1>R_0$ and $\e>0$ such that $|x_s^{\pm}(\omega_0)| \geq R_1$ for all $s \in (s'-\e, s'+\e)$. Now consider $s_1, s_2 \in (s'-\e, s'+\e)$ such that $|x_{s_2}^{\pm}(\omega_0)| \geq |x_{s_1}^{\pm}(\omega_0)|$. Let $\omega_1=\varphi_{s_1}^{\pm}(\omega_0)$, so $x_{s_2}^{\pm}(\omega_0)=x_{s_2-s_1}^{\pm}(\omega_1)$ and $|x_{s_2-s_1}^{\pm}(\omega_1)| \geq \max\{R_1, x_0^{\pm}(\omega_1)\}$. Thus by part 1)
        \begin{equation}
            \lim_{s \ra \infty} |x_{sgn(s_2-s_1)s+s_1}^{\pm}(\omega_0)| = \lim_{s \ra \infty} |x_{sgn(s_2-s_1)s}^{\pm}(\omega_1)| = \infty.
        \end{equation}
        Then $\omega_0 \not \in \plusminustrappedset$. Since every null-bicharacteristic of $P$ is a reparametrization of a null-bicharacteristic of $P^+$ or $P^-$, the same conclusion applies to $\omega_0 \in \trappedset$. 
        \end{proof}

        This allows us to partition $\characteristicsetofPplusminus$ and prove some additional basic facts which will be useful in the construction of the non-trapping escape function, Lemma \ref{l:nontrapescape}.
        \begin{proposition}\label{prop:partition}
            The following hold.
            \begin{enumerate}
                \item We can partition $\characteristicsetofPplusminus$ as
                        \begin{equation}
                            \characteristicsetofPplusminus = \plusminusforwardtrappedset \sqcup \plusminusforwardnontrappedset = \plusminusbackwardtrappedset \sqcup \plusminusbackwardnontrappedset = \plusminusforwardtrappedset \cup \plusminusbackwardtrappedset \cup \plusminusnontrappedset.
                        \end{equation}

                \item $\plusminusforwardnontrappedset,\plusminusbackwardnontrappedset,\plusminusnontrappedset$ are open in $\characteristicsetofPplusminus$ and $\plusminusforwardtrappedset,\plusminusbackwardtrappedset,\plusminustrappedset$ are closed.
                \item If $K\subset\plusminusnontrappedset$ is compact, then for every $R> R_0$, there exists $T'\geq0$ such that for every $|s|\geq T'$ and $v\in K$,
                    \begin{equation}
                        |x^\pm_s(v)|\geq R.
                    \end{equation}
                Furthermore, for any $W$, a closed subset of $\mathbb R$, the set
                    \begin{equation}
                        \bigcup_{s\in W}\varphi_s^\pm(K),
                    \end{equation}
                is closed in $\Trf\setminus0$.
            \end{enumerate}
        \end{proposition}
        \begin{proof}
            1)  Let $\omega=(t,x,\tau,\xi)\in \characteristicsetofPplusminus$. For any $R>R_0$, either
                \begin{align}
                    &|x_{s'}^+(\omega)|\geq\max\{R,|x|\} \qquad \text{ for some }s'>0 \text{ or } \\
                    &|x_{s}^+(\omega)|<\max\{R,|x|\} \qquad \text{ for all } s>0.
                \end{align}
                    In the former case, Lemma \ref{l:escapeFlow} implies that $\omega\in\plusforwardnontrappedset$. In the latter case, it follows that $\omega\in\plusforwardtrappedset$. By definition, we have $\plusforwardnontrappedset\cap\plusforwardtrappedset=\emptyset$. The remaining three cases are proven analogously. As a consequence,
                        \begin{align*}
                            \plusminusforwardtrappedset\cup\plusminusbackwardtrappedset\cup\plusminusnontrappedset&=\plusminusforwardtrappedset\cup\plusminusbackwardtrappedset\cup(\plusminusforwardnontrappedset\cap\plusminusbackwardnontrappedset)
                            \\
                            &=\plusminusforwardtrappedset\cup\plusminusbackwardtrappedset\cup\left((\plusminusforwardtrappedset)^c\cap(\plusminusbackwardtrappedset)^c\right)
                            \\
                            &=\plusminusforwardtrappedset\cup\plusminusbackwardtrappedset\cup\left(\plusminusforwardtrappedset\cup\plusminusbackwardtrappedset\right)^c
                            \\
                            &=\characteristicsetofPplusminus.
                        \end{align*}
                2) Let $\omega\in\plusminusforwardnontrappedset$. For any $R>R_0$, there exists $s'\geq0$ so that 
                    \begin{equation}
                        |x_{s'}^\pm(\omega)|\geq\max\{2R,2|x_0^{\pm}(\omega)|\}.
                    \end{equation}
                By continuity of the flow, there exists $\delta>0$ such that for all $\zeta\in \characteristicsetofPplusminus$ with $|\omega-\zeta|<\delta$, we have
                    \begin{equation}
                        |x_{s'}^{\pm}(\omega)-x_{s'}^{\pm}(\zeta)|<\min\left\{\frac{|x_0^{\pm}(\omega)|}{4},\frac{R}{2}\right\},
                    \end{equation}
                for all $s\in[0,s']$. Then, we have that
                    \begin{equation}
                        |x_0^{\pm}(\zeta)|<\frac{5}{4}|x_0^{\pm}(\omega)|\text{ and }|x_{s'}^{\pm}(\zeta)|\geq\max\left\{\frac{3R}{2},\frac{7}{4}|x_0^{\pm}(\omega)|\right\} > \max\{R, |x_0^{\pm}(\zeta)|\}.
                    \end{equation}
                By Lemma \ref{l:escapeFlow}, $\lim_{s\to\infty}|x_s^\pm(\zeta)|=\infty$. Therefore $\zeta\in\plusminusforwardnontrappedset$, and  $\plusminusforwardnontrappedset$ is open. A similar argument shows that $\plusminusbackwardnontrappedset$ is open. We have that $\plusminusnontrappedset$ is open because it is the intersection of two open sets. Finally, the trapped sets are closed as the complements of open sets.
                
                3) Consider $K\subset\plusminusnontrappedset$ a compact set and let $R> R_0$. Define $f:K\to[0,\infty)$ as follows: 
                \begin{equation}
                    f(\omega) = \inf\{ T \geq 0: \forall |s|\geq T, |x_s^\pm(\omega)|>R\}.
                \end{equation}
                The existence of such a $T$ is guaranteed for $\omega\in K\subset\plusminusnontrappedset$ by definition and Lemma \ref{l:escapeFlow}. This function is continuous in $\omega$ by continuity of the flow. Thus, by compactness of $K$, $f$ has a finite maximum. Set $T'\geq0$ to be that maximum value.

                Now consider $\zeta\in \Trf\setminus0$ with a sequence $\zeta_i\in\bigcup_{s\in W}\varphi^\pm_s(K)$ such that $\zeta_i\to\zeta$. 
                Thus, there exist sequences $s_i\in W$ and $\omega_i\in K$ such that $\zeta_i=\varphi^\pm_{s_i}(\omega_i)$. 
                By the compactness of $K$, up to a replacement by a subsequence, there exists $\omega \in K$ such that $\omega_i \ra \omega$. 
                We claim that, up to replacement by a subsequence, there exists $s\in W$ such that $s_i \to s$. 
                To see this, pick $R>2|x_0(\zeta)|$. 
                Then there exists $T'\geq0$ such that $|x_s(K)|>R>2|x_0(\zeta)|$ for all $|s|\geq T'$. 
                Since $\varphi^{\pm}_{s_i}(\omega_i) \ra \zeta$, this means $s_i \in[-T',T']$, which is compact and proves our claim. 
                
                We next claim that $\varphi_s(\omega)=\zeta$. To see this, observe that
                    \begin{align*}
                        |\varphi_s(\omega)-\zeta|\leq|\varphi_s(\omega)-\varphi_s(\omega_{i})|+|\varphi_s(\omega_{i})-\varphi_{s_{i}}(\omega_{i})|+|\varphi_{s_{i}}(\omega_{i})-\zeta|.
                    \end{align*}
                The first term can be made arbitrarily small by the continuity of $\varphi$ in its initial data. The second term can be made arbitrarily small by the fundamental theorem of calculus and taking $s_{i}$ arbitrarily close to $s$. The third can be made arbitrarily small by convergence of $\zeta_{i}$ to $\zeta$. Thus, $\zeta\in\bigcup_{s\in W}\varphi^\pm_s(K)$, and so the set is closed. 
        \end{proof}

    \subsection{Consequences of TGCC}
        We conclude this section with two conditions equivalent to the time-dependent geometric control condition and a lemma connecting it with the geometric control condition used in \cite[Definition 2.2]{Kofroth23}. Recall that Definition \ref{Defn:TGCC} is a statement concerning the Hamiltonian flow generated by $p$. The first equivalence tells us that Definition \ref{Defn:TGCC} is equivalent to an analogous statement for the flows generated by $p^\pm$. The second equivalence tells us that Definition \ref{Defn:TGCC} is equivalent to analogous statements for the forward and backward trapped sets of $p^{\pm}$.
        The propositions in this section are similar to \cite[Proposition 2.12]{Kofroth23}, but the proofs are more involved due to the time-dependent nature of our damping and our requirement that the constants in Definition \ref{Defn:TGCC} are uniform in $t$.

        First we show that the time-dependent geometric control condition is equivalent to an analogous statement for the Hamiltonian flow of $p^{\pm}$.

         \begin{proposition}\label{prop:TGCCplusminus}
            Definition \ref{Defn:TGCC} holds if and only if there exist $\Cm_1,T_1>0$ such that for every $\omega\in \mathring{\Omega}^{\pm}_{tr}$ and $T\geq T_1$,
            \begin{equation}
                \frac{1}{2T}\int^T_{-T}a\left(t_s^\pm(\omega), x_s^\pm(\omega)\right)ds\geq \Cm_1.
            \end{equation}
        \end{proposition}
            \begin{proof}
            1) We first establish uniform bounds for $g^{00} p^{\pm}$ on $\homogplusminustrappedset$.
            
            Let $\omega=(t_0,x_0,\tau_0,\xi_0)\in \mathring{\Omega}^{\pm}_{tr}$. In particular, we have that $\omega\in\characteristicsetofPplusminus\cap\Phi^{\pm}(\Trfo)$. Thus, $\tau_0=b^\pm(x_0,\xi_0)$ and there exist $\tau_1\in\mathbb R$ and $\xi_1\in\mathbb R^3$ such that 
                \begin{equation}
                        w=(t_0,x_0,\tau_0,\xi_0)=\left(t_0,x_0,\frac{\tau_1}{b^\pm(x_0,\xi_1)},\frac{\xi_1}{b^\pm(x_0,\xi_1)}\right).
                \end{equation}
            By the 1-homogeneity of $\bpmx$ in $\xi$
                \begin{equation}
                    \tau_0=b^\pm(x_0,\xi_0)=b^\pm\left(x_0,\frac{\xi_1}{b^\pm(x_0,\xi_1)}\right)=\frac{b^\pm\left(x_0,\xi_1\right)}{b^\pm(x_0,\xi_1)}=1.
                \end{equation}
            Furthermore, since $\tau_s(\omega)=\tau_0(\omega)$ and $\bpmxs=\bpm(x^{\pm}_0, \xi^{\pm}_0)$ for all $s$, we have 
                    \begin{align*}
                        p^\mp(\varphi_s(\omega))&=\tau_s(\omega)-b^\mp(\varphi_s(\omega))
                        \\
                        &=b^\pm(x_s(\omega),\xi_s(\omega))-b^\mp(x_s(\omega),\xi_s(\omega)).
                    \end{align*}
            Now note that, since $\omega\in\mathring{\Omega}^{\pm}_{tr}$, $x_s(\omega)$ and $\xi_s(\omega)$ remain in compact sets for all $s\in\mathbb R$. In particular, by Lemma \ref{l:escapeFlow} and \eqref{eq:xibhomog}, the resulting compact subset of $T^*\mathbb R^3$ is independent of $\omega$. Moreover, the function 
                \begin{equation}
                    (x,\xi)\mapsto b^\pm(x,\xi)-b^\mp(x,\xi),
                \end{equation}
            is continuous and thus attains its minimum and maximum on that compact set. In fact, it is signed ($b^+-b^->0$ and $b^--b^+<0$). Additionally, recall that $-C \leq g^{00} \leq -\frac{1}{C}$ since $g$ is asymptotically flat, $\p_t$ is uniformly time-like, and constant time-slices are uniformly space-like. Thus, there exist $c^* ,C^*>0$ such that for all $s\in\mathbb R$, and $\omega \in \homogplusminustrappedset$
                \begin{align}\label{eq:pminusbound}
                     -C^*&\leq g^{00}(\varphi_s^{\pm}(\omega)) p^-(\varphi_s^{\pm}(\omega))\leq  -c^*\\
                    c^*&\leq g^{00}(\varphi_s^{\pm}(\omega))  p^+(\varphi_s^{\pm}(\omega))\leq C^*.
                \end{align}

            2) We now prove that Definition \ref{Defn:TGCC} implies the integral bound. To do so we first work with $\omega \in \homogplustrappedset$. Recall the reparameterization from Lemma \ref{l:CorrespondenceOfFlows} of a null bicharacteristic of $p^{+}$ into a null bicharacteristic of $p$ given by 
                \begin{equation}
                    \frac{d}{dr'}s'(r')=g^{00}(\varphi_{s'(r')}^+(\omega)) p^-\left(\varphi_{s'(r')}^+(\omega)\right),\quad s'(0)=0.
                \end{equation}
            Note that since $g^{00} p^- < 0$, this is an orientation-reversing reparameterization. In particular, $s'$ is a monotonically decreasing function of $r'$ and its inverse, by an abuse of notation denoted by $r'$, is therefore also monotonically decreasing. Denote by $\overline C_0$ and $T_0$ the constants from Definition \ref{Defn:TGCC}. Pick $T_1 = C^* T_0$. Then since $s'(0)=0$, and 
                \begin{equation} \label{eq:sprimesize}
                    -C^* \leq \frac{d}{dr'} s'(r') \leq -c^*,
                \end{equation}
            integrating from $r'=0$ to $r'= \pm T_0$ we obtain 
                \begin{align}
                    -T_1 = -C^* T_0 \leq &s'(T_0) \leq -c^* T_0\\
                    c^*T_0 \leq &s'(-T_0) \leq C^* T_0 = T_1.
                \end{align}
            Now since $r'$ is monotone decreasing, we have 
                \begin{align}
                    &r'(-T_1) \geq r'(s'(T_0)) = T_0 \label{eq:rprimebound} \\
                   & r'(T_1) \leq r'(s'(-T_0))=-T_0.
                \end{align}
            Now set $ C_1=\frac{c^*}{C^*}\Cm_0$, and consider any $\omega \in \homogplustrappedset$. By substituting $s=s'(r')$, we have
                \begin{align}
                    \frac{1}{2T_1}&\int^{T_1}_{-T_1}a\left(t_{s}^+(\omega), x_{s}^+(\omega)\right)ds
                    \\
                    &=\frac{1}{2T_1}\int^{r'(T_1)}_{r'(-T_1)}a\left(t_{s'(r')}^+(\omega),  x_{s'(r')}^+(\omega)\right)\left(\frac{d}{dr'} s'(r') \right)dr'
                    \\
                    &=\frac{1}{2T_1}\int^{r'(-T_1)}_{r'(T_1)}a\left(t_{s'(r')}^+(\omega),x_{s'(r')}^+(\omega),\right) \left(-\frac{d}{dr'} s'(r') \right)dr'.
                \end{align}
            Now using that $\varphi^+_{s'(r')} = \varphi_{r'}$ and \eqref{eq:sprimesize} we have 
                \begin{align}
                    &\frac{1}{2T_1}\int^{r'(-T_1)}_{r'(T_1)}a\left(t_{s'(r')}^+(\omega), x_{s'(r')}^+(\omega)\right)\left( - \frac{d}{dr'} s'(r') \right)dr'\\
                    &\geq\frac{c^*}{2T_1}\int^{r'(-T_1)}_{r'(T_1)}a\left(t_{r'}(\omega), x_{r'}(\omega)\right)dr'.
                \end{align}
            Then by \eqref{eq:rprimebound} and since $a\geq 0$,
                \begin{align}
                    \frac{c^*}{2T_1}\int^{r'(-T_1)}_{r'(T_1)}a\left(t_{r'}(\omega), x_{r'}(\omega)\right)dr'
                    &\geq\frac{c^*}{2T_1}\int^{T_0}_{-T_0}a\left(t_{r'}(\omega), x_{r'}(\omega)\right)dr'.
                \end{align}
            Now, note that $\omega \in \homogplustrappedset \subset \plustrappedset \subset \trappedset$ and $\omega_{\tau}=1$. Then combining the preceding chain of inequalities and applying  Definition \ref{Defn:TGCC} and the definition of $T_1$ and $C_1$
                \begin{align}
                    \frac{1}{2T_1}\int^{T_1}_{-T_1}a\left(t_{s}^+(\omega), x_{s}^+(\omega)\right)ds &\geq \frac{c^*}{2T_1}\int^{T_0}_{-T_0}a\left(t_{r'}(\omega), x_{r'}(\omega)\right)dr'\\
                    &\geq c^* \frac{T_0}{T_1} \Cm_0 = c^* \frac{1}{C^*} \Cm_0
                    \\
                    &=C_1.
                \end{align}
            We now show that for any $T'\geq T_1$, we obtain the claim for a potentially smaller lower bound. Denote by $K\in\mathbb N$ the largest integer such that $KT_1\leq T'$. Then letting $\omega_j=\varphi_{jT_1}(\omega)$
                \begin{equation}
                    \frac{1}{T'}\int_0^{T'}a\left(t_s^+(\omega),x_s^+(\omega)\right)ds'\geq\frac{1}{T'}T_1\sum_{j=0}^{K-1}\frac{1}{T_1}\int_0^{T_1}a\left(t_s^+(\omega_j), x_s^+(\omega_j)\right)ds'.
                \end{equation}
             By the preceding argument, the right-hand side is bounded from below by $\frac{1}{T'}T_1KC_1$. Since $K$ is maximal, we also have $T'\leq(K+1)T_1$. Thus, the right-hand side is bounded from below by
                \begin{equation}
                    \frac{K}{K+1}C_1\geq\frac{1}{2}C_1=:\overline{C}_1.
                \end{equation}
                An analogous proof shows the same conclusion for $\omega \in \homogminustrappedset$.

            3) Now we show that the integral bound implies Definition \ref{Defn:TGCC}. So consider $\omega \in \trappedset$. Since $\trappedset=\plustrappedset \cup \minustrappedset$, we have $\omega \in \plustrappedset$ or $\minustrappedset$. We will assume  $\omega \in \plustrappedset$. The proof for $\omega \in \minustrappedset$ is analogous.
            
            Since $\omega \in \characteristicsetofPplus$, $\omega_{\tau}=b^+(\omega_x,\omega_{\xi})$, and so  
            \begin{equation}
                \ti{\omega} = \left(\omega_t, \omega_x, 1, \frac{\omega_{\xi}}{\omega_{\tau}}\right) \in \homogplustrappedset.
            \end{equation}
            By Proposition \ref{Propn:RescalingPropn} we have
            \begin{equation}
                (t_s^+(\omega),x_s^+(\omega)) = (t_s^+(\ti{\omega}), x_s^+(\ti \omega)). \label{eq:spatialw}
            \end{equation}
            Now recall the reparametrization from Lemma \ref{l:CorrespondenceOfFlows} of a null bicharacteristic of $p$ into a null bicharacteristic of $p^+$ given by 
            \begin{equation}
                \frac{d}{dr} s(r) = \frac{1}{g^{00}(\varphi_{s(r)}(\omega)) p^-(\varphi_{s(r)} (\omega))}, \quad s(0)=0.
            \end{equation}
            Since $g^{00} p^-<0$ this reparametrization is orientation reversing. In particular, $s$ is a monotonically decreasing function of $r$. Furthermore, the inverse of $s$, by an abuse of notation denoted by $r$, is also monotone decreasing. 

            Now noting that 
            \begin{equation}
                p^-(\varphi_{s}(\omega)) = \omega_{\tau} - b^-(x_s,\xi_s) = b^+(x_s,\xi_s)-b^-(x_s,\xi_s),
            \end{equation}
            and using that $b^+(\omega_x,\omega_{\xi})=\omega_{\tau}$ on $\characteristicsetofPplus$ and $b^{\pm} \simeq |\xi|$ by \eqref{eq:xibhomog}, we have for some $C>0$
            \begin{equation} 
                \frac{1}{C} \omega_{\tau} \leq p^-(\varphi_s(\omega)) \leq C \omega_{\tau}.
            \end{equation}
            Therefore for some $C^*, c^*>0$ we have 
            \begin{equation}\label{eq:sparambounds}
                -\frac{1}{c^* \omega_{\tau}} \leq \frac{d}{dr} s(r) \leq -\frac{1}{C^* \omega_{\tau}}. 
            \end{equation}
            Now set $T_0=\frac{T_1}{c^* \omega_{\tau}}$, then since $s(0)=0$
            integrating from $r=0$ to $r=\pm T_1$ we obtain 
            \begin{align}
                -T_0 = -\frac{T_1}{c^* \omega_{\tau}} \leq &s(T_1) \leq - \frac{T_1}{C^* \omega_{\tau}}\\
                \frac{T_1}{C^*\omega_{\tau}}\leq &s(-T_1)\leq \frac{T_1}{c^* \omega_{\tau}} =T_0.
            \end{align}
            Now since $r$ is monotone decreasing, we have 
            \begin{align}\label{eq:rsparam}
                r(-T_0) \geq r(s(T_1)) = T_1\\
                r(T_0) \leq r(s(-T_1)) = -T_1.
            \end{align}
            Now set $C_0=\frac{c^*}{C^*}\Cm_1$. By substituting $s'=s(r)$ we have 
            \begin{align}
                \frac{1}{2T_0} \int_{-T_0}^{T_0} a(t_{s'}(\omega), x_{s'}(\omega)) ds' = \frac{1}{2T_0} \int_{r(T_0)}^{r(-T_0)} a(t_{s(r)}(\omega), x_{s(r)}(\omega)) \left( -\frac{d}{dr}s(r)\right) dr.
            \end{align}
            Now since $\varphi^+_r=\varphi_{s(r)}$ and by \eqref{eq:sparambounds}, \eqref{eq:rsparam}, and \eqref{eq:spatialw} we have 
            \begin{align}
                \frac{1}{2T_0} \int_{r(T_0)}^{r(-T_0)} a(t_{s(r)}(\omega), x_{s(r)}(\omega)) \left( -\frac{d}{dr}s(r) \right) dr 
                &\geq  \frac{1}{2T_0 C^* \omega_{\tau}} \int_{r(T_0)}^{r(-T_0)} a(t_r^+(\omega), x_r^+(\omega)) dr \\
                &\geq \frac{1}{2T_0 C^* \omega_{\tau}} \int_{-T_1}^{T_1} a(t_r^+(\omega), x_r^+(\omega)) dr\\
                &= \frac{1}{2T_0 C^* \omega_{\tau}} \int_{-T_1}^{T_1} a(t_r^+(\ti w), x_r^+(\ti w)) dr. 
            \end{align}
            Now combining the preceding chain of inequalities and noting that since $\ti w \in \homogplustrappedset$ we can apply the assumed integral inequality to obtain 
            \begin{align}
                \frac{1}{2T_0} \int_{-T_0}^{T_0} a(t_{s'}(\omega), x_{s'}(\omega)) ds &\geq \frac{1}{2T_0 C^* \omega_{\tau}} \int_{-T_1}^{T_1} a(t_r^+(\ti w), x_r^+(\ti w)) dr \\
                &\geq \frac{T_1}{T_0 C^* \omega_{\tau}} \Cm_1 = \frac{c^*}{C^*}\Cm_1 = C_0.
            \end{align}
            As in step 2 we can extend this to any $T' \geq T_0= \frac{T_1}{c^* \omega_{\tau}}$ by replacing $C_0$ by $\Cm_0=\frac{1}{2} C_0$.
            \end{proof}

        Before proving our second equivalence, we record some lemmas which will be used in the proof. For $R \geq R_0$, we define the spatially compact semi-trapped sets
            \begin{equation}
                \pmhomogftrappedcompactset=\homogplusminusforwardtrappedset\cap\{|x|\leq R\},
            \end{equation}
        and
            \begin{equation}
                \pmhomogbtrappedcompactset=\homogplusminusbackwardtrappedset\cap\{|x|\leq R\}.
            \end{equation}
        First, a forward, respectively backward, trapped trajectory with initial position $\leq R$ must remain in a compact spatial region forward, respectively backward, in $s$. This result is analogous to the second part of Lemma \ref{l:escapeFlow}.
        \begin{lemma}\label{l:starlemma}
            Let $R> R_0$. If $\omega\in\pmhomogftrappedcompactset$, resp. $\pmhomogbtrappedcompactset$, then $|x_s^{\pm}(\omega)|\leq R$ for all $s\geq0$, resp. $s\leq0$.
        \end{lemma}
        \begin{proof}
            Suppose not, so there exists $s'\geq 0$ such that $|x_{s'}^\pm(\omega)|>R$. 
            Since $\omega\in\pmhomogftrappedcompactset$, we have $R \geq x_0^\pm(\omega)$, and so $|x_s^{\pm}(\omega)| >R = \max\{R, |x_0^{\pm}(\omega)|\}$. Thus by Lemma \ref{l:escapeFlow} we have that $|x_s^\pm(\omega)|\to\infty$ as $s\to\infty$, which contradicts $\omega\in\pmhomogftrappedcompactset$.
        \end{proof}
        Next forward, respectively  backward, trapped trajectories with initial position $\leq R$ become arbitrarily close to trapped trajectories, where the size of $s$ required to achieve this is uniform over $\pmhomogftrappedcompactset$, respectively $\pmhomogbtrappedcompactset$.
        \begin{lemma}\label{l:halftrappedlemma}
            Let $R> R_0$. For all $\varepsilon>0$, there exists $T>0$, resp. $T<0$, such that for any $\omega\in \pmhomogftrappedcompactset$, resp. $\pmhomogbtrappedcompactset$, and for all $s\geq0$, resp. $s\leq0$, we have
                \begin{equation}
                    d_g\Big(\big(x^\pm_{s+T}(\omega),\xi^\pm_{s+T}(\omega)\big),\Pi_{x,\xi}\homogplusminustrappedset \Big)<\varepsilon. 
                \end{equation}
        \end{lemma}
        \begin{proof}
        We prove the $\pmhomogftrappedcompactset$ case and the $\pmhomogbtrappedcompactset$ case is analogous. 
            Suppose the desired conclusion does not hold. Then there exists $\varepsilon_0>0$ and sequences $\omega_j\in\pmhomogftrappedcompactset$, and $s_j\geq0$, such that 
                \begin{equation}
                    d_g\Big(\big(x^\pm_{s_j+j}(\omega_j),\xi^\pm_{s_j+j}(\omega_j)\big),\Pi_{x,\xi}\homogplusminustrappedset \Big)\geq\varepsilon_0. \label{EQN:lowerboundondistance}
                \end{equation}
            From the invariance of $\pmhomogftrappedcompactset$ under the flow and Lemma \ref{l:starlemma}, we have that for all $j\in\mathbb N$,
                \begin{equation}
                    |\xi_{s_j+j}(\omega_j)|\approx1 \text{ and }|x_{s_j+j}(\omega_j)|\leq R.
                \end{equation}
            Therefore, the sequence $\left(x^\pm_{s_j+j}(\omega_j),\xi^\pm_{s_j+j}(\omega_j)\right)_{j\in\mathbb N}\in T^*\mathbb R^3$ is bounded. Thus, after potentially passing to a subsequence, there exists $\big(\overline x,\overline\xi\big)\in T^*\mathbb R^3$ such that
                \begin{equation}\label{eq:xsjconverge}
                    \lim_{j\to\infty}\left(x^\pm_{s_j+j}(\omega_j),\xi^\pm_{s_j+j}(\omega_j)\right)=\big(\overline x,\overline\xi\big).
                \end{equation}
            We first claim that $\big(\overline x,\overline\xi\big)\in\Pi_{x,\xi}\plusminustrappedset$. Denote by $\tilde\varphi$ the null bicharacteristic flow on $T^*\mathbb R^3$. In particular, $\tilde\varphi_{s_j+j}(\omega_j)=\left(x^\pm_{s_j+j}(\omega_j),\xi^\pm_{s_j+j}(\omega_j)\right)$. From the group law, we know that
                \begin{equation}
                    \tilde\varphi_{s+s_j+j}(\omega_j)=\tilde\varphi_s\left(\tilde\varphi_{s_j+j}(\omega_j)\right).
                \end{equation}
            Now fix an $s\in\mathbb R$. By continuity of $\tilde\varphi_s(\cdot)$ and \eqref{eq:xsjconverge} for any $\e>0$ there exists $J \in \Nb$, $J >-s$, such that for all $j \geq J$

                \begin{equation}
                \big|\tilde\varphi_{s+s_j+j}(\omega_j)-\tilde\varphi_s\big(\overline x,\overline\xi\big)\big|<\varepsilon.
                \end{equation}
            That is, for each $s\in\mathbb R$, $\lim_{j\to\infty}\tilde\varphi_{s+s_j+j}(\omega_j)=\tilde\varphi_s\big(\overline x,\overline\xi\big)$. Since $s+s_j+j \geq 0$ and $\omega_j \in \pmhomogftrappedcompactset$, by Lemma \ref{l:starlemma} we have $|x_{s+s_j+j}(\omega_j)|\leq R$. Projecting the previous limit onto its $x$-coordinate yields $\big|x_s\big(\overline x,\overline\xi\big)\big|\leq R$ for all $s\in\mathbb R$. By definition, $\big(\overline x,\overline\xi\big)\in\Pi_{x,\xi} \plusminustrappedset$.

            Now we will show $(\overline x, \overline \xi) \in \Pi_{x,\xi} \homogplusminustrappedset$, which will contradict \eqref{EQN:lowerboundondistance}. Let 
            \begin{equation}
                \tau_j= \bpm(x^{\pm}_{s_j+j}(\omega_j), \xi_{s_j+j}^{\pm}(\omega)).
            \end{equation}
            Then by continuity of $\bpm$ if we define $\tau=\lim_{j \ra \infty} \tau_j$, then $\tau=\bpm(\overline x, \overline \xi)$. Now since $\omega_j \in \pmhomogftrappedcompactset$, $\tau_j=1$ for all $j$. Thus $\tau=1$ and so for any $t$ we have $(t, \overline{x}, 1, \overline{\xi}) \in \homogplusminustrappedset$. Thus $(\overline{x}, \overline{\xi}) \in \Pi_{x,\xi} \homogplusminustrappedset$.
        \end{proof}
        We have a second equivalent statement to the time-dependent geometric control condition, but involving only forward and backward trapped trajectories with initial position $\leq R$.
        \begin{proposition}\label{prop:oneSidedGCC}
            Definition \ref{Defn:TGCC} holds if and only if there exist $T_2,\Cm_2>0$ such that for every $T'\geq T_2$, the following holds: If $\omega\in \pmhomogftrappedcompactset$, then
                \begin{equation}
                    \frac{1}{T'}\int_0^{T'}a\left(t_s^{\pm}(\omega), x_s^{\pm}(\omega)\right)ds\geq \Cm_2.
                \end{equation}
            Similarly, if $\omega\in\pmhomogbtrappedcompactset$, then
                \begin{equation}
                    \frac{1}{T'}\int_{-T'}^0a\left(t_s^{\pm}(\omega), x_s^{\pm}(\omega)\right)ds\geq \Cm_2.
                \end{equation}
        \end{proposition}
        \begin{proof}
            1) First assume that the integral lower bound holds. 
            Note that $\homogplusminustrappedset = \homogplusminusforwardtrappedset \cap \homogplusminusbackwardtrappedset$ and by Lemma \ref{l:escapeFlow} $\homogplusminustrappedset \subset \{|x| \leq R\}$. 
            Therefore $\homogplusminustrappedset \subset \pmhomogftrappedcompactset \cap \pmhomogbtrappedcompactset$. 
            Thus for any $T' \geq T_2$ and any $\omega \in \homogplusminustrappedset$ we have 
            \begin{align}
                &\frac{1}{2T'} \int_{-T'}^{T'} a(t_s^{\pm}(\omega), x_s^{\pm}(\omega)) ds \\
                &= \frac{1}{2} \left(\frac{1}{T'} \int_0^{T'} a(t_s^{\pm}(\omega), x_s^{\pm}(\omega)) ds + \frac{1}{T'} \int_{-T'}^0 a(t_s^{\pm}(\omega), x_s^{\pm}(\omega)) ds \right) \\
                &\geq \Cm_2.
            \end{align}
            Then by Proposition \ref{prop:TGCCplusminus}, Definition \ref{Defn:TGCC} holds. 
        
            2) Now assume Definition \ref{Defn:TGCC} holds. We will prove the case $\omega \in \pmhomogftrappedcompactset$ in detail. The case $\omega \in \pmhomogbtrappedcompactset$ is analogous. 
            
            Since Definition \ref{Defn:TGCC} holds, we can apply Proposition \ref{prop:TGCCplusminus} to obtain $\Cm_1, T_1 >0$, such that for every $\omega \in \homogplusminustrappedset$ and $T \geq T_1$
                \begin{equation}
                    \frac{1}{2T} \int_{-T}^T a(t_s^{\pm}(\omega), x_s^{\pm}(\omega)) ds \geq \Cm_1.
                \end{equation}
            By uniform continuity of $a$, there exists $\delta>0$ such that if $|x-y|<\delta$, then
                \begin{equation}
                    \left|a(t,x)-a(t,y)\right|\leq\frac{\Cm_1}{2}.
                \end{equation}
                Next, by the compactness of $\Pi_{x,\xi} \homogplusminusforwardtrappedset \times[0,2T_1]_s$, the flow $\tilde\varphi_s^{\pm}=\Pi_{x,\xi} \varphi_s^{\pm}$ is uniformly continuous there in $(x,\xi)$ and $s$. Thus, there exists $\varepsilon>0$ such that for all $s_0, s_1\in[0,2T_1]$, if 
            \begin{equation}
                |s_0-s_1|+|x_0-x_1|+|\xi_0-\xi_1|<\varepsilon,
            \end{equation}
            then 
            \begin{equation}
                |\tilde\varphi_{s_0}^{\pm}(x_0,\xi_0)-\tilde\varphi_{s_1}^{\pm}(x_1,\xi_1)|<\delta.
            \end{equation} 
            From Lemma \ref{l:halftrappedlemma}, there exists $\tilde T>0$ so that for each $\omega\in\homogplusminusforwardtrappedset$, there exists $(\overline x,\overline\xi)\in\Pi_{x,\xi}\homogplusminustrappedset$ such that $|x_{\tilde T}^+(\omega)-\overline x|+|\xi_{\tilde T}^+(\omega)-\overline\xi|<\varepsilon$. Thus, for all $s\in[0,2T_1]$, we have
                \begin{equation}
                    |\tilde\varphi_{s}^{\pm}(\ti \varphi_{\ti T}(\omega))-\tilde\varphi_{s}^{\pm}(\overline x,\overline\xi)|<\delta.
                \end{equation}
            It follows that for all $t\in\mathbb R$,
                \begin{equation}
                    \left|a(t,x_{s}^{\pm}(\overline x,\overline\xi))-a(t,x_{\tilde T+s}^{\pm}(\omega))\right|\leq\frac{\Cm_1}{2}. \label{eq:gccequiv2Intermed}
                \end{equation}
            Set $\ti{\omega}=\varphi^{\pm}_{T_1}\left(t_{\tilde T}(\omega), \overline x, \tau_0(\omega)=1, \overline\xi\right)$. Since $(\overline x,\overline\xi)\in\Pi_{x,\xi}\homogplusminustrappedset$, we have that $\ti \omega \in\homogplusminustrappedset$. By Proposition \ref{prop:TGCCplusminus}, and then letting $s=s'+T_1$, we have 
                \begin{equation}
                    \Cm_1 \leq\frac{1}{2T_1}\int_{-T_1}^{T_1}a\left(t_s^{\pm}(\ti \omega), x_s^{\pm}(\ti \omega )\right)ds=\frac{1}{2T_1}\int_{0}^{2T_1}a\left(t_{\tilde T+s'}^{\pm}(\omega),x_{s'}^{\pm}(\overline x,\overline\xi)\right)ds'.
                \end{equation}
            By \eqref{eq:gccequiv2Intermed} we have 
                \begin{equation}
                    \frac{1}{2T_1}\int_{0}^{2T_1}\left|a\left(t_{\tilde T+s'}^{\pm}(\omega), x_{s'}^{\pm}(\overline x,\overline\xi)\right)-a\left(t_{\tilde T+s'}^{\pm}(\omega), x_{\tilde T+s'}^{\pm}(\omega)\right)\right|ds'\leq\frac{\Cm_1}{2}.
                \end{equation}
            Therefore,
                \begin{equation}
                    \frac{\Cm_1}{2}\leq\frac{1}{2T_1}\int_{\ti{T}}^{2T_1+\ti{T}} a\left( t_r^{\pm}(\omega), x_r^{\pm}(\omega)\right) dr 
                    \leq \frac{1}{2T_1}\int_0^{2T_1+\tilde T}a\left(t_{r'}^{\pm}(\omega), x_{r'}^+(\omega)\right)dr'.
                \end{equation}
            So,
                \begin{equation}
                    \frac{1}{2T_1+\tilde T}\int_0^{2T_1+\tilde T}a\left(t_s^{\pm}(\omega), x_s^{\pm}(\omega)\right)ds\geq\frac{T_1\Cm_1}{2T_1+\tilde T}=:\Cm_2.
                \end{equation}
            Define $T_2:=\tilde T+2T_1$. 

            For $T\geq T_2$, replace $ \Cm_2$ by half of its original value and apply the same type of argument as in the second part of step 2 of Proposition \ref{prop:TGCCplusminus}. 
        \end{proof}

    We conclude this section by showing that when $a$ does not depend on $t$, our Definition \ref{Defn:TGCC} is equivalent to the geometric control condition in \cite[Definition 2.2]{Kofroth23}.
        \begin{lemma}\label{l:gccequiv}
            Suppose $a(t,x)=a(x)$. Definition \ref{Defn:TGCC} holds if and only if for all $\omega \in \trappedset$, there exists $s \in \Rb$, such that $a(x_s(\omega))>0$.
        \end{lemma}
        \begin{proof}
        1) Assume the condition does not hold. That is, for some $\omega \in \trappedset$ we have $a(x_s(\omega)) =0$ for all $s \in \mathbb{R}$. Then 
        \begin{equation}
            \frac{1}{T} \int_{-T}^T a(x_s(\omega)) ds =0, \qquad \text{ for all } T>0,
        \end{equation}
        so Definition \ref{Defn:TGCC} does not hold. 
        
        2) Using Proposition \ref{prop:oneSidedGCC}, it is enough to show that there exist $T_2, \Cm_2>0$ such that for every $T' \geq T_2$, for all $\omega \in \pmhomogftrappedcompactset$
            \begin{equation}
                \frac{1}{T'} \int_0^{T'} a(x_s^{\pm}(\omega)) ds \geq \Cm_2,
            \end{equation}
            and for all $\omega \in \pmhomogbtrappedcompactset$
            \begin{equation}
                \frac{1}{T'} \int_{-T'}^{0} a(x_s^{\pm}(\omega)) ds \geq \Cm_2.
            \end{equation}
            We will prove the statement for $\omega \in \pmhomogftrappedcompactset$ and the proof for $\omega \in \pmhomogbtrappedcompactset$ is analogous. 

            Consider the map $f:\pmhomogftrappedcompactset \ra \Rb_+$ defined by
            \begin{equation}
                f(\omega)=\inf\{s: a(x_s^{\pm}(\omega))>0\}.
            \end{equation}
            By \cite[Proposition 2.12]{Kofroth23}, the assumption on $\trappedset$ implies that there exists $s \geq 0$, such that $a(x^{\pm}_s(\omega)) >0$, so $f$ is always finite.  
            By the continuity of the flow and $a$, $f$ is continuous in $\omega$. Furthermore, $f$ depends only on the $x$ and $\xi$ components of $\omega$. Since $\pmhomogftrappedcompactset$ is compact in $x$ and $\xi$, $f$ attains a maximum over $\pmhomogftrappedcompactset$. Call the maximum $T_0$. Then $a(x_{T_0}^{\pm}(\omega))>0$ for all $\omega \in \pmhomogftrappedcompactset$. Now define 
            \begin{equation}
                g(\omega)=a(x_{T_0}^{\pm}(\omega)),
            \end{equation}
            and note it is continuous in $\omega$ and depends only on $x$ and $\xi$. By the compactness of $\pmhomogftrappedcompactset$ in $x$ and $\xi$, $g$ attains a minimum $m_0>0$. That is $a(x_{T_0}^{\pm}(\omega)) \geq m_0>0$ for all $\omega \in \pmhomogftrappedcompactset$. Now by uniform continuity of $a$, there exists $\e>0$, such that $a(x_s^{\pm}(\omega)) >m_0/2$ for all $\omega \in \pmhomogftrappedcompactset$ and $s \in (T_0-\e, T_0)$. Therefore for all $\omega \in \pmhomogftrappedcompactset$ 
            \begin{equation}
                \frac{1}{T_0} \int_{0}^{T_0} a(x_s^{\pm}(\omega)) ds \geq \frac{m}{4T_0} \e.
            \end{equation}
            Using the same argument as in the second part of step 2 of Proposition \ref{prop:TGCCplusminus}, we have for all $T \geq T_0$ and all $\omega \in \pmhomogftrappedcompactset$
            \begin{equation}
                \frac{1}{T} \int_{0}^{T} a(x_s^{\pm}(\omega)) ds \geq \frac{m}{8T_0} \e:=\Cm_2.
            \end{equation}
            This is exactly the desired integral lower bound.
        \end{proof}

\newcommand{\escapeFunctionStatement}{Fix $0<\d\ll1$ from Definition \ref{d:asymflat2}. There exists $\kappa\geq 1, C>0$ and symbols $\ti{q}_j \in S^j(\Trf)$ and $m \in S^0(\Trf)$, all supported in $|\xi| \geq 1$, $|\tau| \geq 1$, such that for $q=\tau \ti{q}_0+\ti{q}_1$,
    \begin{equation}
        (H_p q +2 \kappa \tau a q + pm)(t,x,\tau,\xi) \geq C \mathbbm{1}_{|\xi|\geq 1} \mathbbm{1}_{|\tau| \geq 1} \<x\>^{-2-2\d}(\tau^2+|\xi|^2).
    \end{equation}}

\section{Escape Function Construction}\label{s:escapeFunction}
In this section we construct an escape function which we use in Section \ref{s:propagation} to prove the high frequency estimate. Specifically, we show:
\begin{proposition}\label{p:EscapeFunction}
\escapeFunctionStatement
\end{proposition}
Note the difference in the power of $\<x\>$ compared to \cite[Lemma 2.4]{Kofroth23}, \cite[Lemma 4.1]{MST20} and our Lemma \ref{l:combineEscapeFunc}. However, this difference is irrelevant to the final application of this Lemma in the proof of Lemma \ref{l:highprop}. In that proof we eventually restrict to $|x|<2R_0$ and so the power on $\<x\>$ is simply absorbed into a constant depending on $R_0$.

We follow the general approach of \cite[Section 2.4]{Kofroth23} and \cite[Lemma 4.1]{MST20}. In particular, we construct our escape function separately on the trapped and non-trapped regions, then combine them together and construct an elliptic correction term $m$. As in \cite{Kofroth23} we construct our escape function on the trapped set first in neighborhoods of individual $\omega$ and then combine these to obtain an escape function for the entire trapped set. 

However, due to the time dependence of our damping we must work on $\Trf$ rather than $T^* \Rb^3$. Because of this, we cannot use the compactness of $(\pmhomogftrappedcompactset \cup \pmhomogbtrappedcompactset) \cap T^* \Rb^3$. We instead must cover $\pmhomogftrappedcompactset \cup \pmhomogbtrappedcompactset \subset \Trf$, which is not compact in $t$, using these neighborhoods of $\omega$. In order to achieve this, we require some uniformity in $\omega$ of our escape functions. We then use the compactness of $(\pmhomogftrappedcompactset \cup \pmhomogbtrappedcompactset) \cap [0,T_2]$, where $T_2$ is from Propostion \ref{prop:oneSidedGCC}, to obtain a finite subcover and then extend to a locally finite cover of $\pmhomogftrappedcompactset \cup \pmhomogbtrappedcompactset$ via this uniformity and the time-independence of the $g$.

\subsection{Outline}\label{sec:escapeoutline}
Before starting our construction, we will first outline the main steps. The particular steps depend on the particular subset of $\Trf$ where the construction is taking place.
\begin{enumerate}
    \item \textbf{\underline{On $\characteristicsetofP$.}} 
    We consider $\omega$ as the initial data of null bicharacteristics of $P^{\pm}$. There are two regimes to consider: the null bicharacteristics are semi-trapped, or non-trapped. After constructing escape functions in each separate regime, we then combine the results to obtain an escape function on the characteristic set of $P$ in Lemma \ref{l:combineEscapeFunc}. 
    \begin{enumerate}
        \item \textbf{\underline{$\{|x| \leq R\}$, Semi-Trapped Null-Bicharacteristics.}} 
        Here, we work with $\omega \in \{|x| \leq R \} \subset \Trf$ producing trajectories of $P^{\pm}$ that are bounded forward or backward in time. Our construction proceeds by constructing escape functions $\qo$ and corrections $\aco$ for individual $\omega$. In particular, we obtain 
        \begin{equation}
            \Hpm \qo + \aco \geq c,
        \end{equation}
        on neighborhoods of uniform width around $\omega$. We cover the semi-trapped region with these neighborhoods and then use time-invariance of $g$ and compactness to reduce to a locally finite cover in Lemma \ref{l:fulltrappedescapelemma}.
        We perform our construction separately on two subsets of this region.
            \begin{enumerate}
                 \item \textbf{\underline{Where the Damping is Large:}} Recall $\Cm_2$ from Proposition \ref{prop:oneSidedGCC}. In this region, we have $a(\omega) \geq \frac{\Cm_2}{2}$. Because $a$ is bounded from below, $\qo$ can be taken to be identically $0$ and we still obtain 
                 \begin{equation}
                     \Hpm \qo + a >0.
                 \end{equation}
                 Some additional care is required to ensure compatibility with the locally-finite argument. This case is addressed in Lemma \ref{l:escapeLargeDamp}.
                 
                \item \textbf{\underline{Where the Damping is Small:}} In this region, we have $a(\omega) \leq \frac{\Cm_2}{2}$. Here, the damping is not large enough to reinforce local energy decay and so an escape function is needed to quantify how energy flows into the region where the damping is large. We begin by constructing coordinates around each semi-trapped null-bicharacteristic in Lemma \ref{l:productCoord}. In these coordinates, we explicitly construct the escape function in terms of averages of the damping along the trajectory in Section \ref{s:smallDampEscape}. It is here that the time-dependent geometric control condition from Proposition \ref{prop:oneSidedGCC} is used.
            \end{enumerate} 
        \item \textbf{\underline{Non-Trapped Null-Bicharacteristics.}}  For the non-trapped trajectories in $\{|x| \leq R\}$ the energy flows away from compact sets which naturally produces local energy decay. Because our semi-trapped escape function has uniformity in $t$, we are able to separate this step from the damping and its time-dependence. Thus we follow the approach of \cite[Lemma 2.16]{Kofroth23} and \cite{BoucletRoyer2014}. In the region $\{|x| \geq R\}$, the trapping and damping are irrelevant and we use the same multiplier approach from \cite{Kofroth23}, \cite[Lemma 4.1]{MST20} and \cite{MMT08}.
    \end{enumerate} 
    \item \textbf{\underline{On the elliptic set of $P$:}} We finally construct a lower order correction term $m$ to ensure positivity away from the characteristic set. We identify $m$ using the same quadratic equation analysis of \cite[Lemma 4.1]{MST20} and then estimate the behavior of $H_p q + 2\kappa \tau a + pm$ for large $x,\xi$, and $\tau$. This completes the proof of Proposition \ref{p:EscapeFunction}.
\end{enumerate}

\subsection{Semi-Trapped Escape Function Construction}\label{s:semitrapescape}
To set some notation, we will write together $(t,x)=z$ or $(\tpmso, \xpmso)=\zpmso$ and $(\tau, \xi)=\zeta$ or $(\taupmso, \xipmso)=\zetapmso$. We will write $\Piz, \Pizeta$ for projections onto these coordinates.  We will write $a(\gopms):=a(\Piz(\gopms)),$ and $a(\omega):=a(\Piz \omega)$. We define also $\Pi_t^{\bot}(t, x, \tau, \xi)=(x,\tau, \xi)$. 


To begin our semi-trapped escape function construction, we define two subsets of $\Trf$:
\begin{align} 
\pmsemitrappedset&= \pmftrappedcompactset \cup \pmbtrappedcompactset,\\
\homogpmsemitrappedset &= \pmsemitrappedset \cap \Phi^{\pm}(\Trfo) = \pmhomogftrappedcompactset \cup \pmhomogbtrappedcompactset.
\end{align}
Note by Proposition \ref{Propn:RescalingPropn} and Lemma \ref{l:starlemma}, that $\homogpmsemitrappedset$ is compact in $x, \xi,$ and $\tau$.

\subsubsection{Escape Functions Near Large Damping}
We first consider null bicharacteristics with initial data $\omega$ satisfying $a(\omega) \geq \frac{\Cm_2}{2}$. For such $\omega$ we can bound $\Hpm \qo +\aco$ from below by taking $\qo \equiv 0$ and $\aco$ given by a cutoff version of the damping.
 
\begin{lemma}\label{l:escapeLargeDamp}
    There exists a constant $r_1 >0$, such that for any $\omega \in \homogpmsemitrappedset\cap\{a \geq\frac{\overline C_2}{2}\}$, there exist $\aco, \ro \in \Cc(\Phipm(\Trfo))$ such that 
    \begin{enumerate}
        \item $0 \leq \ro \leq \aco \leq a$,
        \item $\ro \geq \frac{\Cm_2}{4}$ on $B(\omega,r_1) \subset \Phi^{\pm}(\Trfo)$,
        \item
        \begin{equation}
            \Pi_t \supp(\aco), \Pi_t \supp(\ro) \subset \{t \in [\omega_t-1,\omega_t+1]\}, \qquad \text{and}
        \end{equation}
        \item for each multi-index $\alpha, \beta$, there exists $C_{\alpha\beta} >0$, such that for all $\omega$ and all $(z,\zeta) \in \Phipm(\Trf)$
        \begin{align}
             |D_z^{\beta} D_{\zeta}^{\alpha} \aco(z,\zeta) |, |D_z^{\beta} D_{\zeta}^{\alpha} \ro(z,\zeta) | \leq C_{\alpha,\beta}.
        \end{align}
    \end{enumerate}
\end{lemma}
\begin{proof}
    Let $\psi \in \Cc([-1,1],[0,1])$ satisfy $\psi(t) \equiv 1$ for $t \in [-1/2, 1/2]$. 
    Then let $\aco = a(t,x) \psi(t-\omega_t)$. Clearly $\aco \leq a.$
    
    Now, by the uniform continuity of $a$, there exist $r_1 \in (0,1/4)$ such that $a(z) \geq \frac{\Cm}{4}$ for $|z-\omega_z|<2r_1$ with $z \in \Rb^4$. 
    Let $\rho \in \Cs(B(\omega_z,2r_1),[0,1])$ with $\rho \equiv 1$ on $B(\omega_z,r_1) \subset \Rb^4.$ Then 
    \begin{equation}
        \ro(t,x) = \frac{\Cm}{4} \rho(t,x),
    \end{equation}
    satisfies
    \begin{equation}
        \aco = a(t,x) \psi(t-\omega_t) \geq  \frac{\Cm}{4} \rho(t,x)=\ro(t,x)  \geq 0.
    \end{equation}
    Note that $r_1$ does not depend on $\omega$ and $\ro \geq \frac{\Cm_2}{4}$ on $B(\omega, r_1) \subset \Phipm(\Trfo)$.
    Clearly, the $t$ supports of $\aco$ and  $\ro$ are contained in a ball of radius $1$ around $\omega_t$. Finally, since the construction of $\aco, \ro$ is uniform in $\omega$, it is immediate that they satisfy the desired derivative estimates uniformly in $\omega$. Note also that, although we consider $\aco, \ro$ as functions on $\Phipm(\Trfo),$ they have no $\zeta$ dependence.
\end{proof}

\subsubsection{Product Coordinates Around Null Bicharacteristics}
    We now turn our attention to $\omega$ satisfying $a(\omega) \leq \frac{\Cm_2}{2}$.

    To begin, we define product coordinates around the null bicharacteristic starting from $\omega$, and point out a continuity property of the damping $a$ in these coordinates when the null bicharacteristic parameter $s$ is held constant.
\begin{lemma}\label{l:productCoord}
    (Product coordinates for $\omega \in \homogpmsemitrappedset$) Let $\omega \in \pmhomogftrappedcompactset$, resp. $\pmhomogbtrappedcompactset$, and let $\Sigma_{\omega} = \{t = \omega_t\} \times \Rb^3$ be a hypersurface in $\Rb^4$ transverse to $\gopms$, then define
    \begin{align}
        &\Psiopm: [-2, T_2+2] \times \Sigma_{\omega} \times \Pizeta(\Phipm(\Trf)) \ra \Phipm(\Trfo),  \quad\text{resp.} [-T_2-2,2]\\
        &\Psiopm(s, \ti z,\ti \zeta):= \varphi^{\pm}_s(\ti{z}, \ti{\zeta}).
    \end{align}
    This $\Psiopm$ is a diffeomorphism onto its image.

    Furthermore, there exists $r_0>0$ such that for all $\omega \in \pmhomogftrappedcompactset$, resp. $\pmhomogbtrappedcompactset$, all $s \in [-2,T_2+2], \text{ resp. }s \in [-T_2-2,2]$, all $\ti z\in \{\omega_t\} \times B(\omega_x, 2r_0) \subset \Sigma_{\omega}$, and all $\ti \zeta\in B(\omega_z, 2r_0) \subset \Pizeta(\Phipm(\Trfo))$, 
    we have
    \begin{equation}\label{eq:productCoordaest}
        \left| a(\Psiopm(s,\ti z,\ti \zeta))-a(\Psiopm(s,\omega_z, \omega_{\zeta})) \right| \leq \frac{\Cm_2}{4}.
    \end{equation}
\end{lemma}
    See Figure \ref{fig:prodCoords} for a diagram of these sets and points.
    
    \begin{figure}[ht]
        \centering
        	\begin{tikzpicture}[scale=1.5]
		
		\def\radius{1}
		\def\startAngle{-20}
		\def\endAngle{20}
		\def\tiltAngle{45}
		
		\begin{scope}[rotate=\tiltAngle]
			\draw[thick] (0,0) -- (\startAngle:\radius);
			\draw[thick] (0,0) -- (\endAngle:\radius);
			\draw[imayou, thick] (\startAngle:\radius) arc (\startAngle:\endAngle:\radius) node[above left, imayou] {$B(\omega_{\zeta}, 2r_0)$};
			\draw[persred, ->] (0,0)--(0:\radius+.2) node[above]{$\omega_{\zeta}$};
		\end{scope}
		\draw[-](-3,0)--(3,0) node[right] {$t=\omega_t$};
		
		\draw[->] (-3,-1) -- (3,-1) node[right] {$x \in \mathbb{R}^3$};
		\draw[->] (-3,-1) -- (-3,2) node[above] {$t$};
		
		\draw[sand, very thick] (-.5,0)--(.5,0) node[below right] {$B(\omega_x, 2r_0)$};
		
		\node at (0,0) [below] {$\omega_{x}$};
		
		\draw[-] (-3.1,-.5)--(-2.9,-.5);
		\node at (-3,-.5) [left] {$s=-2\;$};
		\draw[-] (-3.1,1.5)--(-2.9,1.5);
		\node at (-3,1.5) [left] {$s=T_{2}+2\;$};

	\end{tikzpicture}
        \caption{Key sets and points used in the product coordinate construction.}
        \label{fig:prodCoords}
    \end{figure}
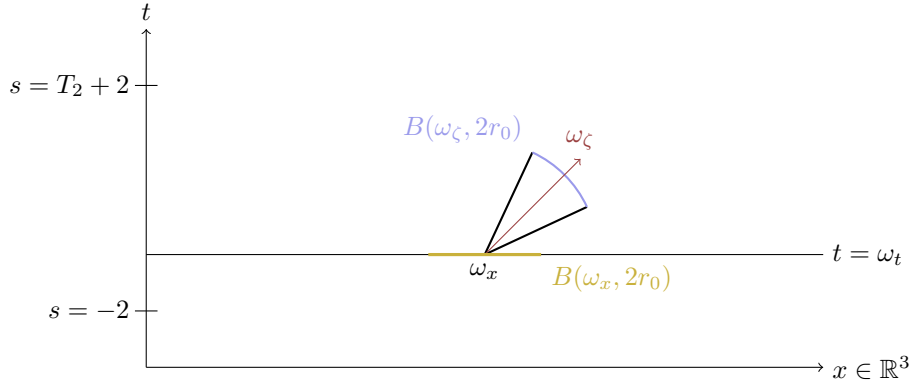

\begin{proof}
    It is immediate that $\Psiopm$ is a diffeomorphism onto its image, where we note that the Hamilton flow maintains inclusion in $\Phipm(\Trfo)$ and we do not track $\tau$, as it is constant under the Hamilton flow. 

    We prove that $r_0$ can be chosen uniformly for all $\omega \in \pmhomogftrappedcompactset$, as the proof for $\omega \in \pmhomogbtrappedcompactset$ is analogous. Note that by the uniform continuity of $a$, there exists a $\d>0$ such that $|z_1-z_2|<\d$ implies $|a(z_1)-a(z_2)| \leq \Cm/4.$ Then define
    \begin{equation}
        f(\omega, s, \eta) = \text{diam}(\{\Psi_{\omega}(s,\ti z, \ti \zeta); \ti z \in \{\omega_t\}\times B(\omega_x, 2\eta), \ti \zeta \in B(\omega_{\zeta}, 2 \eta)\}).
    \end{equation}
    This $f$ is continuous in $\omega, s,$ and $\eta,$ by continuity of the flow. 
    Note also, because the space-time, is stationary $f(\omega,s, \eta)$ does not depend on the $t$ coordinate of $\omega.$ 
    Since $s\in[-2,T_2+2]$, $\eta \in [0,1]$, and $\omega \in \pmhomogftrappedcompactset$ is compact in $x,\xi$ and $\tau$, $f$ is continuous on a compact set. Thus $f$ is uniformly continuous, and there exists $r_1>0$ such that $dist((\omega_1, s_1, \eta_1),(\omega_2,s_2,\eta_2))<r_1$ implies $|f(\omega_1,s_1,\eta_1)-f(\omega_2,s_2,\eta_2)|<\d.$ 
    This along with $f(\omega,s,\eta=0)=0$ means $|f(\omega, s, r_1)| <\d,$ for all $\omega \in \pmhomogftrappedcompactset, s\in [-2,T_2+2]$. 
    So, choosing $r_0<r_1$ guarantees that for all $\omega$ with $\ti z \in \{\omega_t\}\times B(\omega_x, 2r_0), \ti \zeta \in B(\omega_{\zeta}, 2r_0)$ that \eqref{eq:productCoordaest} is satisfied. 
\end{proof}
Although we define our coordinates for  $s\in [-2,T_2+2]$, to ensure that our escape function can turn off smoothly we will restrict its $s$ support to a smaller interval. We define the key values used to determine this smaller interval below, and point out the key properties involving the TGCC time and the damping $a$, which we will use later.
\begin{lemma}\label{l:epsilonCoordinateDefine}
    For $\omega \in \pmhomogftrappedcompactset$, resp. $\omega \in \pmhomogbtrappedcompactset$, and $\omega \in \{a(\omega) \leq \frac{\Cm_2}{2} \}$, let $T_{\omega}$ be the smallest $T$ such that 
   \begin{align}\label{eq:smallestTgcc}
        \frac{1}{T} \int_0^{T} a(\gopms) ds &= \Cm_2\\
        \text{resp.} \qquad \frac{1}{T} \int_{-T}^0 a(\gopms) ds &= \Cm_2.  
    \end{align}
    Then $\Topm \leq T_2$. Furthermore, $a(\gopmr{\Topm}) \geq \Cm_2$ and there exists $\e \in (0,1)$ such that for all $\omega \in \pmhomogftrappedcompactset$, resp. $\omega \in \pmhomogbtrappedcompactset$, we have
    \begin{equation}
        a(\gopms) \geq \frac{3 \Cm_2}{4} \text{ for } s \in [\Topm, \Topm+\e],\quad \text{ resp. } s \in [-\Topm-\e, -\Topm].
    \end{equation}
    \end{lemma}
    \begin{proof}
   We prove the case where $\omega \in \pmhomogftrappedcompactset$, as the argument when $\omega \in \pmhomogbtrappedcompactset$ is analogous.  
   It is immediate from Proposition \ref{prop:oneSidedGCC} that $\Topm \leq T_2$. 

   By definition, $\Topm$ is the first positive zero of the differentiable function
   \begin{equation}
       f(T)=\Cm_2 T - \int_0^T a(\gopms) ds.
   \end{equation}
   We have that $f(0)=0, f'(0) \geq \Cm_2 - \frac{\Cm_2}{2}>0$. Therefore $f$ is positive for $0<T<\Topm$, so we have $0 \geq f'(\Topm)=\Cm_2-a(\gopmr{\Topm})$. That is $a(\gopmr{\Topm}) \geq \Cm_2.$

    Now, for all $\omega \in \pmhomogftrappedcompactset,$ $s \in \Rb$, and 
    $\e>0$ 
    by the definition of the flow \eqref{eq:flowdef} we have
    \begin{equation}
         \xpms(t+\e)-\xpms(t)=\int_t^{t+\e} \frac{d}{ds} \xpms ds =\int_t^{t+\e} \nabla_{\xi} \ppm(\phipm_s(\omega))ds = \int_t^{t+\e} \nabla_{\xi} \bpmxs ds.
    \end{equation}
    Note also that $\omega_{\tau}=1$, since $\omega \in \homogplusminusforwardtrappedset$, so $\frac{d}{ds} t_{s}^{\pm}=1$. Therefore
    \begin{equation}
        \bigg| \Piz\bigg(\gopms-\gopmr{s+\e}\bigg)\bigg|\leq|x^{\pm}_{s+\e}(\omega)-x^{\pm}_s(\omega)| + |t^{\pm}_{s+\e}(\omega)-t^{\pm}_s(\omega)| \leq \e (\max_{\omega \in \pmhomogftrappedcompactset} \left(|\nabla_{\xi}\bpm(x,\xi)| +1  \right). 
    \end{equation}
    Note that $\nabla_{\xi} b$ attains a maximum, as a continuous function on a compact set (of course $\pmhomogftrappedcompactset$ is not compact in $t$, but $b$ does not depend on $t$). Then by uniform continuity of $a$, there exists $\e>0$, such that $|z_1-z_2| < \e \max(|\nabla_{\xi} b|+1)$, implies $|a(z_1)-a(z_2)| \leq \frac{\Cm_2}{4}.$ Applying this to $a(\gopmr{\Topm}) \geq \Cm_2$ proves the desired claim. 
    \end{proof}

\subsubsection{Escape Functions Around Individual Null Bicharacteristics - Small Damping}\label{s:smallDampEscape}
We now construct escape functions along null bicharacteristics with initial data $\omega$ satisfying $a(\omega) \leq \frac{\Cm_2}{2}$. 

We do this across four lemmas. 
\begin{enumerate}
    \item In Lemma \ref{l:escapeqdef}, we define the escape function $\qo$ and estimate $\Hpm \qo$.
    \item In Lemma \ref{l:escapeadef}, we define a correction term $\aco$ which is bounded by a multiple of the damping $a$, and which we add to $\Hpm \qo$ to ensure positivity
    \item In Lemma \ref{l:escaperdef}, we define $\ro$ and show it bounds $\Hpm \qo + \aco$ from below, and is uniformly positive in a fixed width neighborhood of $\omega$.
    \item Finally in Lemma \ref{l:escapedefUniform}, we demonstrate fixed size $t$ support and derivative bounds of $\qo, \aco$, and $\ro$ which are uniform in $\omega$. This uniformity is key to our combination of these functions from multiple null bicharacteristics to obtain a global semi-trapped escape function.
\end{enumerate}

In all of these lemmas we make use of  the product coordinates $\Psiopm$ from Lemma \ref{l:productCoord}. Recall also $r_0$ from that lemma, and let 
\begin{equation}
    \phiopm \in \Cc(\{\omega_t\} \times B(\omega_x, 2r_0)), \qquad  \chiopm \in \Cc(B(\omega_{\zeta}, 2r_0)),
\end{equation}
be non-negative and $1$ on $\{\omega_t\} \times B(\omega_x, r_0)$, resp. $B(\omega_{\zeta}, r_0) \subset\Pizeta(\Phipm(\Trfo)).$ These functions $\phiopm, \chiopm$ will determine the $\ti z$ and $\ti \zeta$ dependence of $\qo, \aco,$ and $\ro$. Only the $s$ dependence will change between $\qo,\aco,$ and $\ro$. 

Recall also $\Cm_2, T_2$ from Proposition \ref{prop:oneSidedGCC} and $\Topm, \e$ from Lemma \ref{l:epsilonCoordinateDefine}. From Lemma \ref{l:epsilonCoordinateDefine} $\Topm \leq T_2$ and $\e \in (0,1)$, so $[-1,\Topm+\e] \subset [-2, T_2+2]$ and $[-\Topm-\e,1] \subset [-T_2-2,2]$. The functions $\qo, \aco,$ and $ \ro$ will be compactly supported for $s\in [-1,\Topm+\e]$, or $[-\Topm-\e,1]$, and then extended by zero for all other values of $s$.

We begin by defining the escape function $\qo$ and compute the Hamilton flow applied to it.
\begin{lemma}\label{l:escapeqdef}
    There exist constants $C^*, \e_1>0$ such that for all $\omega \in \pmhomogftrappedcompactset \cap\{a(\omega)\leq\frac{\overline C_2}{2}\}$, resp. $\omega \in \pmhomogbtrappedcompactset\cap\{a(\omega)\leq\frac{\overline C_2}{2}\}$, there exists a nonnegative, resp. non-positive, function $\qo \in \Cc(\Phipm(\Trfo))$ given by
    \begin{align}
        &\qo(s,\ti z,\ti \zeta) = \alphaopm(s) \phiopm(\ti z) \chiopm(\ti \zeta)\\
        &\alphaopm(s) = \left(\Cm_2s - \int_0^s a(\gopmr{r})dr + \e^2\right), \quad \text{for } s \in [0,\Topm],\\
        \text{resp. } &\alphaopm(s) = \left(\Cm_2s + \int_s^0 a(\gopmr{r})dr - \e^2\right), \quad \text{for } s \in [-\Topm,0],
    \end{align}
    such that 
    \begin{align}
        &\Hpm \qo(s,\ti z,\ti \zeta) = \p_s \alphaopm(s) \phiopm(\ti z) \chiopm(\ti \zeta), \qquad \text{ and } \\
        &\p_s \alphaopm(s)  \geq \begin{cases}
        0 & s \in [-1,0], \quad \text{resp. } s\in [0,1]\\
        \frac{\Cm_2}{4} & s \in [-\e_1, \e_1]\\
        (\Cm - a(\gopms))  & s\in [0,\Topm], \quad \text{resp. } s\in [-\Topm,0]\\
        -C^* & s \in [\Topm, \Topm+\e], \quad \text{resp. } s\in [-\Topm-\e, -\Topm].
        \end{cases} \label{eq:alphaconditions}
    \end{align}
\end{lemma}
\begin{proof}
    We assume $\omega \in \pmhomogftrappedcompactset$, as the proof with $\omega \in \pmhomogbtrappedcompactset$ is analogous. 
    In the product coordinates from $\Psiopm$, we have
    \begin{equation}
        \Hpm\qo = \p_s \qo(s,\ti z,\ti \zeta) = \p_s \alphaopm(s) \phiopm(\ti z) \chiopm(\ti \zeta).
    \end{equation}
    Therefore $\qo$ solves
    \begin{equation}
        \begin{cases}
            \Hpm \qo = (\Cm_2 - a(\gopms)) \phiopm(\ti z) \chiopm(\ti \zeta) \\
            \qo = \e^2 \phiopm(\ti z) \chiopm(\ti \zeta) \quad \text{on } \quad\Psiopm(s=0,\ti z,\ti \zeta),
        \end{cases}
    \end{equation}
    for $s \in [0,\Topm]$. Note that $\alphaopm \geq \e^2$ on $[0,\Topm]$, by definition of $\Topm$ as the smallest $T$ such that $\frac{1}{T} \int_0^T a(\gopms) ds = \Cm_2$, since $a(\gopmr{0}) \leq \frac{\Cm_2}{2}$.
    
    To extend $\qo$ to all of $\Phipm(\Trfo)$ we will extend $\alphaopm$ to a smooth non-negative function compactly supported in $s \in [-1,\Topm+\e]$, which is compatible with the conditions in \eqref{eq:alphaconditions}.

    First recall $a(\gopmr{0}) \leq \frac{\Cm_2}{2}$, and so, by uniform continuity of $a$, there exists $\e_1$ such that $a(\gopms) \leq \frac{3\Cm_2}{4}$ for all $\omega$ and all $s \in [0,
    \e_1]$. Therefore $\p_s \alphaopm(s) = \Cm_2 - a(\gopms) \geq \frac{\Cm_2}{4}$ for $s \in [0,\e_1]$. 
    
    Since $\alphaopm(-1)=0$, $\alphaopm(0)=\e^2>0$, and $\p_s \alphaopm|_{s=0} = \Cm_2 - \alpha(\gopmr{0}) \in [\frac{\Cm_2}{2}, \Cm_2]$, it is straightforward
    to see that $\alphaopm$ can be defined so that $\p_s \alphaopm \geq 0$ on $[-1,0]$ and $\p_s \alphaopm \geq \frac{\Cm_2}{4}$ on $[-\e_1, 0]$.
    

    Finally, since $|\p_s \alphaopm(\Topm)| = |\Cm - a(\gopmr{\Topm}| \leq \lp{a}{\infty}$, and $\alphaopm(\Topm)=\e^2$, then $\alphaopm$ can be constructed so that for some $C^* \geq 0$ and all $\omega \in \pmhomogftrappedcompactset$
    \begin{equation}
        \p_s \alphaopm(s) \geq -C^* \text{ for } s \in [\Topm, \Topm+\e].
    \end{equation}
\end{proof}
We now define the function we add to $\Hpm \qo$ to obtain positivity for $s \in [\e_1, \Topm +\e]$.
\begin{lemma}\label{l:escapeadef}
    
    Let $\ti{a} \in \Cc\left(\left\{a \geq \frac{\Cm_2}{2}\right\}\right)$ be a non-negative, smooth function with uniform derivative bounds such that 
    \begin{align}
        &\ti{a} \leq a, \text{ and} \\
        &\ti{a}=a \text{ on } \left\{a \geq \frac{3\Cm_2}{4} \right\}.
    \end{align} 
    For $\omega \in \pmhomogftrappedcompactset$, resp. $ \omega \in \pmhomogbtrappedcompactset$, let $\rhopm$ be a smooth nonnegative function, bounded by 1, such that 
    \begin{enumerate}
        \item $\rhopm$ is compactly supported in $[-2, \Topm+2]$, resp. $[-\Topm-2, 2]$.
        \item $\rhopm \equiv 1$ on $[-1, \Topm+\e]$, resp. $[-\Topm-\e,1]$.
        \item $\rhopm$ is chosen independent of $\omega$ on $[-2,-1]\cup[\Topm+\e,\Topm+2]$, resp. $[-\Topm-2, -\Topm-\e] \cup[1,2]$.
    \end{enumerate}
    Then, recalling $C^*$ from Lemma \ref{l:escapeqdef}, define
    \begin{equation}
        \aco(s, \ti z, \ti \zeta) = 2\left(1+\frac{C^*}{\Cm_2}\right) \ti{a}(\gopms) \rhopm(s) \phiopm(\ti z) \chiopm(\ti \zeta).
    \end{equation}
    There exists $C_1>0$ such that for any $\omega \in \homogpmsemitrappedset\cap\{a(\omega)\leq\frac{\overline C_2}{2}\}$, we have 
    \begin{equation}
        0 \leq \aco \leq C_1 a.
    \end{equation}
\end{lemma}
\begin{remark}
    It would be more straightforward if we could obtain positivity for 
    $$\Hpm \qo + Ca(\Psi(s,\ti z, \ti \zeta)).$$
    However, we must be more careful because $\Hpm \qo$ is expressed in terms of $a(\gopms)$ and 
    \begin{equation}
     a=a(\Psiopm(s,\ti z,\ti \zeta)) \neq a(\Psiopm(s,\omega_{z}, \omega_{\zeta}))=a(\gopms).
    \end{equation}
    In particular, the former depends on $(s, \ti z,\ti \zeta)$ while the latter depends only on $s$.
    
\end{remark}
\begin{proof}
    Again we only prove the case $\omega \in \pmhomogftrappedcompactset$, as the proof when $\omega \in \pmhomogbtrappedcompactset$ is analogous. 
    
    Since $\gopms=\Psiopm(s, \omega_z, \omega_{\zeta})$, by Lemma \ref{l:productCoord} we have
    \begin{equation}
        \left| a(\gopms) - a(\Psiopm(s,\ti z,\ti \zeta)) \right| \leq \frac{\Cm_2}{4}, \quad (s,\ti z,\ti \zeta) \in [-1, \Topm+\e] \times\{\omega_t\}\times B(\omega_x, 2r_0) \times B(\omega_{\zeta}, 2r_0).
    \end{equation}
    Therefore, when $a(\gopms) \geq \frac{\Cm_2}{2}$, i.e. where $\ti{a} \geq 0$, we have  
    \begin{equation}
        a(\Psiopm(s, \ti z,\ti \zeta)) \geq a(\gopms) - \frac{\Cm_2}{4} \geq \frac{\Cm_2}{2}-\frac{\Cm_2}{4}=\frac{\Cm_2}{4},
    \end{equation}
    then rearranging
    \begin{equation}
        \ti{a}(\gopms) \leq a(\gopms) \leq a(\Psiopm(s,\ti z,\ti \zeta)) + \frac{\Cm_2}{4} \leq 2 a(\Psiopm(s,\ti z,\ti \zeta)).
    \end{equation}
    Thus, there exists $C_1:=4(1+\frac{C^*}{\Cm_2})>0$, so that $\aco(s,\ti z,\ti \zeta) \leq C_1 a(\Psiopm(s,\ti z,\ti \zeta)).$ 
\end{proof}
We now show that $\Hpm \qo+\aco$ is bounded from below in a fixed width neighborhood of $\omega$ and characterize this lower bound.
\begin{lemma}\label{l:escaperdef}
    Recall the constant $\e_1>0$ from Lemma \ref{l:escapeqdef}. Let $\mu(s) \in \Cs([-\e_1, \e_1]:[0,1])$ satisfy $\mu(s) =1$ for $s \in \left[-\frac{\e_1}{2}, \frac{\e_1}{2}\right]$. 
    There exists a constant $r_1>0$ such that for any $\omega \in \homogpmsemitrappedset\cap\{a(\omega)\leq\frac{\overline C_2}{2}\}$, if we define
    \begin{equation}
        \ro(s,\ti z,\ti \zeta) = \frac{\Cm_2}{4} \mu(s) \phiopm(\ti z) \chiopm(\ti \zeta),
    \end{equation}
    then we have 
    \begin{enumerate}
        \item $\ro \geq 0$
        \item $\ro \geq \frac{\Cm_2}{4}$ on $B(\omega, r_1) \subset \Phipm(\Trfo)$, and 
        \item For all $(z,\zeta) \in \Phipm(\Trfo)$
        \begin{equation}
            \Hpm \qo(z,\zeta) + \aco(z,\zeta) \geq \ro(z,\zeta).
        \end{equation}
    \end{enumerate}
\end{lemma}
\begin{proof}
    We again specialize to the case $\omega \in \pmhomogftrappedcompactset$, as the proof when $\omega \in \pmhomogbtrappedcompactset$ is analogous.
    
    1) Note that by construction $\ro \geq 0.$

    2) Since $\phiopm \equiv 1$ on $B(\omega_x, r_0), \chiopm\equiv 1$ on $B(\omega_{\zeta}, r_0)$ we have $\ro(s,\ti z,\ti \zeta) \geq \frac{\Cm}{4}$ for $(s,\ti z,\ti \zeta) \in (-\frac{\e_1}{2}, \frac{\e_1}{2}) \times \{\omega_t\}\times B(\omega_x, r_0) \times B(\omega_{\zeta}, r_0)$. Now using the continuity of the flow, the compactness of $\Pi_t^{\bot} \pmhomogftrappedcompactset$, and the time-independence of the flow, there exists $r_1 >0$, such that for all $\omega \in \pmhomogftrappedcompactset$, the set $B(\omega, r_1) \subset \Phipm(\Trfo)$ satisfies the inclusion
    \begin{equation}
        B(\omega, r_1) \subset \Psiopm\left(  (-\frac{\e_1}{2}, \frac{\e_1}{2}) \times \{\omega_t\}\times B(\omega_x, r_0) \times B(\omega_{\zeta}, r_0) \right).
    \end{equation}
    Therefore $\ro \geq \frac{\Cm_2}{4}$ on $B(\omega, r_1)$ as desired. 
    
    3) To prove the lower bound, beginning with Lemmas \ref{l:escapeqdef} and \ref{l:escapeadef} we have
    \begin{equation}
        (\Hpm \qo + \aco)(s,\ti z, \ti \zeta) = \phiopm( \ti z ) \chiopm(\ti \zeta) ( \p_s \alphaopm(s) + 2 \left(1+\frac{C^*}{\Cm_2} \right) \ti{a}(\gopms))\rhopm(s).
    \end{equation}
    Thus it is enough to show that for $s \in [-2, \Topm +2]$
    \begin{equation}
    \p_s \alphaopm(s) + 2 \left(1+\frac{C^*}{\Cm_2} \right)\ti{a}(\gopms) \rhopm(s) \geq \frac{\Cm_2}{4} \mu(s).
    \end{equation}
     For $s \in [-2,-1]\cup [\Topm+\e, \Topm+2]$ we have $\mu=\alphaopm=0,$ and $ \ti{a}, \rhopm \geq 0$, so the desired statement is immediately true there. Since $\rhopm(s) \equiv 1$ for $s \in [-1,\Topm+\e]$ we drop it from the remaining computations. We consider three cases
    \begin{enumerate}[label=\roman*)]
        \item $s \in [-1,0]$,
        \item $s \in [0,\Topm]$,
        \item $s \in [\Topm, \Topm+\e]$.
    \end{enumerate}
    In case i) $\mu(s)\leq 1$ on $[-\e_1, 0]$ and $\mu=0$ elsewhere, so by \eqref{eq:alphaconditions}
    \begin{equation}
        \p_s \alphaopm(s) + 2 \left(1+ \frac{C^*}{\Cm_2}\right) \ti{a}(\gopms) \geq \p_s \alphaopm(s) \geq \frac{\Cm_2}{4} \mu(s).
    \end{equation}
    In case ii) consider two subcases 

    \noindent a) On the set where $\{a \leq \frac{3\Cm_2}{4}\}$, $\p_s \alphaopm=\Cm_2 -a(\gopm) \geq \frac{\Cm_2}{4}$ and so 
    \begin{equation}
        \p_s \alphaopm(s) + 2\left(1+\frac{C^*}{\Cm_2}\right) \ti{a}(\gopms) \geq \frac{\Cm_2}{4} \geq \frac{\Cm_2}{4} \mu(s).
    \end{equation}
    b) On the set where $\{a \geq \frac{3 \Cm_2}{4}\}$, $\ti{a}(\gopms)=a(\gopms) \geq 0$ and so 
    \begin{equation}
        \p_s \alphaopm(s) + 2\left(1+\frac{C^*}{\Cm_2}\right) \ti{a}(\gopms) \geq \Cm_2 - a(\gopms) + 2 a(\gopms) \geq \Cm_2 \geq \frac{\Cm_2}{4} \mu(s). 
    \end{equation}
    In case iii) by Lemma \ref{l:epsilonCoordinateDefine}, $a(\gopms) \geq \frac{3\Cm_2}{4}$, so $\ti{a}(\gopms)=a(\gopms)\geq \frac{3\Cm_2}{4}.$ Also by \eqref{eq:alphaconditions}, $\p_s \alphaopm(s) \geq - C^*$ here. Thus 
    \begin{equation}
        \p_s \alphaopm(s) + 2\left(1+\frac{C^*}{\Cm_2}\right)  \ti{a}(\gopms) \geq -C^* + \frac{2C^*}{\Cm_2}\frac{3\Cm_2}{4} \geq 0 = \frac{\Cm_2}{4}\mu(s),
    \end{equation}
    because $\mu=0$ outside of $[-\e_1, \e_1].$
\end{proof}
We now mention some properties of these functions that are uniform in $\omega$. This uniformity is a key part of our next step and follows from the preceeding constructions.
\begin{lemma}\label{l:escapedefUniform}
    For $\omega \in \pmhomogftrappedcompactset$ (resp. $\pmhomogbtrappedcompactset$) the functions $\qo, \aco, \ro$ defined in Lemmas \ref{l:escapeqdef}, \ref{l:escapeadef}, and \ref{l:escaperdef} satisfy
    \begin{enumerate}
        \item  
        \begin{align}
            \Pi_t \supp(\qo), \Pi_t \supp(\aco), \Pi_t \supp(\ro) \subset \{t \in [\omega_t-2,\omega_t+T_2+2)]\}, \\
            \text{resp. }\Pi_t \supp(\qo), \Pi_t \supp(\aco), \Pi_t \supp(\ro) \subset \{t \in [\omega_t-T_2-2,\omega_t+2)]\},
        \end{align}
        \item for each multi-index $\alpha, \beta$, there exists $C_{\alpha\beta} >0$, such that for all $\omega$ and all $(z,\zeta) \in \Phipm(\Trf)$
        \begin{align}
             |D_z^{\beta} D_{\zeta}^{\alpha} \qo(z,\zeta) |, |D_z^{\beta} D_{\zeta}^{\alpha} \aco(z,\zeta) |, |D_z^{\beta} D_{\zeta}^{\alpha} \ro(z,\zeta) | \leq C_{\alpha,\beta}.
        \end{align}
    \end{enumerate}
\end{lemma}
\begin{proof}    
    We again only prove the case $\omega \in \pmhomogftrappedcompactset$ as the proof when $\omega \in \pmhomogbtrappedcompactset$ is analogous. 
    
    To see property (1), first note that for $\omega \in \homogpmsemitrappedset \subset \Phipm(\Trfo) \cap \characteristicsetofPplusminus$, we have
    $\omega_{\tau}=1$. Since $\frac{d}{ds} \Pi_t \gopms = \p_{\tau} \ppm=1$, we have
    \begin{equation}
        |\Pi_t \gopms - \omega_t | = s.
    \end{equation}
    By their constructions $\qo, \aco, \ro$ have compact $s$ support in $[-2,\Topm+2]$, and so their support in $t$ is contained in $[-2+ \omega_t, \Topm+2 + \omega_t]$ as desired. 
    
    To see property (2), first note that for each $\omega$ such a $C_{\alpha \beta}$ exists by the smoothness and compact support of $\qo, \aco, \ro$. To see that this $C_{\alpha \beta}$ can be taken uniformly in $\omega$, first note by their construction $\phiopm, \chiopm$ have a uniform upper bound on their derivatives. Similarly, by  the construction of $\alphaopm$ in $[-1,0] \cup [\Topm, \Topm+\e]$, there is a uniform bound on its derivatives there. Similarly, because the bounds on the derivatives of $a$ have uniform upper bounds, the same is true for derivatives of $\ti{a},$ where we note that $\ti{a}$ is defined globally and does not depend on $\omega$. This, along with the choice of $\rhopm$ means there are bounds on the derivatives of $\qo, \aco$ which are uniform in $\omega$. Finally $\mu(s)$ does not depend on $\omega$ and so there are bounds on the derivatives of $\ro$, which are uniform in $\omega$.
\end{proof}

\subsubsection{Reduction to a Locally Finite Number}
At this point, we have constructed an escape function for each $\omega\in\homogpmsemitrappedset$. We reduce to a locally finite number of escape functions by covering $\homogpmsemitrappedset$ with  neighborhoods where $\{\ro \geq \frac{\Cm_2}{4}\}$ and using compactness and time-invariance of the space-time. Then, for that locally finite number of escape functions, we extend them to $\Trf$ via homogeneity.

\begin{lemma}\label{l:fulltrappedescapelemma}
Fix $R>R_0$. There exists an open set $\semitrappedescapeset \supset \pmsemitrappedset$, a constant $C_a>0$, and functions $\qpm, \acpm, \rpm \in \Cs(\Trfo)$, with $\acpm, \rpm$ nonnegative such that 
\begin{enumerate}
    \item 
    \begin{equation}
        \Hpm \qpm + \acpm \geq \rpm \geq \frac{\Cm_2}{4} \mathbbm{1}_{\semitrappedescapeset}.
    \end{equation}
    \item The symbol $\acpm$ is uniformly bounded by a multiple of the damping
    \begin{equation}
        \acpm(t,x,\tau,\xi) \leq C_a a(t,x) \quad \text{ for all } (t,x,\tau,\xi)\in \Trfo.
    \end{equation}
    \item The symbols $\qpm, \acpm, \rpm$, are $0$-homogeneous in $\xi$ and $\tau$. That is there exist $\qpmo, \acpmo, \rpmo \in \Cc(\Trfo)$ such that 
    \begin{equation}
        \qpm = \qpmo \circ \Phipm, \quad \acpm =\acpmo \circ \Phipm, \quad \rpm = \rpmo \circ \Phipm.
    \end{equation}
    \item The set $\semitrappedescapeset$ is uniform in $t$. That is, $\semitrappedescapeset$ is a product of a set that does not depend on $t$, with $\Rb_t$
    \begin{equation}
        \semitrappedescapeset = \Rb_t \times \Pi_t^{\bot} \semitrappedescapeset.
    \end{equation}
    \item Finally $\qpm, \acpm, \rpm$ satisfy $S^0(\Trfo)$ symbol estimates from 
    Definition \ref{def:Symbol}. 
\end{enumerate}
\end{lemma}

\begin{proof}
    We first construct $\qpmo, \acpmo,\rpmo$ and $\homogsemitrappedescapeset$, then extend them to define $\qpm, \acpm, \rpm$ and $\semitrappedescapeset$.
    
    The set $\homogpmsemitrappedset \cap \{t \in [0, T_2]\}:= \homogpmcompactsemi$ is compact. For $\omega \in \homogpmcompactsemi$, if $a(\omega) \geq \frac{\Cm_2}{2}$ we let $\aco$ and $\ro$ be as in Lemma \ref{l:escapeLargeDamp} and $\qo \equiv 0$. If $a(\omega) \leq \frac{\Cm_2}{2}$ we let $\qo, \aco, $ and $\ro$ be as in Lemma \ref{l:escaperdef}, where if $\omega \in \pmhomogftrappedcompactset \cap \pmhomogbtrappedcompactset$ we choose the $\qo, \aco, $and $\ro$ from $\pmhomogftrappedcompactset$.
    
    We also let $r_1$ be the minimum of the $r_1$'s from Lemmas \ref{l:escapeLargeDamp} and \ref{l:escaperdef}, so that for all $\omega \in \homogpmcompactsemi$ we have $\{\ro \geq \Cm_2/4\} \supset B(\omega, r_1)$. Then 
    \begin{equation}
    \bigcup_{\omega \in \homogpmcompactsemi} B(\omega, r_1) \supset \homogpmcompactsemi.
    \end{equation}
    Using compactness we reduce this cover to a finite subcover 
    \begin{equation}
        \bigcup_{j=1}^n B(\omega_{j}, r_1) \supset \homogpmcompactsemi= \homogpmsemitrappedset \cap \{t \in [0, T_2]\}.
    \end{equation}
    Note that because the space-time is stationary, $\homogpmsemitrappedset$ does not depend on $t$. 
    Thus there exists $\mathcal{U}_R$, an open set in $\Pi_t^{\bot} \Trf$, such that 
    \begin{equation}
         \bigcup_{j=1}^n B(\omega_{j}, r_1) \supset [0,T_2] \times \mathcal{U}_R  \supset  \homogpmsemitrappedset \cap \{t \in [0, T_2]\}.
    \end{equation}
    Furthermore translating the $t$-coordinate of each $\omega_j$ by $kT_2$ for any $k \in \Zb$, provides a finite cover of $\homogpmsemitrappedset \cap \{t \in [kT_2, (k+1)T_2]\}$. That is, defining 
    \begin{equation}
        \omega_{j,k}=((\omega_j)_t+kT_2, (\omega_j)_x, (\omega_j)_{\tau}, (\omega_j)_{\xi}),
    \end{equation}
    we have
    \begin{equation}
        \bigcup_{j=1}^n B(\omega_{j,k}, r_1) \supset [kT_2, (k+1)T_2] \times\mathcal{U}_R \supset \homogpmsemitrappedset \cap \{t \in [kT_2, (k+1)T_2]\}.
    \end{equation}
    Then defining 
    \begin{equation} \label{eq:VRdef}
        \homogsemitrappedescapeset :=  \Rb_t \times \mathcal{U}_R, \text{ we have } 
       \homogpmsemitrappedset \subset  \homogsemitrappedescapeset \subset \bigcup_{k=-\infty}^{\infty} \bigcup_{j=1}^n B(\omega_{j,k}, r_1).
    \end{equation}
    We further define 
    \begin{align}
        \qpmo(t, x, \tau, \xi) =\sum_{k=-\infty}^{\infty} \sum_{j=1}^n q^{\pm}_{\opmjk}(t, x, \tau, \xi), \\
        \acpmo(t, x, \tau, \xi)=\sum_{k=-\infty}^{\infty} \sum_{j=1}^n \ac^{\pm}_{\opmjk}(t, x, \tau, \xi), \\
        \rpmo(t, x, \tau, \xi)=\sum_{k=-\infty}^{\infty} \sum_{j=1}^n r^{\pm}_{\opmjk}(t, x, \tau, \xi). 
    \end{align}
    For any $t_0 \in \Rb,$ there exists exactly one $k \in \Zb$ such that $t_0 \in [kT_2, (k+1)T_2)$. Without loss of generality we may assume $T_2 \geq 1$. Then by Lemmas \ref{l:escapeLargeDamp} and \ref{l:escapedefUniform} we have for $l \not \in\{k-2, k-1, k, k+1, k+2\}$
    \begin{equation}
        q^\pm_{\omega_{j,l}}(t_0,\cdot) \equiv \ac^\pm_{\omega_{j,l}}(t_0,\cdot)  \equiv r^\pm_{\omega_{j,l}}(t_0,\cdot)  \equiv 0.
    \end{equation}
    Therefore all of the above sums are locally finite in $t$. That is, for a given $t$, the number of non-zero terms is finite. 
    Thus $\qpmo, \acpmo, \rpmo$ are all smooth, and have compact support in $x, \xi$ and $\tau$. Note also that there exists $C_a >0$, such that $|\acpmo(t,x,\tau,\xi)| \leq C_a a$ for all $(t,x,\tau,\xi) \in \Phipm(\Trfo)$ because each $\ac^{\pm}_{\opmjk}$ satisfies the bound, and there at most $5n$ of the $\ac^{\pm}_{\opmjk}$ which contribute to the value of $\acpmo$ at any point. Furthermore combining  \eqref{eq:VRdef} and Lemmas \ref{l:escapeLargeDamp} and \ref{l:escaperdef} we have 
    \begin{equation}
        \Hpm \qpmo + \acpmo \geq \rpmo \geq \frac{\Cm_2}{4} \mathbbm{1}_{\homogsemitrappedescapeset},
    \end{equation}
    since the above holds,  with $\homogsemitrappedescapeset$ replaced by $B(\omega_{j,k},r_1)$, for each individual $\omega_{j,k}$ in the sums used to define $\qpmo, \acpmo, \rpmo$.

    Now we will extend these functions from $\Phipm(\Trfo)$ to $\Trfo$ via the rescaling. First, we define $V_R^\pm=(\Phi^\pm)^{-1}(\homogsemitrappedescapeset)$. Then define the functions $\qpm, \acpm, \rpm: \Trfo \ra \Rb$ by 
    \begin{equation}
        \qpm=\qpmo \circ \Phipm, \quad \acpm = \acpmo \circ \Phipm, \quad \rpm = \rpmo \circ \Phipm.
    \end{equation}
    Note that using the same $C_a >0$, $|\acpm(t,x,\tau,\xi)| \leq C_a a$ for all $(t,x,\tau,\xi) \in \Trfo$.
    Now to prove the lower bound on $\Hpm \qpm + \acpm$ we will first relate $\Hpm \qpm$ and $\Hpm \qpmo$. Let $(\tpms, \xpms, \taupms, \xipms) =\gpms(t,x,\tau,\xi)$, then
    \begin{equation}
        \Hpm \qpm |_{(t,x,\tau,\xi)} = \frac{d}{ds} (\qpm(\tpms, \xpms, \taupms, \xipms))|_{s=0}.
    \end{equation}
    Since $\bpm$ is constant under the Hamilton flow $\frac{d}{ds} \bpmxs=0$.
    Then, letting 
        $$
            \gpmsd=\left( \tpms, \xpms, \frac{\taupms}{|\bpmxs|}, \frac{\xipms}{|\bpmxs|} \right),
        $$ 
    we have 
    \begin{align}
        \frac{d}{ds} (\qpm(\gpms)) &= \frac{d}{ds} \qpmo(\gpmsd)\\
        &= (\nabla_x \qpmo)(\gpmsd) \frac{d}{ds}{\xpms} + (\nabla_{\xi} \qpmo)(\gpmsd)  \frac{d}{ds} \left(\frac{\xipms}{|\bpmxs|}\right)  \\ 
        &\quad + (\p_t \qpmo)(\gpmsd)  \frac{d}{ds}{\tpms}+ (\p_{\tau} \qpmo)(\gpmsd)  \frac{d}{ds}\left(\frac{\tau_s}{|\bpmxs|} \right) \\
        &=(\nabla_x \qpmo)(\gpmsd)\cdot(\nabla_{\xi} \ppm)(\gpms) - \frac{1}{|\bpmxs|} (\nabla_{\xi} \qpmo)(\gpmsd) \cdot (\nabla_x \ppm)(\gpms)\\
        &\quad+ (\p_t \qpmo)(\gpmsd)  +0 \\
        &=(\nabla_x \qpmo)(\gpmsd)  \cdot (\nabla_{\xi} \ppm)(\gpmsd) - (\nabla_{\xi} \qpmo)(\gpmsd)  \cdot(\nabla_x \ppm)(\gpmsd)\\
        &\quad+ (\p_t \qpmo)(\gpmsd) (\p_{\tau} \ppm)(\gpmsd) - (\p_{\tau} \qpmo)(\gpmsd)(\p_t \ppm)(\gpmsd)\\
        &= \Hpm \qpmo |_{\gpmsd},
    \end{align}
where the 1-homogeneity of $\bpm$, and thus of $\ppm$, allowed us to pull the factor $|\bpmxs|^{-1}$ inside of $\nabla_x \ppm(\gpms)$ in the penultimate equality. Similarly $\nabla_{\xi} \ppm$ is $0$-homogeneous so $(\nabla_{\xi} \ppm)(\gpms) = (\nabla_{\xi} \ppm)(\gpmsd)$.

So then
\begin{align}
    \Hpm \qpm |_{(z,\zeta)} + \acpm |_{(z,\zeta)} &= \Hpm \qpmo |_{\left(z, \frac{\zeta}{|\bpm(x,\xi)|}\right)} + \acpmo |_{\left(z, \frac{\zeta}{|\bpm(x,\xi)|}\right)} \\
    &\geq \rpmo |_{\left(z, \frac{\zeta}{|\bpm(x,\xi)|}\right)} = \rpm|_{(z,\zeta)}.
\end{align}
Now note
\begin{equation}
    \left\{\rpm \geq \frac{\Cm_2}{4}\right\} = \left\{\rpmo \circ \Phipm \geq \frac{\Cm_2}{4}\right\} \supseteq  (\Phipm)^{-1} (\homogsemitrappedescapeset)=\semitrappedescapeset \supset \pmsemitrappedset,
\end{equation}
so indeed $\Hpm \qpm + \acpm \geq \frac{\Cm_2}{4} \mathbbm{1}_{\semitrappedescapeset}$.

To prove the symbol estimates, consider 
\begin{equation}\label{eq:symbolestEscape1}
    D_z^{\beta} D_{\zeta}^{\alpha}\qpm(z,\zeta) = D_z^{\beta} D_{\zeta}^{\alpha} \left( \qpmo \left( z, \frac{\zeta}{b(x,\xi)} \right) \right).
\end{equation}
By the chain rule, every differentiation in $z$ produces a term growing at most like $\frac{\zeta \bpm_x(x,\xi)}{\bpm(x,\xi)^2}$ which is bounded by a constant, since $\bpm(x,\xi) \simeq \bpm_x(x,\xi) \simeq \zeta$ on $\supp \qpmo \subset \Phipm(\Trfo)$. Similarly every differentiation in $\zeta$ produces a term growing at most like 
\begin{equation}\label{eq:symbolestEscape2}
    \frac{|\bpm(x,\xi)| + |\zeta \bpm_{\zeta}(x,\xi)|}{|\bpm(x,\xi)|^2} \leq \frac{C}{|\zeta|,}
\end{equation}
where the bound follows since $\bpm(x,\xi) \simeq \zeta$ and $|\bpm_{\zeta}(x,\xi)| \leq C$ on $\supp \qpmo \subset \Phipm(\Trfo)$.

Now note that by the construction of $\qpmo$ in terms of the $q_{\omega_j}^{\pm}$ and the uniform control of the derivatives of the $q_{\opmj}$
\begin{equation}
    \left|(D_z^{\beta} D_{\zeta}^{\alpha} \qpmo) \left(z, \frac{\zeta}{\bpm(x,\xi)}\right) \right| \leq C_{\alpha\beta}.
\end{equation}
Combining this with \eqref{eq:symbolestEscape1} and \eqref{eq:symbolestEscape2}, $\qpm$ satisfies the desired symbol estimates. An analogous argument applies to $\rpm$ and $\acpm$.
\end{proof}

\subsection{Non-trapping escape function construction}\label{s:nontrapescape}
In this section, we construct an escape function $\qpm$, such that $\Hpm \qpm$ is bounded from below near the initial data for non-trapped trajectories in $\{|x|\leq R\}$ and everywhere on $\{|x| \geq R\}$.  
Note that because our set $\semitrappedescapeset$ is a product $\Rb_t \times \Pi_t^{\bot} \semitrappedescapeset$ the time-dependence of $a$ does not influence our construction beyond working in $\Trf$ rather than $T^* \Rb^3$. Thus we follow the approach of \cite[Lemma 2.16]{Kofroth23}  and \cite[Lemma 4.1]{MST20}.

As a preliminary, when we write $|x| \simeq 2^j$ in this section, we mean $2^{j-1} \leq |x| \leq 2^{j+1}$. Now, we recall \cite[Prop 2.4]{Kofroth23}, see also \cite[Section 2]{MT12}. 
\begin{proposition}\label{p:slowvaryf}
    Let $\sigma>0$. Recall $c_j$ from Definition \ref{d:asymflat2}. Then, there exists $f \in C^{\infty}(\Rb)$ and $c_{\sigma}, C_{\sigma}>0$, such that  $c_{\sigma} \leq f(r)\leq C_{\sigma}$, when $r>R_0$. Furthermore, when $r \simeq 2^j > R_0$.
    \begin{equation}
        \frac{\sigma}{2} c_j 2^{-j}f(r) \leq f'(r)  \leq 8\sigma c_j 2^{-j} f(r).
    \end{equation}
\end{proposition}

Now we proceed with our escape function construction.
\begin{lemma}\label{l:nontrapescape}
    For all $R > R_0$ large enough, there exists $W^{\pm} \subset \plusminusnontrappedset, \qpm \in S^0(\Trfo),$ and $C_W \geq 0$ so that 
    \begin{enumerate}
        \item $\semitrappedescapeset \cup W^{\pm} = \characteristicsetofPplusminus$.
        \item For $j\gg1$, 
        \begin{equation}
            \Hpm \qpm \geq C_W c_j 2^{-j} \mathbbm{1}_{W^{\pm}} \text{ on } 2^{j-1} \leq |x| \leq 2^{j+1}.
        \end{equation}
    \end{enumerate}
    Furthermore $\qpm = \e \qpmin + \qpmout$, where $\qpmin = \qpmint \circ \Phipm$ with $\qpmin \in C^{\infty}(\Trfo)$ supported in $\{|x|\leq 4R\},$ $\qpmout \in S^0(\Trfo)$, and $\e>0$ is sufficiently small.
\end{lemma}
\begin{proof}
    Recall, from \eqref{eq:xibhomog} there exists $c^{\pm}, C^{\pm} >0$ such that 
    \begin{equation}
        c^{\pm} |\bpm(x,\xi)| \leq |\xi| \leq C^{\pm}|\bpm(x,\xi)|,
    \end{equation}
    on $\Trfo$. Now choose $\psipm \in \Ci(\Trfo)$ such that 
    \begin{align}
        \supp\psipm \subset \plusminusnontrappedset \cap \{|x| \leq R\} \cap \left\{\frac{\cpm}{2} <|\xi|, |\tau| < \Cpm+1\right\}, \\
        \psipm \equiv 1 \text{ on } \mathring{U}_R^{\pm} :=\bigg(\plusminusnontrappedset \cap \{|x|\leq R\} \cap \Phipm(\Trfo) \bigg) \backslash \homogsemitrappedescapeset,
    \end{align}
    where we recall $\homogsemitrappedescapeset=\Phipm(\semitrappedescapeset)$. Note that we can construct such a $\psipm$ because by Proposition \ref{prop:partition}(2), $\plusminusnontrappedset$ is open and $\mathring{U}_R^{\pm}$ is a compact subset of $\plusminusnontrappedset$. Recall by the proof of Lemma \ref{l:fulltrappedescapelemma}, $\homogsemitrappedescapeset = \Rb_t \times \Pi_t^{\bot} \homogsemitrappedescapeset$. Note further that because the space-time is stationary, $\plusminusnontrappedset = \Rb_t \times \Pi_t^{\bot} \plusminusnontrappedset$. Therefore this $\psipm$ can be defined so that $\psipm(t, x, \tau, \xi)=\psipm(x,\xi,\tau)$, that is it does not depend on $t$.


    Now 
    define
    \begin{equation}
        \qpmint(t, x, \tau, \xi) = -\chi_{<2R}(|x|) \int_0^{\infty} \psipm \circ \gpms(t, x, \tau, \xi) ds.
    \end{equation}
    Note that because $\psipm \circ \gpms$ does not depend on $t$, neither does $\qpmint$.


    Now we claim there exists $T'<\infty$ such that all null-bicharacteristics spend at most $T'$ within $\supp \psipm$. 
    To see this, let $\Ypm$ be an open neighborhood of $\supp \psipm$ such that $\overline{\Ypm} \subset \plusminusnontrappedset$. Such a $\Ypm$ exists because $\plusminusnontrappedset$ is open and $\supp \psipm \subset \plusminusnontrappedset$ is closed. 
    Apply Proposition \ref{prop:partition}(3) with  $K=\overline{\Ypm}$ and let $T'$ be as in that result. 
    Now we will show that for all $\omega \in \Trfo$ there exists a neighborhood of $\omega,$ $U_{\omega}^{\pm}$, and $s_{\omega}^{\pm} \geq  0$ such that $\psipm \circ \gpms(z)=0$ for all $z \in U_{\omega}^{\pm}$ and $s  \in [0,s_{\omega}^{\pm}]\cup[s_{\omega}^{\pm}+T',\infty)$. 
    There are three cases to consider 
    \begin{enumerate}
        \item For $\omega \in \supp(\psipm) \subset \Ypm$, this is satisfied with $s_{\omega}^{\pm}=0$ and $U_{\omega}^{\pm}=Y^{\pm}$ by Proposition  \ref{prop:partition}(3).
        \item For $\omega \not \in \bigcup_{s \in [0,\infty)} \gpms(\supp \psipm)=:X^{\pm}$, note that $X^{\pm}$ is closed by Proposition  \ref{prop:partition}(3). Thus, there exists a neighborhood $U_{\omega}^{\pm}$ of $\omega$ such that $X^{\pm} \cap U_{\omega}^{\pm} =\emptyset.$ Thus, for each $z \in U_{\omega}^{\pm}$ we have $\gpms(z) \not \in \supp \psipm$ for all $s \in [0, \infty).$ That is, $\psipm \circ \gpms(z)=0$ for $s \geq 0$.
        \item If $\omega \in X^{\pm} \backslash \supp(\psipm)$, then $\gopmr{s'} \in \supp(\psipm)$ for some $s'>0$. Then by continuity of the flow, there exists $s_{\omega}>0$ such that $\gopmr{s_{\omega}} \in \Ypm$ and $\gpms(\omega) \not \in \supp(\psipm)$ for all $s \in [0,s_{\omega}].$ 
        By continuity of the flow in its initial data, we can extend the previous sentence to: there exists a neighborhood $U_{\omega}^{\pm} \ni \omega$ such that for all $z \in U_{\omega}^{\pm}$, $\varphi^{\pm}_{s_{\omega}^{\pm}}(z) \in \Ypm$ and $\psipm \circ \gpms(z) =0$ for all $s \in [0, s_{\omega}^{\pm}].$ Then by Proposition  \ref{prop:partition}(3), $\psipm \circ \gpms(z)=0$ for all $z \in U_{\omega}^{\pm}$ and $s \in [0,s_{\omega}^{\pm}] \cup [s_{\omega}^{\pm}+T', \infty).$
    \end{enumerate}
Thus, the integrand present in $\qpmint$ is non-zero for $s$ in an interval of maximal length $T'$. Thus, the function is well-defined and differentiation under the integral sign is not problematic. Combining this with regularity of the flow, we have $\qpmint \in \Ci(\Trfo)$. Furthermore, because of the $\chi_{<2R}(x)$, $\qpmint$ is compactly supported in $\{|x|\leq 4R\}$, and by compactness is bounded in all derivatives on $\{|x| \leq 4R\} \cap \Phipm(\Trfo).$

Let $\qpmin=\qpmint \circ\Phipm$ be defined on $\Trfo.$ As in the construction of the trapped escape function, in the proof of Lemma \ref{l:fulltrappedescapelemma}
\begin{equation}
    \Hpm \qpmin |_{(x,\xi)} = \Hpm \qpmint |_{\Phipm(x,\xi)}.
\end{equation}
Now note that $\qpmint$ does not depend on $t$, so $\p
_{\tau} \ppm \p_t \qpmint=0$, we have 
\begin{align}
    \Hpm \qpmint =& \p_{\tau} \ppm \p_t \qpmint - \p_t \ppm \p_{\tau} \qpmint + \nabla_{\xi} \ppm \nabla_x \qpmint - \nabla_x \ppm \nabla_{\xi} \qpmint \\
    =&0 - \chi\left(\frac{|x|}{2R}\right) \Hpm \int_0^{\infty} \psipm \circ \gpms\left(t,x,\tau, \xi \right) ds \\
    &+ \left( \nabla_{\xi} \bpm(x,\xi) \cdot \nabla_x \chi \left( \frac{|x|}{2R} \right)\right)\int_0^{\infty} \psipm \circ \gpms(t,x,\tau, \xi) ds.
\end{align}
Now evaluating on the image of $\Phipm$, using that $\Hpm \psipm\circ \gpms = \p_s (\psipm \circ \gpms)$, $\varphi_0^{\pm}(t, x, \tau, \xi)=(t, x, \tau, \xi)$, and that all null bicharacteristics exit $\supp \psipm$ in finite time, we have 
\begin{align}
    \Hpm& \qpmint |_{\Phipm(x,\xi)} = \chi_{<2R}(|x|) \psipm\left(t, x,  \frac{\tau}{|\bpmx|}, \frac{\xi}{|\bpmx|} \right)\\
    &+\frac{1}{2R} \nabla_{\xi} \bpm\left(x,\frac{\xi}{|\bpmx|}\right) \cdot \frac{x}{|x|} \chi'\left(\frac{|x|}{2R}\right) \int_0^{\infty} \psipm \circ \gpms\left(t,x, \frac{\tau}{|\bpmx|}, \frac{\xi}{|\bpmx|}\right) ds.
\end{align}
The first term is non-negative, supported in $\plusminusnontrappedset \cap \{|x| \leq R\}$, and is equal to $1$ on $U^{\pm}_R:=\Phi^{-1}(\mathring{U}_R^{\pm})$. The second term is an error term supported in $\{2R \leq |x| \leq 4R\}$ and will be absorbed by $\qpmout$. 

To define $\qpmout$, let $f$ be the function from Proposition \ref{p:slowvaryf} and define
\begin{equation}
    \qpmout(t, x, \tau, \xi) = \qpmout(x,\xi) = -\chi_{>R}(|x|) f(|x|) \nabla_{\xi}\bpmx \cdot \frac{x}{|x|}.
\end{equation}
Note that $\qpmout$ is smooth, bounded in all $x$ derivatives by asymptotic flatness, and homogeneous in $\xi$ and $\tau$ of degree $0$, so $\qpmout \in S^0(\Trfo).$

Now noting that $\qpmout$ does not depend on $t$ and $\tau$, and recalling $\ppm = \tau - \bpmx$ we have 
\begin{align}
    \Hpm \qpmout =& \nabla_{\xi} \ppm \nabla_x \qpmout - \nabla_x \ppm \nabla_{\xi} \qpmout \\
    =&\nabla_{\xi}\bpm \cdot \nabla_x\left(\chi_{>R}(|x|) f(|x|) \nabla_{\xi} \bpm \cdot \frac{x}{|x|}\right)\\
    &-\nabla_x \bpm \cdot \nabla_{\xi} \left(\nabla_{\xi} \bpm \cdot \frac{x}{|x|}\right) \chi_{>R}(|x|) f(|x|).
\end{align}
Evaluating the terms on the right hand side and recalling the definition of $\chi_{>R}$
\begin{align}
    \Hpm \qpmout=&-\nabla_{\xi} \bpm \cdot \frac{x}{R|x|} \chi'\left(\frac{|x|}{R}\right) f(|x|) \nabla_{\xi}\bpm\cdot \frac{x}{|x|} \\
    \label{eq:hpoutMain} &+ \nabla_{\xi}\bpm \cdot \frac{x}{|x|} f'(|x|) \nabla_{\xi} \bpm \cdot \frac{x}{|x|} \chi_{>R}(|x|) \\
    &+\chi_{>R}(|x|) f(|x|) \p_{\xi_k} \bpm \left(\d_{kl} - \frac{x_k x_l}{|x|^2}\right) \p_{\xi_l} b  \frac{1}{|x|}\\
    &+\left(-\nabla_x \bpm \cdot \nabla_{\xi}(\nabla_{\xi} \bpm \cdot \frac{x}{|x|}) + \nabla_{\xi} \bpm \cdot \nabla_x \nabla_{\xi} \bpm \cdot \frac{x}{|x|}\right) \chi_{>R}(|x|) f(|x|).
\end{align}
Note the last term 
\begin{align}
    &\left(-\nabla_x \bpm \cdot \nabla_{\xi}(\nabla_{\xi} \bpm \cdot \frac{x}{|x|}) + \nabla_{\xi} \bpm \cdot \nabla_x \nabla_{\xi} \bpm \cdot \frac{x}{|x|}\right) \chi_{>R}(|x|) f(|x|)
    =O(\<x\> |\p g| )\chi_{>R}(|x|) |x|^{-1}.
\end{align}
which is small for $|x|>R$ by the definition of $b$ in terms of $g$, \eqref{eq:bdef}, and asymptotic flatness, Definition \ref{d:asymptoticFlat1}, and is localized to that region by $\chi_{>R}$. The first term on the right hand side of \eqref{eq:hpoutMain} is non-negative because $\chi'\leq0$. So to obtain a lower bound on $\Hpm \qpmout$ it is sufficient to obtain a lower bound on 
\begin{equation}
    \left( \nabla_{\xi} \bpm \cdot \frac{x}{|x|}\right)^2 f'(|x|) \chi_{>R}(|x|) + \chi_{>R}(|x|) f(|x|) \p_{\xi_k} \bpm \left(\d_{kl} - \frac{x_k x_l}{|x|^2}\right) \p_{\xi_l} b  \frac{1}{|x|}.
\end{equation}
For $|x| \simeq 2^j$, $f'(|x|) \geq \frac{\sigma}{2} c_j 2^{-j} f(|x|)$ by Proposition \ref{p:slowvaryf}. Using this and re-writing the Einstein notation, we have on $|x| \simeq 2^j$
\begin{align}\label{eq:hpmoutpreCauchy}
    \Hpm \qpmout &\geq \chi_{>R}(|x|) f(|x|) \left( \frac{\sigma}{2} c_j 2^{-j} \frac{|x \cdot \nabla_{\xi} \bpm|^2}{|x|^2} + \frac{1}{|x|} \left(|\nabla_{\xi}\bpm|^2 - \frac{|x\cdot \nabla_{\xi}\bpm|^2}{|x|^2} \right) \right).
\end{align}
Choosing $\sigma \geq 8$, we can rewrite the terms in parentheses as 
\begin{align}
    &\frac{\sigma}{2} c_j 2^{-j} \frac{|x \cdot \nabla_{\xi} \bpm|^2}{|x|^2} + \frac{1}{|x|} \left(|\nabla_{\xi}\bpm|^2 - \frac{|x\cdot \nabla_{\xi}\bpm|^2}{|x|^2} \right) \\
    &= \left(\frac{\sigma}{2} c_j 2^{-j}-\frac{2c_j}{|x|}\right) \frac{|x \cdot \nabla_{\xi}\bpm|^2}{|x|^2}+ \frac{1-2c_j}{|x|} \left(|\nabla_{\xi} \bpm|^2 - \frac{|x\cdot \nabla_{\xi} \bpm|^2}{|x|^2}\right) + \frac{2c_j}{|x|}|\nabla_{\xi}\bpm|^2,
\end{align}
since $2^{j-1} \leq |x|$ and $\sigma \geq 8$, the first term is non-negative. Applying Cauchy-Schwarz to see $\frac{|x\cdot \nabla_{\xi} \bpm|}{|x|^2} \leq |\nabla_{\xi} \bpm|^2$, and since $c_j<\frac 12$, for $|x|>R$ when $R$ is taken large enough, we see the second term is non-negative. Finally, since $|x| \leq 2^{j+1}$ the final term is bounded from below by $c_j 2^{-j}|\nabla_{\xi} \bpm|^2$.

Plugging this back into \eqref{eq:hpmoutpreCauchy} we have
\begin{align}
    \Hpm \qpmout 
    &\geq \chi_{>R}(|x|) f(|x|) c_j2^{-j} |\nabla_{\xi} \bpm|^2.
\end{align}
Now, by asymptotic flatness, $|\nabla_{\xi} \bpm|^2 \simeq 1$, and by Proposition \ref{p:slowvaryf}, $f \simeq 1$, so for $|x| \geq R$ and for some $C>0$,
\begin{equation}
    \Hpm \qpmout \geq C c_j 2^{-j} \chi_{>R}(x), \quad \text{ on } |x| \simeq 2^j.
\end{equation}
In particular, $\Hpm \qpmout$ is non-negative, and strictly positive for $|x| \geq R.$

Since the error term from $\Hpm \qpmin$ is bounded and supported in $\{2R \leq |x|\leq 4R\}$ and $\Hpm \qpmout$ is strictly positive on the support of this error, we can choose $\e>0$ small enough, so that 
\begin{equation}
    \qpm:=\e \qpmin + \qpmout \in \Ci(\Trfo),
\end{equation}
has 
\begin{equation}
    \Hpm \qpm >C c_j 2^{-j} \chi_{>R} \text{ for } |x| \simeq 2^j \text{ in }W^{\pm}:= U^{\pm}_R \cup \{(x,t,\xi, \tau) \in \characteristicsetofPplusminus : |x| > R\}.
\end{equation}
Now note computing directly and applying Proposition \ref{prop:partition} 
\begin{align}
    \semitrappedescapeset \cup U_R^{\pm} &= \semitrappedescapeset \cup ( (\plusminusnontrappedset \cap \{|x|\leq R\})\backslash \semitrappedescapeset) \\
    &=(\pmsemitrappedset \cup \plusminusnontrappedset) \cap\{|x|\leq R\} \\
    &=\characteristicsetofPplusminus \cap \{|x|\leq R\}.
\end{align}
Therefore $\semitrappedescapeset \cup W^{\pm} = \characteristicsetofPplusminus$.

We have $\Hpm \qpm = 1$ for $(t, x, \tau, \xi) \in U_R^{\pm}$ and 
\begin{equation}
    \Hpm \qpm \geq C c_j 2^{-j} \chi_{>R} \quad \text{ for }|x| \simeq 2^j, |x|>R.
\end{equation}
By compactness of $\{|x|\leq R\}$, there exists some $C_W>0$ such that 
\begin{equation}
    \Hpm \qpm \geq C_W c_j 2^{-j} \mathbbm{1}_{W^{\pm}} \quad \text{ for } |x| \simeq 2^j.
\end{equation}
\end{proof}

\subsection{Combination of escape function constructions}
In this section, we combine the escape functions constructed in Sections \ref{s:semitrapescape} and \ref{s:nontrapescape} to obtain an escape function $q$ such that 
\begin{equation}
    H_p q + 2 \kappa \tau a q \geq C,
\end{equation}
on $\characteristicsetofP$. We then choose $m \in S^0$ so that adding $mp$ to the left hand side gives positivity away from $\characteristicsetofP$. 

We roughly follow the approach of \cite[Lemma 2.4]{Kofroth23} and \cite[Lemma 4.1]{MST20}. However our argument necessarily differs from both of these references. Comparing to \cite{Kofroth23}, we must work on $\Trf$ rather than $T^* \Rb^3$ to account for the time dependence of our damping. Comparing to \cite{MST20}, our space-time is trapping and so we have two separate escape functions which we must combine before choosing $m$. We have also made expositional changes such as splitting this step into two lemmas. 

We first combine our semi-trapped and non-trapped escape functions and show that the combination is bounded from below on the characteristic set of $P$. It is at this step in our escape function construction that we convert back from the half-wave decomposition of $\ppm$ to the full wave operator $P$. 
\begin{lemma}\label{l:combineEscapeFunc}
Fix $0<\d\ll1$ from Definition \ref{d:asymflat2}. There exists $\kappa \geq 1, C>0$ and symbols $\ti{q}_j \in S^j(\Trf)$  supported in $|\xi| \geq 1$, $|\tau| \geq 1$, such that for $q=\tau \ti{q}_0+\ti{q}_1$ and for $(t,x,\tau, \xi) \in \characteristicsetofP$
    \begin{equation}
        (H_p q +2 \kappa \tau a q)(t,x,\tau,\xi) \geq C \mathbbm{1}_{|\xi|\geq 1} \mathbbm{1}_{|\tau| \geq 1} \<x\>^{-1-\d}(\tau^2+|\xi|^2).
    \end{equation}
\end{lemma}

\begin{proof}
      We will first define $q$, and then compute $H_p q + 2\kappa a q$ on $\characteristicsetofP =\{\tau = b^+\} \cup \{\tau = b^-\}$.
    
    Let $\qpm_1$ be the escape function defined in Lemma \ref{l:fulltrappedescapelemma} and let $\qpm_2$ be the escape function defined in Lemma \ref{l:nontrapescape}. Then define
    \begin{equation}
        \qpmlg = \exp(-\sigma(\qpmo+\qpmt)) \chigl(|b^{\pm}|) \chigl(|\tau|).
    \end{equation}
    By its construction $\qpmlg \in S^0(\Trf)$. Further define
    \begin{equation}
        q= p^+ \qmlg + p^- \qplg = (\tau - b^+) \qmlg + (\tau-b^-)\qplg.
    \end{equation}
    This $q$ will be our escape function. Note that indeed $q= \tau \ti{q_0} + \ti{q_1}$, $\ti{q_j} \in S^j (\Trf)$, where
    \begin{align}
        \tau \ti{q_0} &= \tau(\qplg + \qmlg) \\
        \ti{q}_1 &= -b^+ \qmlg - b^-\qplg.
    \end{align}

    Now to compute $H_p q + 2\kappa \tau a q$, first note that by definition of $q$
    \begin{equation}\label{eq:escapecombine2}
        (H_p q + 2\kappa \tau a q) |_{\tau =\bpm} = H_p q|_{\tau =\bpm} + 2\kappa \bpm (\bpm-\bmp) a \qpmlg.
    \end{equation}
    To compute $H_p q$, first recall $p=g^{00} p^+ p^-, \ppm = \tau- \bpm$. By the product rule
    \begin{equation}
        H_p q = g^{00}p^- H_{p^+} q  + g^{00} p^+ H_{p^-}q + p^+ p^- H_{g^{00}} q.
    \end{equation}
    Then since $\ppm|_{\tau=\bpm}=0$, we have  
    \begin{equation}\label{eq:escapecombinehpq}
        H_p q |_{\tau =\bpm}=g^{00}(p^{\mp} \Hpm q)|_{\tau =\bpm}.
    \end{equation}
    Now we write $\Hpm q$ in terms of $\qpmlg$, noting that $\Hpm \ppm =0$
    \begin{align}\label{eq:escapecombinehpcomp}
        \Hpm q &= \Hpm (p^+ \qmlg + p^- \qplg) \\
        &= \ppm \Hpm \qmplg + p^{\mp} \Hpm \qpmlg + \qpmlg \Hpm p^{\mp}.
    \end{align}
    Note the first term on the right hand side evaluates to $0$ at $\tau = \bpm$. We compute the third term on the right hand side directly 
    \begin{align}
        \Hpm p^{\mp}|_{\tau=\bpm} &= \nabla_{\zeta} (\tau - \bpm)\cdot\nabla_z (\tau -\bmp) - \nabla_{\zeta}(\tau-\bmp)\cdot\nabla_z(\tau-\bpm)\\
        &=\nabla_{\xi} \bpm \cdot\nabla_x b^{\mp} - \nabla_{\xi} b^{\mp} \cdot\nabla_x \bpm.
    \end{align}
    Combining this with \eqref{eq:escapecombinehpq} and \eqref{eq:escapecombinehpcomp} and using $p^{\mp}|_{\tau=\bpm}= \bpm - b^{\mp}$,
    \begin{align}
        H_p q |_{\tau = \bpm} &= g^{00} p^{\mp} \left( 0 + p^{\mp} \Hpm \qpmlg + \qpmlg \Hpm p^{\mp} \right)\bigg|_{\tau = \bpm} \\
        &=g^{00} \bigg((b^+ - b^-)^2 \Hpm \qpmlg + (\bpm - \bmp) \qpmlg(\nabla_{\xi} \bpm \nabla_x \bmp - \nabla_x \bpm \nabla_{\xi} \bmp)\bigg) \bigg|_{\tau = \bpm}. \label{eq:escapecombinehpintermediate}
    \end{align}
    Now we compute $\Hpm \qpmlg$ in terms of $\qpmo$ and $\qpmt$ 
    \begin{align}
        \Hpm \qpmlg &= \Hpm \left( e^{-\sigma (\qpmo +\qpmt)} \chi_{> 1}(|\bpm|) \chigl (|\tau|)\right)\\
        &=-\sigma \qpmlg\Hpm(\qpmo+\qpmt)  + e^{-\sigma(\qpmo+\qpmt)} \Hpm\left( \chigl(|\bpm|) \chigl(|\tau|) \right) \\
        &= -\sigma \qpmlg \Hpm(\qpmo+\qpmt)  + 0, 
    \end{align}
    where the last equality follows because the cutoffs and $\ppm$ are functions of $\bpm$ and $\tau$.
    Combining this with \eqref{eq:escapecombinehpintermediate} 
    \begin{align}
        H_p q |_{\tau=\bpm}
        &= -g^{00} \bigg(\sigma(b^+ - b^-)^2 \qpmlg \Hpm(\qpmo + \qpmt) -(\bpm - \bmp) \qpmlg (\nabla_{\xi} \bpm \nabla_x \bmp - \nabla_x \bpm \nabla_{\xi} \bmp)\bigg) \bigg|_{\tau =\bpm}.\label{eq:escapecombine1}
    \end{align}

    Recalling the definition of $b$ in terms of $g$ from \eqref{eq:bdef}, and the asymptotic flatness of $g$, Definition \ref{d:asymflat2}, there exist $c_j \simeq 2^{-\d j}$, such that  
    \begin{equation}\label{eq:escapecombine3}
        |(\bpm - \bmp)  (\nabla_{\xi} \bpm \nabla_x \bmp - \nabla_x \bpm \nabla_{\xi} \bmp) | \lesssim c_j |\xi|^2 \<x\>^{-1} \simeq c_j 2^{-j} |\xi|^2 \text{ on } \<x\> \simeq 2^j.
    \end{equation}
    Applying $g^{00} \geq -C$, and combining \eqref{eq:escapecombine2}, \eqref{eq:escapecombine1}, and \eqref{eq:escapecombine3}: there exists $C_1 >1 $ such that on $\<x\> \simeq 2^j$ 
    \begin{align}
        (H_p q + 2\kappa \tau a q)_{\tau=\bpm}\geq  C \qpmlg \bigg(&C_1^{-1} \sigma |\xi|^2  \Hpm (\qpmo +\qpmt) -C_1 c_j|\xi|^2 2^{-j}+2\kappa \bpm(\bpm-\bmp)a \bigg) .
    \end{align}
    Now using that $\frac{\bpm}{\bpm-\bmp}\simeq 1$ and $|\bpm-\bpm|^2\simeq |\xi|^2$
    \begin{equation}
        (H_pq+2\kappa \tau a q)|_{\tau=\bpm} \geq C |\xi|^2 \qpmlg \left(C_1^{-1} \sigma\left(\Hpm\qpmo+\frac{2\kappa}{\sigma}a\right) + C_1^{-1} \sigma \Hpm \qpmt - C_1 c_j 2^{-j}\right) \text{ on } \<x\> \simeq 2^j.
    \end{equation}
     By Lemma \ref{l:fulltrappedescapelemma} and Lemma \ref{l:nontrapescape}, choosing $\kappa = \frac{C_a \sigma}{2}$
    \begin{align}
        (H_p q +2\kappa \tau a q)|_{\tau=\bpm} \geq &|\xi|^2 \qpmlg \left( C_1^{-1} \frac{\Cm_2}{2} \sigma \mathbbm{1}_{\semitrappedescapeset} + C_1^{-1} C_W \sigma c_j 2^{-j}  \mathbbm{1}_{W^{\pm}} - C_1 c_j 2^{-j}\right) \text{ on }\<x\> \simeq 2^j.
    \end{align}
    So for $\sigma$ large enough, since $\semitrappedescapeset \cup W^{\pm} \supset \characteristicsetofPplusminus = \{\tau = \bpm\}$, on $\<x\> \simeq 2^j$
    \begin{align}
        (H_p q+2\kappa \tau a q)|_{\tau = \bpm} &\geq C |\xi|^2 \qpmlg\left( \frac{\Cm_2 \sigma}{4} \mathbbm{1}_{\semitrappedescapeset} + \frac{C_W \sigma}{2} c_j 2^{-j} \mathbbm{1}_{W^{\pm}} \right)\\
        &\geq C |\xi|^2 \chigl(|\bpm|) \chigl(|\tau|) \exp(-\sigma(\qpmo+\qpmt)) \left( \mathbbm{1}_{\semitrappedescapeset} + c_j 2^{-j} \mathbbm{1}_{W^{\pm}} \right).
        \end{align}
        Now note that since $\qpmo,\qpmt \in S^0(\Trf)$, in particular they are bounded, then $\exp(-\sigma(\qpmo+\qpmt)) \geq C>0$. This, along with  $\semitrappedescapeset \cup W^{\pm} \supset \characteristicsetofPplusminus = \{\tau=\bpm\}$ gives on $\<x\> \simeq 2^j$
        \begin{align}
        (H_p q+2\kappa \tau a q)|_{\tau = \bpm} &\geq C |\xi|^2 \chigl(|\bpm|) \chigl(|\tau|) \left( \mathbbm{1}_{\semitrappedescapeset} + c_j 2^{-j} \mathbbm{1}_{W^{\pm}} \right) \\
        &\geq C|\xi|^2 \mathbbm{1}_{|\xi|>1} \mathbbm{1}_{|\tau|> 1} c_j 2^{-j}.
        \end{align}
        Now since $c_j \geq 2^{-\d j}$, and $\<x\> \simeq 2^j$, we have 
        \begin{align}
        (H_p q+2\kappa \tau a q)|_{\tau = \bpm}  &\geq C |\xi|^2 \mathbbm{1}_{|\xi|>1} \mathbbm{1}_{|\tau| >1} \<x\>^{-1-\d},\label{eq:hpdesiredinequality}
    \end{align}
    Since $\characteristicsetofP=\{\tau = \bpm\}$ and thus $|\tau|^2=|\bpm|^2 \simeq |\xi|^2$, this gives the desired bound on the characteristic set. 
\end{proof}

We now construct an elliptic correction term to ensure positivity away from $\characteristicsetofP$ and conclude the proof of Proposition \ref{p:EscapeFunction}, which we restate here for the convenience of the reader.
\setcounter{definition}{0}
\begin{proposition}
\escapeFunctionStatement
\end{proposition}

\begin{proof}[Proof of Proposition \ref{p:EscapeFunction}]
    Let $q$ and $\kappa \geq 1$ be as in Lemma \ref{l:combineEscapeFunc}. To extend the bound from $\characteristicsetofP$ to all of $\{|\xi| \geq 1, |\tau| \geq 1\}$, we construct a correction term on the elliptic set of $p$. 
    
    Since $q \in S^1(\Trf)$ and $p \in S^2(\Trf)$, we can write 
    \begin{equation}
        H_p q(t, x, \tau, \xi) + 2\kappa a(t,x) \tau q(t, x, \tau, \xi) = E(t, x, \tau, \xi) \tau^2 + F(t, x, \tau, \xi)\tau +G(t, x, \tau, \xi),
    \end{equation}
    where $E \in S^0(\Trf), F=F_1 F_2$, $F_1 \in S^0(\Trf), F_2 \in S^1(T^*\Rb^3)$ and $G=G_1 G_2$ with $G_1 \in S^0(\Trf), G_2 \in S^2(T^* \Rb^3)$. In Lemma \ref{l:combineEscapeFunc} we have shown that 
    \begin{equation}\label{eq:Hppolypositive}
        \left(E\tau^2 + F \tau + G \right)|_{\tau=\bpmx} \geq C|\xi|^2 \mathbbm{1}_{|\xi| \geq 1}\mathbbm{1}_{|\tau|\geq 1}\<x\>^{-1-\d} .
    \end{equation}
    Recalling that $p=g^{00} (\tau-b^+)(\tau-b^-)$, for an $m:=-\frac{\ti m}{g^{00}} \in S^0(\Trf)$ to be determined, we can rewrite
    \begin{equation}\label{eq:polyplusm}
        E \tau^2 + F \tau + G \tau + p m = (E- \ti m) \tau^2 + (F+(b^+ +b^-) \ti m) \tau + (G-b^+ b^- \ti m).
    \end{equation}
    This is quadratic in $\tau$, so if the second order coefficient is positive
    \begin{equation}
        E-\ti m>0 \text{ on } |\xi|, |\tau| \geq 1,
    \end{equation}
    and the discriminant is negative
    \begin{equation}
        \mathcal{P}(\ti m):=(F+(b^+ + b^-) \ti m)^2 - 4(E-\ti m)(G-b^+b^- \ti m) < 0 \text{ on } |\xi|, |\tau| \geq 1,
    \end{equation}
    then we will have 
    \begin{equation}
        E \tau^2 + F \tau + G + pm >0 \text{ on } |\xi|,|\tau| \geq 1.
    \end{equation}
    We now will define an $\ti m$, and show that it satisfies these inequalities. We can rewrite the discriminant 
    \begin{equation}
        \mathcal{P}(\ti m)=(b^+-b^-)^2 \ti m^2 + (2F(b^++b^-) + 4 Eb^+ b^- + 4G) \ti m +(F^2-4EG),
    \end{equation}
    as a quadratic polynomial in $\ti m$. Note that its second-order coefficient is positive. Thus the minimum value of $\mathcal{P}(\ti m)$ is attained at its vertex, which is exactly
    \begin{equation}
        \ti m := - \frac{F(b^+ + b^-)+2Eb^+ b^- + 2G}{(b^+-b^-)^2}.
    \end{equation}
    Note that on $|\tau|, |\xi| \geq 1$, $\ti m \in S^0(\Trf)$ because $E \in S^0, F=F_1 F_2 \in S^1, G=G_1G_2 \in S^2, b^{\pm} \in S^1$ and $(b^+ - b^-)^{2} \gtrsim |\xi|^2$. Furthermore, since $g^{00} \geq -C$ and is asymptotically flat, we have $m \in S^0(\Trf)$.

    Writing $Z:=-\big(F(b^+ +b^-) + 2(Eb^+ b^-+G)\big)$ and plugging $\ti m=(b^+-b^-)^{-2} Z$ into the original form of $\mathcal{P}(\ti m)$ gives
    \begin{align}
        \mathcal{P}(\ti m)&=(b^+ - b^-)^{-4} \bigg((F(b^+-b^-)^2 + (b^++b^-)Z)^2 \\
        &\qquad\qquad\qquad\qquad-4(E(b^+-b^-)^2-Z)(G(b^+-b^-)^2-b^+b^- Z)\bigg)\\
        &=(b^+-b^-)^{-2}\bigg(Z^2+Z(2F(b^++b^-)+4(Eb^+b^-+G)) \\
        &\qquad \qquad\qquad\qquad+(F^2-4EG)(b^+-b^-)^2\bigg)\\
        &=(b^+-b^-)^{-2}(-Z^2+(F^2-4EG)(b^+-b^-)^2)\\
        &= -4(b^+-b^-)^{-2} (E(b^+)^2+F b^+ +G)(E(b^-)^2+F b^-+G)\\
        &= -4(b^+-b^-)^{-2} \left(\left(H_p q+2\kappa \tau a q\right) |_{\tau=b^+}\right)  \left(\left(H_p q+2\kappa \tau a q\right) |_{\tau=b^-}\right)<0,\label{eq:discriminant}
    \end{align}
    where the final inequality follows from \eqref{eq:hpdesiredinequality}.
    Furthermore 
    \begin{align}
        E - \ti m &= \frac{E(b^+ - b^-)^2 +2Eb^+b^- + F(b^+ + b^-) + 2G}{(b^+ - b^-)^2} \\
        &= (b^+ - b^-)^{-2} \left( E( b^+)^2 + F(b^+) + G + E(b^-)^2 + F b^- + G\right) \\
        &= (b^+ -b^-)^{-2} \bigg( ( H_p q + 2 \kappa \tau a q)|_{\tau=b^+} + (H_p q + 2\kappa \tau a q) |_{\tau=b^-} \bigg)\\
        &\geq C \mathbbm{1}_{|\xi| \geq 1} \mathbbm{1}_{|\tau| \geq 1} \<x\>^{-1-\d} >0,
        \label{eq:aminusm}
    \end{align}
    where the final inequality follows from \eqref{eq:hpdesiredinequality} and the fact that $(b^+ - b^-)^{-2} |\xi|^2 \simeq 1.$ Since $E-\ti m>0$ and the discriminant $\mathcal{P}(\ti m)<0$, we indeed have
    \begin{equation}
        E \tau^2 + F\tau + G + pm = H_p q + 2 \kappa \tau a q + pm >0 \text{ on } |\xi|, |\tau| \geq 1.
    \end{equation}

    It remains to be seen that we have the desired growth in $\xi, \tau$, and $\<x\>$. To see this, we will consider two cases: $|\tau| \leq C_0|\xi| \<x\>^{\frac{1+\d}{2}}$ and $|\tau| \geq C_0|\xi|\<x\>^{\frac{1+\d}{2}}$, with $C_0 > 1$ to be determined. First we rewrite \eqref{eq:polyplusm} by completing the square
    \begin{equation}\label{eq:hpqfinal}
         H_p q + 2\kappa a \tau q + pm = (E-\ti m) \left( \tau +\frac{F+\ti m(b^++b^-)}{2(E-\ti m)}\right)^2 - \frac{\mathcal{P}(\ti m)}{4(E-\ti m)}.
    \end{equation}
    Now note that by \eqref{eq:discriminant}, and \eqref{eq:aminusm}, and then applying \eqref{eq:hpdesiredinequality}
    \begin{align}
         - \frac{\mathcal{P}(\ti m)}{4(E-\ti m)}&=\frac{\left(\left(H_p q+2\kappa \tau a q\right) |_{\tau=b^+}\right)  \left(\left(H_p q+2\kappa \tau a q\right) |_{\tau=b^-}\right)}{\left( ( H_p q + 2 \kappa \tau a q)|_{\tau=b^+} + (H_p q + 2\kappa \tau a q) |_{\tau=b^-} \right)}\\
         &\geq C\min\left\{\left(H_p q+2\kappa \tau a q\right)|_{\tau=b^+},  \left(H_p q+2\kappa \tau a q\right)|_{\tau=b^-}\right\}\\
         &\geq C \mathbbm{1}_{|\xi| \geq 1} \mathbbm{1}_{|\tau| \geq 1} \<x\>^{-1-\d}|\xi|^2 . \label{eq:Pmamfinal}
    \end{align}
    Now we consider the case $\tau \geq C_0|\xi| \<x\>^{\frac{1+\d}{2}}$. Note by the definition of $\mathcal{P}(m)$,
    \begin{align}
        \left| \frac{F+(b^+ + b^-)\ti m}{2(E-\ti m)}\right| &= \frac{\sqrt{(\mathcal{P}(\ti m)+4(E-\ti m)(G-b^+b^-\ti m)}}{2(E-\ti m)}\\
        &= \frac{1}{2(E-\ti m)^{1/2}}\sqrt{\frac{\mathcal{P}(\ti m)}{E-\ti m}+4(G-b^+ b^-\ti m)}.
    \end{align}
    Then applying \eqref{eq:Pmamfinal}, using that $|\bpm| \simeq|\xi|$, $G, \ti m \in S^0$, and applying \eqref{eq:aminusm}, there exists $C^* >0$ such that
    \begin{align}
        \left| \frac{F+(b^+ + b^-)\ti m}{2(E-\ti m)}\right|& \leq \frac{1}{(E-\ti m)^{1/2}} \sqrt{G-b^+ b^-\ti m}
        \leq \frac{C}{(E-\ti m)^{1/2}}|\xi| \leq C^*|\xi|\<x\>^{\frac{1+\d}{2}}. \label{eq:Fbpm}
    \end{align}
    Therefore choosing $C_0 >\max(2C^*,1)$, if we apply \eqref{eq:Fbpm} and use that $\tau \geq C_0|\xi| \<x\>^{\frac{1+\d}{2}}$, we have
    \begin{align}
        (E-\ti m)\left( \tau+ \frac{F+(b^+ + b^-)\ti m}{2(E-\ti m)} \right)^2 &\geq C \mathbbm{1}_{|\xi| \geq 1} \mathbbm{1}_{|\tau| \geq 1} \<x\>^{-1-\d} \left( \tau - C^*|\xi|\<x\>^{\frac{1+\d}{2}} \right)^2\\
        &\geq C \mathbbm{1}_{|\xi| \geq 1} \mathbbm{1}_{|\tau| \geq 1} \<x\>^{-1-\d} \tau^2.
    \end{align}
    Combining this with \eqref{eq:hpqfinal} and \eqref{eq:Pmamfinal} we obtain
    \begin{equation}
        H_p q+ 2\kappa a \tau q+pm \geq \mathbbm{1}_{|\xi|\geq 1}\mathbbm{1}_{|\tau| \geq 1} \<x\>^{-1-\d} (\tau^2 + |\xi|^2).\label{eq:finalcombine1}
    \end{equation}
    Since $\<x\>^{-1-\d} \geq \<x\>^{-2-2\d}$, this shows the desired behavior when $\tau  \geq C_0 |\xi| \<x\>^{\frac{1+\d}{2}}$.
    
    We now consider the other case $\tau \leq C_0|\xi|\<x\>^{\frac{1+\d}{2}}$. In this case we have $|\xi|^2 \geq \frac{1}{2}|\xi|^2+\frac{1}{2C_0^2}\<x\>^{-1-\d} \tau^2$. Combining this with \eqref{eq:hpqfinal}, using that 
    \begin{equation}
        (E-\ti m)\left(\tau+ \frac{F+\ti m(b^+ + b^-)}{2(E-\ti m)}\right)^2 \geq 0,
    \end{equation} and applying  \eqref{eq:Pmamfinal}, we have  
    \begin{align}
        H_p q + 2\kappa a \tau q + p m &\geq - \frac{\mathcal{P}(\ti m)}{4(E-\ti m)} \\
        &\geq C\mathbbm{1}_{|\xi| \geq 1} \mathbbm{1}_{|\tau| \geq 1} \<x\>^{-1-\d}|\xi|^2 \\
        &\geq C\mathbbm{1}_{|\xi| \geq 1} \mathbbm{1}_{|\tau| \geq 1} \left(\<x\>^{-1-\d}|\xi|^2 + \<x\>^{-2-2\d} \tau^2 \right)\\
        &\geq C\mathbbm{1}_{|\xi| \geq 1} \mathbbm{1}_{|\tau| \geq 1} \<x\>^{-2-2\d}(|\xi|^2+\tau^2).
    \end{align}
    This together with \eqref{eq:finalcombine1} give the desired growth behavior in $\tau, \xi$ and $\<x\>$ for all $\tau$.
\end{proof}

\section{Case Reduction}\label{s:caseReduce}
In this section we reduce the proofs of Theorems \ref{thm:iled} and \ref{thm:highfreq} to simpler problems.
Specifically we show the following.

\begin{proposition}\label{prop:iledCaseReduction}
    If there exists $C>0$, such that for all $T>0$ and $v$ with $v[0]=v[T]=0$, $Pv \in LE^*$ with $Pv$ compactly supported, we have 
    \begin{equation}
        \LEoT{v} \leq C \LEST{Pv},
    \end{equation}
    then there exists $C>0$, such that for all $T>0$, and $u$ with $u[0] \in \dot{H}^1 \times L^2$ we have 
    \begin{equation}\label{eq:iledCaseEst}
        \LEoT{u} + \nm{\p u}_{L^{\infty}_t L^2_x[0,T]} \leq C \left( \ltwo{\p u(0)}+ \LEsltxT{Pu}\right).
    \end{equation}
    That is, the conclusion of Theorem \ref{thm:iled} holds.
\end{proposition}

\begin{proposition}\label{prop:CaseReduction}
    If there exists $C>0$, such that for all $T>0$ and $v(t)$ supported in $\{|x|\leq 2R_0\}$ for $t \in [0,T]$, with $v[0]=v[T]=0$ and $Pv \in LE^*_c$, we have 
    \begin{equation}
        \LEoT{v} \leq C \left( \LtxT{v} + \LEST{Pv}\right),
    \end{equation}
    then there exists $C>0$, such that for all $T>0$, and $u$ with $u[0] \in \dot{H}^1 \times L^2$ we have
    \begin{equation}\label{eq:highFreqEst}
        \LEoT{u} + \nm{\p u}_{L^{\infty}_t L^2_x[0,T]} \leq C \left( \ltwo{\p u(0)}+ \LET{\<x\>^{-2} u}  + \LEsltxT{Pu}\right).
    \end{equation}
    That is, the conclusion of Theorem \ref{thm:highfreq} holds.
\end{proposition}
Note that the hypotheses and conclusions of the two propositions are different. Namely, Proposition \ref{prop:iledCaseReduction} does not restrict to $v$ supported in a fixed radius in the assumed estimate, and does not have $L^2_{t,x}$ or $LE$ terms on either right hand side.

Our proofs for these propositions follow a similar approach to that of \cite[Proposition 2.21, Section 3.3]{Kofroth23} and \cite[Section 4, Section 7]{MST20}. However our damping $a$ depends on time, which is not the case in \cite{Kofroth23, MST20} and the potential presence of our damping in the asymptotically flat region breaks time-reversal symmetry, which both other papers use. These differences are most significant in Section \ref{s:caseCauchy}. We also include complete details to emphasize that Propositions \ref{prop:iledCaseReduction} and \ref{prop:CaseReduction} follow from essentially the same arguments despite being written as separate arguments in each of \cite{Kofroth23} and \cite{MST20}.

We will prove these by successive reductions. Namely
\begin{enumerate}
    \item First, we show it suffices to control just the $LE^1$ norm of $u$ by the right hand side of \eqref{eq:iledCaseEst} or \eqref{eq:highFreqEst} (Lemma \ref{l:case1}).
    \item Next, we show that it suffices to consider $u$ with compactly supported Cauchy data $u[0]$ and inhomogeneity $Pu$ (Lemma \ref{l:case2}).
    \item Third, we show it suffices to consider $u$ with trivial Cauchy data $u[0]=u[T]=0$ and with $Pu \in LE^*_c$ (Lemma \ref{l:case3}).
    \item Finally, only for the proof of Proposition \ref{prop:CaseReduction}, we show that it suffices to consider $u$ supported within $\{|x|\leq 2R_0\}$ for all $t \in [0,T]$ with $u[0]=u[T]=0$ and $Pu \in LE^*$ (Lemma \ref{l:case4}).
\end{enumerate}
To prove these lemmas for both cases simultaneously, we include a term 
\begin{equation}
    \vartheta \LET{\<x\>^{-2} u},
\end{equation} 
on right hand sides and take $\vartheta \in \{0,1\}$. Specifically, to prove Proposition \ref{prop:iledCaseReduction} we take $\vartheta=0$, and to prove Proposition \ref{prop:CaseReduction} we take $\vartheta=1$.

We begin with a standard uniform energy inequality and a backwards-in-time version. 
\begin{lemma}\label{l:uniformEnergy}
    Let $P$ be a damped wave operator on a stationary space-time, with $\p_t$ uniformly time-like, and constant time-slices uniformly space-like.
    \begin{enumerate}
        \item There exists $C>0$ such that for all $T>0$ and $u[0] \in \dot{H}^1 \times L^2$
    \begin{equation}
        \ltwo{\p u(t)}^2 \leq C \left( \ltwo{\p u(0)}^2 + \int_0^T \int_{\Rb^3} |Pu \p_t u| dx dt \right), \quad 0 \leq t \leq T.
    \end{equation}
    \item Furthermore, there exists $C>0$ such that for all $T>0$ and $u[0] \in \dot{H}^1 \times L^2$
    \begin{equation}
        \ltwo{\p u(t)}^2 \leq C e^{2T\lp{a}{\infty}} \left( \ltwo{\p u(T)}^2 +  \int_0^T \int_{\Rb^3} |Pu \p_t u| dx dt  \right), \quad 0 \leq t \leq T.
    \end{equation}
    \end{enumerate}

\end{lemma}

\begin{proof}
    Let $Pu=f$ and define the energy
    \begin{equation}
        E[u](t) = \int_{\Rb^3} D_i g^{ij} D_j u \bar{u} - g^{00} |\p_t u|^2 dx.
    \end{equation}
    Integrating by parts, using the uniform ellipticity of $g^{ij}$ and that $g^{00} \geq -C$ we have 
    \begin{equation}\label{eq:ElikeG}
        E[u](t) = \int_{\Rb^3} g^{ij} D_j u D_i \bar{u} -g^{00} |\p_t u|^2 dx \simeq \int_{\Rb^3} |\nabla u|^2 + |\p_t u|^2 dx = \ltwo{\p u(t)}^2.
    \end{equation}
    Thus we can study $E[u](t)$ in place of $\ltwo{\p u(t)}^2$. We differentiate the energy with respect to $t$, then integrate by parts and apply $Pu=f$
    \begin{align}
        \frac{d}{dt} E[u](t) &= \int_{\Rb^3} D_i g^{ij} D_j \p_t u \bar{u} + D_i g^{ij} D_j u \p_t \bar{u} dx - \int_{\Rb^3} g^{00} (\p_t^2 u \p_t \bar{u} + \p_t u \p_t^2 \bar{u}) dx \\
        &=\int_{\Rb^3}(g^{00} D_t^2 + D_i g^{ij} D_j ) u \p_t \bar{u} + \p_t u \overline{ (g^{00}D_t^2 + D_i g^{ij} D_j u)}\,dx \\
        &=\int_{\Rb^3} (f- (g^{0j}D_j D_t + D_j g^{0j} D_t + i aD_t))u \p_t \bar{u}dx \\
        &\quad+ \int_{\Rb^3} \p_t u \overline{(f- (g^{0j}D_j D_t + D_j g^{0j} D_t + i aD_t))u} \,dx.
    \end{align} 
    Performing another integration by parts, the $g^{0j} D_j D_t + D_j g^{0j} D_t$ cross-terms cancel, and we are left with 
    \begin{equation}\label{eq:energyDeriv}
        \frac{d}{dt} E[u](t) = 2 \Re \int_{\Rb^3} \bar{f} \p_t u \,dx- 2 \int_{\Rb^3} a |\p_t u|^2 dx.
    \end{equation}
    Now integrating from $t=0$ to $t=s$, for $s \in [0,T]$ and using that $a \geq 0$ we have
    \begin{align}
        E[u](s) &= E[u](0) + \int_0^s \int_{\Rb^3} 2\Re(\bar{f} \p_t u) - 2 a|\p_t u|^2 dx dt \\
        &\leq E[u](0) + 2\int_0^s \int_{\Rb^3} |f \p_t u| dx dt \\
        &\leq E[u](0) + 2\int_0^T \int_{\Rb^3}  |f \p_t u| dx dt.
    \end{align}
    Applying \eqref{eq:ElikeG} provides the first conclusion.

    To see the second conclusion, we again estimate \eqref{eq:energyDeriv}
    \begin{align}
        \frac{d}{dt} E(u)(t) & \geq -2 \int_{\Rb^3} |f \p_t u| dx - 2 \lp{a}{\infty} \int_{\Rb^3} |\p_t u|^2 dx \\
        &\geq - 2\int_{\Rb^3}|f \p_t u| dx - 2 \lp{a}{\infty} E(u)(t).
    \end{align}
    Then by Lemma \ref{l:easyGronwall} for any $t \in [0,T]$
    \begin{align}
        E(u)(t) &\leq e^{2(T-t)\lp{a}{\infty}}\left(E(u)(T) + 2\int_{t}^{T} \int_{\Rb^3} |f \p_t u| dxdt  \right)\\
        &\leq e^{2T\lp{a}{\infty}}\left(E(u)(T) + 2\int_{0}^{T} \int_{\Rb^3} |f \p_t u| dxdt  \right).
    \end{align}
    Applying \eqref{eq:ElikeG} provides the second conclusion. 
\end{proof}

Next we show that this energy estimate gives control over the $L^{\infty}_t L^2_x$ norm of $\p u$ by the right hand side of \eqref{eq:iledCaseEst} or \eqref{eq:highFreqEst}, plus an absorbable error term. We also record a backwards-in-time version of this $L^{\infty}_t L^2_x$ estimate for later use.
\begin{lemma}\label{l:Linftyeasy}
    Let $P$ be a damped wave operator on a stationary space-time, with $\p_t$ uniformly time-like and constant time-slices uniformly space-like.
    \begin{enumerate}
        \item Then there exists $C>0$ such that for all $T>0$, all $u$ with $u[0] \in \dot{H}^1 \times L^2$, and all $0<\e\ll1$
    \begin{align}
         &\nm{\p u}_{L^{\infty}_t L^2_x[0,T]} \leq C \left( \ltwo{\p u(0)} + \nm{Pu}_{L^1_t L^2_x[0,T]}\right), \\
         &\nm{\p u}_{L^{\infty}_t L^2_x[0,T]} \leq C \left( \ltwo{\p u(0)} + \LEST{Pu}^{1/2} \LEoT{u}^{1/2}\right),\\
        &\nm{\p u}_{L^{\infty}_t L^2_x[0,T]} \leq C \left( \ltwo{\p u(0)} + \e^{-1} \nm{Pu}_{LE^*+L^1_tL^2_x[0,T]}+\e \LEoT{u}\right).
    \end{align}
    \item Furthermore, there exists $C>0$ such that for all $S>0$, all $u$ with $u[0] \in \dot{H}^1 \times L^2$, and all $0<\e\ll1$
        \begin{align}
         &\nm{\p u}_{L^{\infty}_t L^2_x[0,S]} \leq C e^{CS \lp{a}{\infty}} \left( \ltwo{\p u(S)} + \nm{Pu}_{L^1_t L^2_x[0,S]}\right), \\
         &\nm{\p u}_{L^{\infty}_t L^2_x[0,S]} \leq  C e^{CS \lp{a}{\infty}} \left( \ltwo{\p u(S)} + \nm{Pu}_{LE^*[0,t]}^{1/2} \nm{u}_{LE^1[0,S]}^{1/2}\right),\\
        &\nm{\p u}_{L^{\infty}_t L^2_x[0,S]} \leq  C e^{CS \lp{a}{\infty}} \left( \ltwo{\p u(S)} + \e^{-1} \nm{Pu}_{LE^*+L^1_tL^2_x[0,S]}+\e \nm{u}_{LE^1[0,S]}\right).
    \end{align}
    \end{enumerate}
\end{lemma}
We only prove the second set of statements. The first set of statements is exactly \cite[Corollary 2.19]{Kofroth23}, and the proof is similar.
\begin{proof}
1) By Lemma \ref{l:uniformEnergy} for $s \in [0,S]$ we have 
\begin{equation}\label{eq:reverseEasyenergy}
    \ltwo{\p u(s)}^2 \leq  C e^{CS \lp{a}{\infty}} \left( \ltwo{\p u(S)}^2 + \int_0^S \int_{\Rb^3} |Pu \p_t u| dx dt \right).
\end{equation}
Now we apply the Schwarz inequality and then the H\"older inequality to the second term
\begin{align}
    \ltwo{\p u(s)}^2 &\leq C e^{CS \lp{a}{\infty}}\left( \ltwo{\p u(S)}^2 + \int_0^S \ltwo{Pu(\cdot,t)}  \ltwo{\p u(\cdot,t)}dt \right) \\
    &\leq C e^{CS \lp{a}{\infty}}\left( \ltwo{\p u(S)}^2 + \nm{\p u}_{L^{\infty}_t L^2_x[0,S]} \nm{Pu}_{L^1_t L^2_x[0,S]} \right).
\end{align}
Now we take the supremum over $s \in [0,S]$ and apply Young's inequality for products to obtain for any $\e>0$
\begin{align}
    \nm{\p u}_{L^{\infty}_t L^2_x[0,S]}^2 &\leq C e^{CS \lp{a}{\infty}} \left( \ltwo{\p u(S)}^2 + \nm{\p u}_{L^{\infty}_t L^2_x[0,S]} \nm{Pu}_{L^1_t L^2_x[0,S]} \right). \\
    &\leq C e^{CS \lp{a}{\infty}} \left( \ltwo{\p u(S)}^2 + \frac{1}{\e} \nm{Pu}_{L^1_t L^2_x[0,S]}^2 +\e\nm{\p u}_{L^{\infty}_t L^2_x[0,S]}^2\right).
\end{align}
We can choose $\e>0$ small enough to absorb the final term back into the left hand side, then take square roots of both sides to obtain the desired inequality. 

2) To obtain the second estimate, we estimate the second term of \eqref{eq:reverseEasyenergy} by rewriting the integral, recalling the definition of $LE^*$, and applying the Schwarz inequality 
\begin{align}
   \int_0^S \int_{\Rb^3} |Pu \p_t u| dx dt &= \int_0^S \int_{\Rb^3} \left( \<x\>^{1/2} |Pu| \right) \left( \<x\>^{-1/2} |\p_t u| \right)dx dt\\
    &\leq \sum_{j=0}^{\infty} \int_0^S \int_{A_j} \left( \<x\>^{1/2} |Pu| \right) \left( \<x\>^{-1/2} |\p_t u| \right)dx dt  \\
    &\leq  \sum_{j=0}^{\infty} \nm{\<x\>^{1/2} Pu}_{L^2_tL^2_x([0,S]\times A_j)} \nm{\<x\>^{-1/2} \p_t u}_{L^2_t L^2_x ([0,S]\times A_j)}. 
\end{align}
Then computing directly and applying the definitions of $LE, LE^1$, and $LE^*$
\begin{align}
  \int_0^S \int_{\Rb^3} |Pu \p_t u| dx dt &\leq  \sup_{j \geq 0} \nm{\<x\>^{-1/2} \p_t u}_{L^2_t L^2_x ([0,S]\times A_j)}  \sum_{j=0}^{\infty} \nm{\<x\>^{1/2} Pu}_{L^2_tL^2_x([0,S]\times A_j)} \\
    &\leq \nm{\p u}_{LE[0,S]} \nm{Pu}_{LE^*[0,S]} \leq \nm{u}_{LE^1[0,S]} \nm{Pu}_{LE^*[0,S]}.
\end{align}
Plugging this back into \eqref{eq:reverseEasyenergy}, taking the supremum over $s \in [0,S]$ and taking square roots of both sides gives the second inequality.

3) To see the final estimate consider a fixed $u$ and $Pu=f$. By definition of the $LE^*+L^1_tL^2_x$ norm there exists $f_1^n \in L^1_t L^2_x[0,T]$ and $f_2^n \in LE^*[0,T]$ such that $f_1^n +f_2^n = f$ and 
\begin{equation}\label{eq:reverseEnergySplit}
    \LtxT{f_1^n} + \LEST{f_2^n}\leq \nm{f}_{LE^*+L^1_t L^2_x[0,T]}+\frac{1}{n}.
\end{equation}
We again estimate the second term of \eqref{eq:reverseEasyenergy}, beginning with the triangle inequality, then we estimate the first term using the approach in the first step and the second term using the approach in the second step to obtain 
\begin{align}
    \int_0^S \int_{\Rb^3} |Pu \p_t u| dx dt &\leq \int_0^S \int_{\Rb^3} |f_1 \p_t u| dx dt + \int_0^S \int_{\Rb^3} |f_2 \p_t u| dx dt\\
    &\leq  \nm{\p u}_{L^{\infty}_t L^2_x[0,S]} \nm{f_1^n}_{L^1_t L^2_x[0,S]} + \nm{u}_{LE^1[0,S]}\nm{f_2^n}_{LE^*[0,S]}.
\end{align}
Plugging this back into \eqref{eq:reverseEasyenergy}, taking the supremum over $s \in [0,S]$, and applying Young's inequality for products we obtain for any $\e>0$
\begin{align}
    \nm{\p u}_{L^{\infty}_t L^2_x[0,S]}^2 \leq C e^{CS \lp{a}{\infty}} \bigg( &\ltwo{\p u(S)}^2 + \e^{-1} (\nm{f_1^n}_{L^1_t L^2_x[0,S]}^2 + \nm{f_2^n}_{LE^*[0,S]}^2) \\
    &+ \e \nm{\p u}^2_{L^{\infty}_t L^2_x[0,S]} + \e \nm{u}_{LE^1[0,S]}^2 \bigg).
\end{align}
Taking $\e>0$ small enough we can absorb the $\p u$ term back into the left hand side. Then applying \eqref{eq:reverseEnergySplit} we have 
\begin{equation}
    \nm{\p u}_{L^{\infty}_t L^2_x[0,S]}^2 \leq C e^{CS\lp{a}{\infty}} \bigg( \ltwo{\p u(S)}^2 + \e^{-1} (\nm{Pu}_{LE^*+L^1_t L^2_x[0,S]}^2 + \frac{1}{n^2}) + \e \nm{u}_{LE^1[0,S]}^2 \bigg).
\end{equation}
We now take $n \ra \infty$, and take square roots to obtain the desired inequality. 
\end{proof}

\subsection{Removal of $L^{\infty}_tL^2_x$ term from left hand side}
We can now reduce \eqref{eq:iledCaseEst} and \eqref{eq:highFreqEst} to controlling just the $LE^1$ norm, using Lemma \ref{l:Linftyeasy}. 
\begin{lemma}\label{l:case1}
    Fix $\vartheta \in \{0,1\}$. If there exists $C>0$, such that for all $T>0$, and $u$ with $u[0] \in \dot{H}^1 \times L^2$ we have
    \begin{equation}\label{eq:highFreqLEo}
        \LEoT{u} \leq C \left( \ltwo{\p u(0)} + \vartheta \LET{\<x\>^{-2}u} + \nm{Pu}_{LE^* + L^1_t L^2_x[0,T]}\right),
    \end{equation}
    then there exists $C>0$, such that for all $T>0$, and $u$ with $u[0] \in \dot{H}^1 \times L^2$, we have
    \begin{equation}
        \LEoT{u} + \nm{\p u}_{L^{\infty}_t L^2_x[0,T]} \leq C \left( \ltwo{\p u(0)}+ \vartheta \LET{\<x\>^{-2} u}  + \LEsltxT{Pu}\right).
    \end{equation}
    That is, when $\vartheta=0$, resp. $\vartheta=1$, the inequality \eqref{eq:iledCaseEst}, resp. \eqref{eq:highFreqEst}, holds.
\end{lemma}
\begin{proof}
    By Lemma \ref{l:Linftyeasy} part 1, there exists $C>0$ such that for any $\e>0$
    \begin{equation}
        \nm{\p u}_{L^{\infty}_t L^2_x[0,T]} \leq C \left( \ltwo{\p u(0)} + \e^{-1} \nm{Pu}_{LE^*+L^1_tL^2_x[0,T]}+\e \LEoT{u}\right).
    \end{equation}
    Combining this with our assumed estimate we have 
    \begin{align}
        \LEoT{u}+\nm{\p u}_{L^{\infty}_t L^2_x[0,T]} \leq C \bigg(& \ltwo{\p u(0)} + \vartheta \LET{\<x\>^{-2}u} 
        \\&+ \e^{-1} \nm{Pu}_{LE^*+L^1_tL^2_x[0,T]}+\e \LEoT{u}\bigg).
    \end{align}
    Now choosing $\e>0$ small enough, we can absorb the final term on the right hand side back into the left hand side and obtain exactly \eqref{eq:iledCaseEst} when $\vartheta=0$, and \eqref{eq:highFreqEst} when $\vartheta=1$.
\end{proof}

\subsection{Reduction to compactly supported Cauchy data and inhomogeneity}

To prove \eqref{eq:highFreqLEo}, we first see that it suffices to consider $u$ with data $u[0]$ and inhomogeneity $Pu$ supported in a set of fixed radius.
Before proving this reduction, we define a small perturbation of $\Box_m$ and cite a local energy decay result for such perturbations.
\begin{definition}
    Consider 
    \begin{equation}
        \ti{P}(t,x,D) = D_{\alpha} c^{\alpha \beta}(t,x) D_{\beta} + b^{\alpha}(t,x) D_{\alpha}.
    \end{equation}
    Recalling the norms from Definition \ref{d:asymptoticFlat1}, we say $\ti{P}$ is a small asymptotically flat perturbation of $\Box_m$ if for some $\e>0$ sufficiently small 
    \begin{equation}
        \|c-m\|_2 +  \|\<x\>b\|_{1} < \e.
    \end{equation}
\end{definition}
Note that this definition of a small AF perturbation of $\Box_m$ is compatible with our definition of asymptotic flatness in Definition \ref{d:asymptoticFlat1}. That is, there are small AF perturbations of $\Box_m$ which agree with $P$ for $|x|>R_0$.
\begin{theorem}\label{thm:mt12}\cite[Theorem 1]{MT12}
    If $\ti{P}$ is a small asymptotically flat perturbation of $\Box_m$, then there exists $C>0$ such that for all $T>0$ and $u$ with $u[0] \in \dot{H}^1 \times L^2$ we have
    \begin{equation}
        \LEoT{u} + \nm{\p u}_{L^{\infty} L^2[0,T]} \leq C \left( \ltwo{\p u(0)} + \LEsltxT{\ti{P}u} \right).
    \end{equation}
\end{theorem}
We now state a preliminary lemma that allows us to separate our solution $u$ into a   a solution of a small AF perturbation of the Minkowski wave operator and a piece with compactly supported initial data and inhomogeneity.
\begin{lemma}\label{l:case2Pre}
    Let $\ti{P}$ be a small asymptotically flat perturbation of $\Box_m$ that agrees with $P$ for $|x| >R_0$. For $u$ with $u[0] \in \dot{H}^1 \times L^2$ let $v$ solve 
    \begin{equation}
        \begin{cases}
            \ti{P}v=Pu \\
            v[0]=u[0],
        \end{cases}
    \end{equation}
    and let $u_1=u-\chi_{>R_0} v$.
    \begin{enumerate}
        \item Then $u_1[0]$ and $Pu_1$ are compactly supported in $\{|x|\leq 2R_0\}$. 
        \item Furthermore, there exists $C>0$ such that
    \begin{align}
        &\LEoT{v} + \LEoT{\chi_{>R_0} v} \leq C \left( \ltwo{\p u(0)} + \LEsltxT{Pu} \right). 
    \end{align}
    \item Additionally,
    \begin{align}
        &\ltwo{\p u_1(0)} \leq C \ltwo{\p u(0)} \\
        &\LET{\<x\>^{-2} u_1} \leq \LET{\<x\>^{-2} u} + \LEoT{v}\\
        &\LEsltxT{Pu_1} \leq \LEsltxT{P u} + C \LEoT{v}.
    \end{align}
    \end{enumerate}
\end{lemma}
\begin{proof}
    1) First, note that 
    \begin{equation}
        u_1[0]=u[0]-\chi_{>R_0} v[0] = (1-\chi_{>R_0}) u[0],
    \end{equation}
    which is compactly supported in $\{|x|\leq 2R_0\}$. Furthermore 
    \begin{align}
        P u_1 &= Pu - P(\chi_{>R_0} v) \\
        &= Pu - \chi_{>R_0} Pv -[P, \chi_{>R_0}]v \\
        &=Pu-\chi_{>R_0} \ti{P}v -[P,\chi_{>R_0}]v \\
        &=(1-\chi_{>R_0}) Pu -[P,\chi_{>R_0}]v, \label{eq:caseReductionPu1formula}
    \end{align}
    and both terms on the right hand side are compactly supported in $\{|x| \leq 2R_0\}$. 
    
    2) Next, using that $\p \chi_{>R_0}$ is compactly supported in $\{|x| \leq 2R_0\}$ and $\<x\>^{-1} \geq c$ there, we have
    \begin{align}
        \LEoT{\chi_{>R_0} v} &= \LET{\p (\chi_{>R_0} v)} + \LET{\<x\>^{-1} \chi_{>R_0} v}\\
        &\leq \LET{\chi_{>R_0} \p v}+\LET{(\p \chi_{>R_0}) v } + \LET{\<x\>^{-1}v}\\
        &\leq C \left( \LET{\p v} + \LET{\<x\>^{-1}v} \right) \\
        &\leq C \LEoT{v}.
    \end{align}
    Since $\ti{P}$ is a small asymptotically flat perturbation of $\Box_m$, by Theorem \ref{thm:mt12} there exists $C>0$ such that for all $T>0$
    \begin{align}
        \LEoT{v} &\leq C \left( \ltwo{\p  v(0)} + \nm{\ti{P} v}_{LE^*+L^1_t L^2_x[0,T]} \right) \\
        &\leq C \left( \ltwo{\p u(0)} + \nm{Pu}_{LE^*+L^1_t L^2_x[0,T]} \right), 
    \end{align}
    where the second inequality follows from the definition of $v$ in terms of $u$. Therefore 
    \begin{equation}\label{eq:CaseReducechivR}
        \LEoT{v}+\LEoT{\chi_{>R_0}v} \leq C \left( \ltwo{\p u(0)} + \nm{Pu}_{LE^*+L^1_t L^2_x[0,T]} \right).
    \end{equation}
    3) Now we prove the remaining estimates. First 
    \begin{equation}
        \ltwo{\p u_1(0)} = \ltwo{\p(( 1-\chi_{>R_0})u)(0)} \leq \ltwo{\p u(0)} + \ltwo{(\nabla \chi_{>R_0}) u(0)}.
    \end{equation}
    To estimate the second term on the right hand side, we compute directly and then apply the Hardy inequality  
    \begin{align}
        \ltwo{(\nabla \chi_{>R_0}) u(0)} &\leq C \nm{u(0)}_{L^2(|x|<2R_0)} \\
        &\leq C \nm{|x|^{-1} u(0)}_{L^2(|x|<2R_0)}\\
        &\leq C \ltwo{|x|^{-1} u(0)} \leq C \ltwo{\nabla u(0)}. \label{eq:hardyApp}
    \end{align}
    Therefore 
    \begin{equation}
        \label{eq:caseReductionu1data}
        \ltwo{\p u_1(0)}\leq C\ltwo{\p u(0)}.
    \end{equation}
    To prove the next inequality we use the triangle inequality, $\<x\>^{-1} \leq 1$, and the definition of $\LEo{\cdot}$
    \begin{align}
        \LET{\<x\>^{-2} u_1} &\leq \LET{\<x\>^{-2}u} + \LET{\<x\>^{-2} \chi_{>R_0} v} \\
        &\leq \LET{\<x\>^{-2}u}+\LET{\<x\>^{-1}  v} \\
        &\leq \LET{\<x\>^{-2}u}+\LEoT{v}. \label{eq:caseReducex2u1}
    \end{align}
    To prove the final inequality we begin by applying  \eqref{eq:caseReductionPu1formula} and the triangle inequality to see
    \begin{align}\label{eq:caseReductionPu1intermed}
        \nm{Pu_1}_{LE^* + L^1_t L^2_x[0,T]}&\leq \nm{Pu}_{LE^* + L^1_t L^2_x[0,T]} + \nm{[P,\chi_{>R_0}]v}_{LE^* + L^1_t L^2_x[0,T]}.
    \end{align}
    To control the second term on the right hand side, first note that $[P,\chi_{>R_0}]$ is a first order space-time differential operator with coefficients compactly supported in $\{|x| < 2R_0\}$. Using this, along with the definition of the $LE^* + L^1_t L^2_x$ and $LE^*$ norms we have
    \begin{align}
        \nm{[P,\chi_{>R_0}]v}_{LE^* + L^1_t L^2_x[0,T]} &\leq \LEST{[P, \chi_{>R_0}]v} \\
        &\leq C \LEST{\chi_{<2R_0} \p v} + \LEST{\chi_{<2R_0}v}\\
         &\leq C \left( \LET{\p v} + \LET{\<x\>^{-1}v} \right)\\
        &\leq C\LEoT{v}.
    \end{align}
    Plugging this back into \eqref{eq:caseReductionPu1intermed} we obtain the desired inequality.
\end{proof}
The case reduction lemma is a straightforward consequence of the preceding one. 
\begin{lemma}\label{l:case2}
    Fix $\vartheta \in \{0,1\}$. Assume that there exists $C>0$, such that for all $T>0$ and $u_1$ with $u_1[0] \in \dot{H}^1 \times L^2$, and $u_1[0],Pu_1$ compactly supported in $\{|x| \leq 2R_0\}$,  the inequality \eqref{eq:highFreqLEo} holds with $\vartheta$. That is
    \begin{equation}
        \LEoT{u_1} \leq C \left( \ltwo{\p u_1(0)} + \vartheta \LET{\<x\>^{-2}u_1} + \nm{Pu_1}_{LE^* + L^1_t L^2_x[0,T]}\right).
    \end{equation}
    Then there exists $C>0$ such that, for all $T>0$ and $u$ with $u[0] \in \dot{H}^1 \times L^2$, 
    \eqref{eq:highFreqLEo} holds, with the same $\vartheta$.
\end{lemma}
\begin{proof}
    Let $\ti{P},v,$ and $u_1$ be as in Lemma \ref{l:case2Pre}. Then by Lemma \ref{l:case2Pre} part 1, $u_1[0]$ and $Pu_1$ are compactly supported. Therefore, by our assumption, there exists $C>0$ such that for all $T>0$
    \begin{equation}\label{eq:caseReduceu1Compact}
        \LEoT{u_1} \leq C \left( \ltwo{\p u_1(0)} +\vartheta  \LET{\<x\>^{-2}u_1} + \nm{Pu_1}_{LE^* + L^1_t L^2_x[0,T]} \right).
    \end{equation}
    Then estimating the terms on the right hand side using Lemma \ref{l:case2Pre} part 3 we have 
    \begin{equation}
        \LEoT{u_1} \leq C \left( \ltwo{\p u(0)} + \vartheta \LET{\<x\>^{-2}u}+\nm{Pu}_{LE^*+L^1_t L^2_x[0,T]} +\LEoT{v} \right). \label{eq:caseReduceu1cauchyCompact}
    \end{equation}
    Applying the triangle inequality, Lemma \ref{l:case2Pre} part 2 and \eqref{eq:caseReduceu1cauchyCompact} we obtain
    \begin{align}
        \LEoT{u} &\leq \LEoT{u_1} + \LEoT{\chigro v} \\
        &\leq C \left( \ltwo{\p u(0)} + \vartheta \LET{\<x\>^{-2} u} + \nm{Pu}_{LE^*+L^1_tL^2_x[0,T]} \right),
    \end{align}
    which is exactly \eqref{eq:highFreqLEo} as desired. 
\end{proof}

\subsection{Reduction to trivial Cauchy data}\label{s:caseCauchy}
We now show that it suffices to consider solutions with trivial Cauchy data, $w[0]=w[T]=0$, with $Pw \in LE^*_c$.

We begin with a preliminary lemma. We split the time interval $[0,T]$ into sub-intervals of length $1$ and obtain estimates for solutions of truncated versions of $Pu$ on each of these subintervals. 

Note that we only match a spatially cutoff version of $u[T]$ in our definition of the $w_k$ below. This is to ensure the size of the compact spatial support of $w_N$ does not depend on $T$ and is one key difference from \cite{MST20} and \cite{Kofroth23}.
\begin{lemma}\label{l:case3pre} 
Given $T>0$, and a function $u$ defined on $[0,T] \times \Rb^3$, such that $Pu$ is well defined, let $N$ be the largest integer strictly less than $T$. For integers $0 \leq k \leq N$, let $w_k$ solve
\begin{equation}
    \begin{cases}
        Pw_k = \mathbbm{1}_{[k,k+1]}(t) Pu \\
        w_0[0] = u[0] \\
        w_k[k]=0, \quad 1 \leq k \leq N-1 \\
        w_N[T]=\chi_{<2R_0} u[T].
    \end{cases}
\end{equation}
Further set $\alpha = \frac{1}{2}(T-N)$ and define $\chi_{[k,k+1]}(t) \in \Cc((k-\alpha, k+1+\alpha): [0,1])$ with $\chi_{[k,k+1]}(t) \equiv 1$ for $t \in [k,k+1]$. 

There exists $C>0$ such that for all $T>0$ and all $u$ with $u[0] \in \dot{H}^1\times L^2$ and $u[0], Pu$ supported in $\{|x| \leq 2R_0\}$,
\begin{enumerate}
    \item For all $0 \leq k\leq N$, the function $w_k(t)$ is compactly supported in $\{|x| \leq C R_0\}$ for $t \in [k-\alpha, k+1+\alpha]$.
    \item  We have
\begin{align}
    &\nm{\chi_{[k,k+1]}w_k}_{LE^1[k-\alpha, k+1+\alpha]} + \nm{w_k}_{LE^1[k-\alpha, k+1+\alpha]} \leq C \nm{\p w_k}_{L^{\infty}_t L^2_x[k-\alpha, k+1+\alpha]},
\end{align}
where when $k=0$, the time intervals on both sides are $[0,1+\alpha]$, and when $k=N$ the time intervals are $[N-\alpha,T]$.
\item Furthermore, for all $\e>0$
\begin{align}
    \sum_{k=0}^N\nm{ \p w_k}_{L^{\infty}_t L^2_x[k-\alpha,k+1+\alpha]}
    \leq C \left( \ltwo{\p u(0)} + \e^{-1} \LEsltxT{Pu} + \e \LEoT{u} \right).
\end{align}
\item Finally, for all $\e>0$
\begin{equation}
     \LEST{\sum_{k=0}^N [P,\chi_{[k,k+1]}]w_k} \leq C \left( \ltwo{\p u(0)} + \e^{-1} \LEsltxT{Pu} + \e \LEoT{u} \right).
\end{equation}

\end{enumerate}
\end{lemma}
\begin{proof}
    1) To begin we note that $u[0]$, and $Pu$ are supported in $\{|x| <2R_0\}$, $w_k[k]=0,$ and $\chi_{<2R_0} u[T]$ is supported in $\{|x| <4R_0\}$. Therefore, by finite speed of propagation there exists $C>0$ such that $w_k(t)$ is supported in $\{|x| <C R_0\}$ for $t \in [k-\alpha, k+1+\alpha]$.

    2) Because of this, by the Poincar\'e inequality, there exists $C_p=C_p(R_0)>0$ such that 
    \begin{align}\label{eq:case2Poincare}
        \ltwo{w_k(t, \cdot)} \leq C_p \ltwo{\nabla_x w_k (t, \cdot)},\quad t \in [k-\alpha, k+1+\alpha].
    \end{align}
    Now we can obtain the desired $LE^1$ estimates. By the definition of $LE^1$ and since $\<x\>^{-1} \leq 1$
    \begin{align}
        \nm{\chi_{[k,k+1]} w_k}_{LE^1[k-\alpha,k+1+\alpha]} +& \nm{w_k}_{LE^1[k-\alpha,k+1+\alpha]} \\
        &\leq C \left( \nm{w_k}_{LE[k-\alpha,k+1+\alpha]} + \nm{\p w_k }_{LE[k-\alpha,k+1+\alpha]} \right).
    \label{eq:Case2LEowk}
    \end{align}
    Note that when $k=0$, we take $[0,1+\alpha]$ as the time interval in these norms, and when $k=N$, we take $[N-\alpha, T]$ as the time interval in these norms. 
    To control the first term on the right hand side of \eqref{eq:Case2LEowk}, we use the definition of $LE$, the Poincar\'e inequality \eqref{eq:case2Poincare}, and that $[k-\alpha, k+1+\alpha]$ has length $ \leq 2$ to see
    \begin{align}
        \nm{w_k}_{LE[k-\alpha,k+1+\alpha]} &\leq \left(\int_{k-\alpha}^{k+1+\alpha}\ltwo{w_k(s, \cdot)}^2 ds \right)^{1/2} \leq C\left(\int_{k-\alpha}^{k+1+\alpha} \ltwo{\nabla w_k(s, \cdot)}^2 ds \right)^{1/2}  \\
        &\leq C\nm{\p w_k}_{L^{\infty}_{t} L^2_x [k-\alpha, k+1+\alpha]}. \label{eq:Case2LEointer1}
    \end{align}
    For the second term on the right hand side of \eqref{eq:Case2LEowk}, again using the definition of $LE$ and that $[k-\alpha, k+1+\alpha]$ has length $ \leq 2$, we estimate  
    \begin{equation}
        \nm{\p w_k}_{LE[k-\alpha,k+1+\alpha]} \leq \left( \int_{k-\alpha}^{k+1+\alpha} \ltwo{\p w_k(s, \cdot)} ds \right)^{1/2} \leq C\nm{\p w_k}_{L^{\infty}_{t } L^2_x[k-\alpha, k+1+\alpha]}.\label{eq:Case2LEointer2}
    \end{equation}
    Plugging this and \eqref{eq:Case2LEointer1} into \eqref{eq:Case2LEowk} we obtain
    \begin{align}\label{eq:LEwkLinfty}
       \nm{\chi_{[k,k+1]}w_k}_{LE^1[k-\alpha,k+1+\alpha]} + \nm{w_k}_{LE^1[k-\alpha,k+1+\alpha]} \leq C  \nm{\p w_k}_{L^{\infty}_{t} L^2_x[k-\alpha, k+1+\alpha]},
    \end{align}
    which is the desired inequality.

    3) We now further estimate the $L^{\infty}_t L^2_x$ norms. By Lemma \ref{l:Linftyeasy} part 1 we have for any $0<\e\ll1$
    \begin{align}
        \nm{\p w_0}_{L^{\infty}_{t}L^2_x[0,1+\alpha]} &\leq C\left( \ltwo{\p w_0(0)} + \e^{-1} \nm{Pw_0}_{LE^*+L^1_t L^2_x[0,1+\alpha]} + \e\nm{w_0}_{LE^1[0,1+\alpha]} \right).
    \end{align}
    Now applying $w_0[0]=u[0], Pw_0=\mathbbm{1}_{[0,1]} Pu$, and part 1 of this Lemma, for any $\e>0$ we have
    \begin{equation}
        \nm{\p w_0}_{L^{\infty}_{t}L^2_x[0,1+\alpha]}  \leq C \left(\ltwo{\p u(0)}+\e^{-1} \nm{Pu}_{LE^*+L^1_t L^2_x[0,1]} + \e\nm{\p w_0}_{L^{\infty}_{t}L^2_x[0,1+\alpha]} \right).
    \end{equation}
    Choosing $\e>0$ small enough we can absorb the final term on the right hand side back into the left hand side to obtain 
    \begin{equation}
         \nm{\p w_0}_{L^{\infty}_{t}L^2_x[0,1+\alpha]}  \leq C \left(\ltwo{\p u(0)}+ \nm{Pu}_{LE^*+L^1_t L^2_x[0,1]} \right). \label{eq:case2w0}
    \end{equation}
    To control the terms for $1 \leq k \leq N-1$, we write 
    \begin{equation}
        \nm{\p w_k}_{L^{\infty}_tL^2_x[k-\alpha,k+1+\alpha]} \leq \nm{\p w_k}_{L^{\infty}_tL^2_x[k-\alpha, k]} + \nm{\p w_k}_{L^{\infty}_t L^2_x[k,k+1+\alpha]}.
    \end{equation}
    Then we estimate the first term using Lemma \ref{l:Linftyeasy} part 2, noting that $[k-\alpha, k]$ has length $\leq 1/2$, and the second term using Lemma \ref{l:Linftyeasy}, part 1. Combining these, recalling $Pw_k=\mathbbm{1}_{[k,k+1]} Pu$, $w_k[k]=0$, and applying part 1 of this lemma we obtain for any $0<\e\ll1$
    \begin{align}
        \nm{\p w_k}_{L^{\infty}_{t}L^2_x[k-\alpha,k+1+\alpha]} &\leq C\left( \ltwo{\p w_k(k)} + \e^{-1} \nm{Pw_k}_{LE^*+L^1_t L^2_x[k-\alpha,k+1+\alpha]} + \e \nm{w_k}_{LE^1[k-\alpha,k+1+\alpha]} \right) \\
        & \leq C\left(\e^{-1} \nm{P u}_{LE^*+L^1_t L^2_x[k,k+1]} + \e \nm{\p w_k}_{L^{\infty}_{t}L^2_x[k-\alpha,k+1+\alpha]} \right).
    \end{align}
    Again choosing $\e>0$ small enough, we can absorb the second term on the right hand side back into the left, to obtain 
    \begin{align}
        \nm{\p w_k}_{L^{\infty}_{t}L^2_x[k-\alpha,k+1+\alpha]}  &\leq C \nm{P u}_{LE^*+L^1_t L^2_x[k,k+1]}.      
        \label{eq:case2wk}
    \end{align}
    Finally we estimate $w_N$ using Lemma \ref{l:Linftyeasy} part 2, recalling that $T-(N-\alpha) < 2$, $w_N[T]=u[T]$, and $Pw_n=\mathbbm{1}_{[N,N+1]} Pu$, and using part 1 of this lemma, to see that for any $\e>0$
    \begin{align}
        \nm{\p w_N}_{L^{\infty}_t L^2_x[N-\alpha,T]}& \leq C \left( \ltwo{\p w_N(T)} + \e^{-1} \nm{Pw_N}_{LE^*+L^1_t L^2_x[N-\alpha,T]} +\e \nm{ w_N}_{LE^1[N-\alpha,T]} \right) \\
        &\leq C \left( \ltwo{\p u(T)} + \e^{-1} \nm{P u}_{LE^*+L^1_t L^2_x[N ,T]} +\e \nm{\p w_N}_{L^{\infty}_t L^2_x[N,T]}  \right).
    \end{align}
    We again choose $\e>0$ small enough so that we can absorb the third term on the right hand side back into the left to obtain
    \begin{align}
        \nm{\p w_N}_{L^{\infty}_t L^2_x[N-\alpha,T]}&\leq  C \left( \ltwo{\p u(T)} + \nm{P u}_{LE^*+L^1_t L^2_x[N,T]} \right).
        \label{eq:case2wN}
    \end{align}
    Now we apply Lemma \ref{l:Linftyeasy} part 1 to estimate $\p u(T)$ 
    \begin{align}
        \ltwo{\p u(T)}^2 &\leq \nm{\p u}_{L_t^{\infty}L^2_x[0,T]} \\
        &\leq C\left( \ltwo{\p u(0)} + \e^{-1} \nm{Pu}_{LE^*+L^1_tL^2_x[0,T]} + \e \nm{u}_{LE^1[0,T]}\right).
    \end{align}
    Combining this with \eqref{eq:case2wN} we have for any $\e>0$
    \begin{align}
        \nm{\p w_N}_{L^{\infty}_t L^2_x[N-\alpha,T]} \leq C\left( \ltwo{\p u(0)} + \e^{-1} \nm{Pu}_{LE^*+L^1_tL^2_x[0,T]} + \e \nm{u}_{LE^1[0,T]}\right). \label{eq:case2wNfinal}
    \end{align}
    Adding together the $L^{\infty}_t L^2_x$ inequalities for $0 \leq k \leq N$ gives the desired inequality.

    4) To see the final inequality we begin by recalling the definition of $LE^*$ and noting that because the $w_k$ are compactly supported inside $\{|x| \leq CR_0\}$ we have $\<x\> \leq CR_0$, so
    \begin{align}
        \LEST{[P,\chi_{[k,k+1]}(t)]w_k} &= \sum_{j=0}^{\infty} \nm{\<x\>^{1/2} [P,\chi_{[k,k+1]}(t)]w_k}_{L^2_t L^2_x([0,T] \times A_j)}\\
        &\leq C\LtxT{[P,\chi_{[k,k+1]}]w_k}.
    \end{align}
    Now note that $[P, \chi_{[k,k+1]}]$ is a first order space-time differential operator with $t$ support contained in $[k-\alpha,k+1+\alpha]$. 
    Using this and the Poincar\'e inequality \eqref{eq:case2Poincare}
    \begin{align}
        \LtxT{[P,\chi_{[k,k+1]}]w_k} &\leq C\left(\nm{\p w_k}_{L^2_{t}L^2_x[k-\alpha,k+1+\alpha]} + \nm{w_k}_{L^2_{t}L^2_x[k-\alpha,k+1+\alpha]}\right) \\
        &\leq C \nm{\p w_k}_{L^2_{t}L^2_x[k-\alpha,k+1+\alpha]}.
    \end{align}
    Now since the time interval $[k-\alpha,k+1+\alpha]$ has length $\leq 2$ we can replace $L^2_t$ by $L^{\infty}_t$ to obtain
    \begin{equation}
        \LtxT{[P,\chi_{[k,k+1]}]w_k}\leq C \nm{\p w_k}_{L^{\infty}_{t}L^2_x[k-\alpha,k+1+\alpha]}.
    \end{equation}
    Combining this with part 2 of this lemma gives the desired inequality.
\end{proof}
We now state the case reduction lemma. 
\begin{lemma}\label{l:case3}
Fix $\vartheta \in \{0,1\}$. Assume that there exists $C>0$, such that for all $T>0$, and $w$ with $w[0]=w[T]=0$, $Pw \in LE^*_c$, we have 
\begin{equation}\label{eq:highFreqCase2}
    \LEoT{w} \leq C \left( \vartheta \LET{\<x\>^{-2}w} + \LEST{P w}\right).
\end{equation}
Then there exists $C>0$, such that for all $T>0$, and $u_1$ with $u_1[0] \in \dot{H}^1 \times L^2$ and $u_1[0], Pu_1$ compactly supported in $\{|x| \leq 2R_0\}$, we have 
\begin{equation}
    \LEoT{u_1} \leq C \left( \ltwo{\p u_1(0)} + \vartheta \LET{\<x\>^{-2}u_1} + \nm{Pu_1}_{LE^* + L^1_t L^2_x[0,T]}\right).
\end{equation}
\end{lemma}
Before proceeding with the proof of this Lemma, we note that Proposition \ref{prop:iledCaseReduction} is an immediate consequence of Lemmas \ref{l:case1}, \ref{l:case2}, and \ref{l:case3} with $\vartheta=0$.

Note also, we define $v_T$ below using a perturbation of $\Box_m$ rather than $P$ to ensure we can apply Lemma \ref{l:Linftyeasy} part 2 with no $T$ dependence. This is our replacement for the loss of time-reversal symmetry of $P$ outside $|x|>R_0$ as compared to \cite{Kofroth23}, as our $a$ need not be compactly supported.
\begin{proof}

    1) Suppose $Pu_1 \in LE^* +L^1_t L^2_x$ and $u_1[0] \in \dot{H}^1 \times L^2$ are both compactly supported in $\{|x| \leq 2R_0\}$.
    Let $N, \alpha, \chi_{[k,k+1]}$, and $w_k$ be as in Lemma \ref{l:case3pre} using $u_1$ in place of $u$. Note that from the same Lemma the $w_k$ are compactly supported. Now let $\ti{\Box}$ be a small asymptotically flat perturbation of $\Box_m$ such that $\ti{\Box}=\Box_g$ for $|x|>R_0$.
    Then let $v_T$ solve 
    \begin{equation}
        \begin{cases}
            \ti{\Box}v_T = P u_1 \\
            v_T[T]= u_1[T].
        \end{cases}
    \end{equation}
    Note that since $u_1[0]$ and $Pu_1$ are compactly supported in $\{|x| \leq R_0\}$, by finite speed of propagation $u_1[T]$ is compactly supported in $\{|x| \leq C(T+1)\}$. Again by finite speed of propagation $v_N(t)$ is compactly supported in $\{x<C(T+1)\}$ for $t \in [0,T]$.
    
    Now let $\psi \in \Cc((T/4, 2T):(0,1))$ have $\psi \equiv 1$ for $t \in [T/2,3T/2]$, and  $|\p_t^k \psi| \leq C T^{-k}$ for $k=1,2$. Then define 
    \begin{equation}
        w(t,x) = u_1(t,x) - \sum_{k=0}^N \chi_{[k,k+1]}(t) w_k(t,x) - \chi_{>2R_0}(x) \psi(t) v_T(t,x).
    \end{equation}
    Therefore to control the $LE^1$ norm of $u_1$ it is enough to control the $LE^1$ norm of $w$, $\chi_{>2R_0}(x) \psi(t) v_T$, and the $w_k$. 
    We will first control $v_T$. Then we will control $w$ using our assumed estimate. We will control the $w_k$ with Lemma \ref{l:case3pre}, then combine all these estimates to conclude. 

    2) To estimate $v_T$ we again apply Theorem \ref{thm:mt12}, and then Lemma \ref{l:Linftyeasy} part 2, to obtain for all $\e>0$
    \begin{align}
        \LEoT{v_T} &\leq C \left( \ltwo{\p  v_T(0)} + \LEsltxT{\ti \Box v_T} \right) \\
        &\leq C \left( \ltwo{\p v_T (T)} + \LEsltxT{\ti \Box v_T } + \e \LEoT{v_T}\right)
    \end{align}
     Note that the constant $C$ from Lemma \ref{l:Linftyeasy} does not depend on $T$ since $\ti{\Box}$ has no damping term $a$. Choosing $\e>0$ small enough to absorb the error term back into the left hand side and applying the definition of $v_T$ we obtain 
    \begin{align}
        \LEoT{v_T} &\leq C \left( \ltwo{\p u_1 (T)} + \LEsltxT{P u_1 } \right). \label{eq:case3vTLEo}
    \end{align}    
    We can similarly apply Theorem \ref{thm:mt12}, and then the definition of $\psi(t)$ and $v_T$ to obtain
    \begin{align}
        \LEoT{\chi_{>2R_0}(x) \psi(t) v_T} &\leq C \left( \ltwo{\p  (\chi_{>2R_0}(x) \psi(t) v_T)(0)} + \LEsltxT{\ti{\Box} (\chi_{>2R_0}(x) \psi(t) v_T)} \right) \\
        &\leq C \bigg(\LEsltxT{\chi_{>2R_0} P u_1} \\
        &\qquad \quad+ \LEsltxT{[\ti \Box, \chi_{>2R_0}(x)] \psi(t) v_T}\\
        &\qquad \quad+ \LEsltxT{\chi_{>2R_0}(x) [\ti \Box, \psi(t)] v_T} \bigg). \label{eq:case3vLEo}
    \end{align}
    Note that since $Pu_1$ is supported in $\{|x| \leq 2R_0\}$, the first term $\chi_{>2R_0} Pu_1 \equiv 0$. Now we estimate the two commutator terms on the right hand side. First note that $[\ti \Box,\chi_{>2R_0}]$ is compactly supported in $\{2 R_0 \leq |x|\leq 4 R_0\}$. In particular $\<x \> \leq CR_0$ on this set, and so by the definition of $LE^*$
    \begin{align}
        \LEsltxT{[\ti \Box, \chi_{>2R_0}(x)] \psi(t) v_T} &\leq \LEST{[\ti  \Box, \chi_{>2R_0}(x)] \psi(t) v_T} \\
        &\leq \sum_{j=0}^{\infty} \nm{\<x\>^{1/2} [\ti  \Box, \chi_{>2R_0}(x)] \psi(t) v_T}_{L^2_t L^2_x([0,T] \times A_j)} \\
        &\leq C \nm{[\ti  \Box, \chi_{>2R_0}(x)] \psi(t) v_T}_{L^2_t L^2_x[0,T]}.
    \end{align}
    Again using that  $[\ti \Box, \chi_{>2R_0}]$ is a first order space-time differential operator with compact support in $\{2R_0 \leq |x|\leq 4R_0\}$, we have $\<x\> \leq C R_0$ on the set and so 
    \begin{align}
        \nm{[\ti  \Box, \chi_{>2R_0}(x)] \psi(t) v_T}_{L^2_t L^2_x[0,T]} \leq &C \left( \nm{\p v_T}_{L^2_t L^2_x[0,T]} + \nm{\<x\>^{-1} v_T}_{L^2_t L^2_x[0,T]}\right) \\
        & \leq C \LEoT{v_T}. \label{eq:case3Pvcomm1}
    \end{align}
    On the other hand $\chi_{>2R_0}[\ti \Box, \psi(t)] v_T$ is supported in $\{|x|<C(T+1)\}$.  Therefore using the definition of the $LE^*$ norm and the fact that $\<x\> \leq C 2^j$ on $A_j$ we  have
    \begin{align}
        \LEsltxT{\chi_{>2R_0}(x)[\ti \Box, \psi(t)] v_T} &\leq \LEST{\chi_{>2R_0}[\ti  \Box, \psi(t)] v_T} \\
        &\leq \sum_{j=0}^{\infty} \nm{\<x\>^{1/2} [\ti \Box, \psi(t)] v_T}_{L^2_t L^2_x([0,T] \times A_j)}\\
        &\leq C \sum_{j=0}^{\log_2(C(T+1)) }2^{j} \nm{\<x\>^{-1/2} [\ti \Box, \psi(t)] v_T}_{L^2_t L^2_x([0,T] \times A_j)} \\
        &\leq C (T+1) \sup_{j \geq 0} \nm{\<x\>^{-1/2} [\ti \Box, \psi(t)] v_T}_{L^2_t L^2_x([0,T] \times A_j)},
        \label{eq:case3Pvcomm2inter}
    \end{align}
    where we note that the $\log_2(C(T+1))$ comes from the number of $A_j =\{|x| \simeq 2^j\}$ contained in $\{|x| <C(R_0+T)\}$ and we use the geometric series partial sum formula to estimate 
    \begin{equation}
        \sum_{j=0}^{\log_2(C(T+1))} 2^j  = \frac{1-2^{\log_2(C(T+1))}}{1-2}\leq C(T+1).
    \end{equation}
    Recall that for $k=1,2$ we have $|\p_t^k \psi| \leq C T^{-k}$. Computing $[\ti \Box, \psi(t)]$ directly and using the asymptotic flatness of $\ti \Box$ we obtain 
    \begin{align}
        \LtxTAj{\<x\>^{-1/2} [\ti \Box, \psi(t)] v_T} \leq &T^{-1} \LtxTAj{\<x\>^{-1/2} \p v_T} \\
        &+ T^{-2} \LtxTAj{\<x\>^{-1/2} v_T}\\
        &+T^{-1} \LtxTAj{\<x\>^{-1/2} \<x\>^{-1} v_T}.
    \end{align}
    Then since $v_T$ is supported in $\{|x| < C(T+1)\}$ we can write $1=T\<x\>^{-1}$ in the second term to obtain
    \begin{align}
        \LtxTAj{\<x\>^{-1/2} \ti  [\Box, \psi(t)] v_T} \leq C& T^{-1} \LtxTAj{\<x\>^{-1/2} \p v_T} \\
        &+C T^{-1} \LtxTAj{\<x\>^{-1/2} \<x\>^{-1} v_T}.
    \end{align}
    Plugging this back into \eqref{eq:case3Pvcomm2inter} and applying the definition of $LE^1$ we have 
    \begin{align}
        \LEsltxT{\chi_{>2R_0}(x)[\ti \Box, \psi(t)] v_T} \leq C \LEoT{v_T}.\label{eq:case3Pvcomm2}
    \end{align}
    Then, combining together \eqref{eq:case3vLEo}, \eqref{eq:case3Pvcomm1}, and \eqref{eq:case3Pvcomm2}, and then applying \eqref{eq:case3vTLEo} we have 
    \begin{align}
        \LEoT{\chi_{>2R_0} \psi(t) v_T}& + \LEoT{v_T}  \\
        &\leq C \left( \LEST{\ti{\Box}(\chi_{>2R_0} \psi(t) v_T)} + \LEoT{v_T}\right) \\
         &\leq C \left(\LEsltxT{P u_1}+ \LEoT{v_T} \right) \label{eq:case3vLEopenultimate}\\
        &\leq       
        C \left( \ltwo{\p u_1 (T)} + \LEsltxT{P u_1}\right)\\
        &\leq C \left( \ltwo{\p u_1(0)} + \e^{-1} \LEsltxT{Pu_1} + \e \LEoT{u_1} \right), 
        \label{eq:case3vLEofinal}
    \end{align}
    where the final inequality follows by Lemma \ref{l:Linftyeasy} part 1.

    3) At this point it is also convenient for us to estimate the $LE^*$ norm of $P(\chi_{>2R_0}(x) \psi(t) v_T)$. Since $\ti \Box = \Box_g=P-a\p_t$ for $|x| > R_0$
    \begin{align}
        P(\chi_{>2R_0}(x) \psi(t) v_T) &= \ti{\Box} (\chi_{>2R_0}(x) \psi(t) v_T) + a \p_t (\chi_{>2R_0}(x) \psi(t) v_T)\\
        &=\ti{\Box} (\chi_{>2R_0}(x) \psi(t) v_T) + a \chi_{>2R_0}(x)\psi(t)\p_t v_T+ a \chi_{>2R_0}(x) (\p_t \psi(t))v_T . \label{eq:case3AFdamp1}
    \end{align}
    We can use \eqref{eq:case3vLEopenultimate} to control the first term, so we focus on the second and third.
    To estimate the second term we use the definition of $LE^*$, compute directly, then use the asymptotic flatness of $a$ and the definition of $LE^1$
    \begin{align}
        \LEST{ a \chi_{>2R_0}(x)\psi(t)\p_t v_T} &\leq \LEST{\<x\>^{-1} \<x\> a \p_t v_T} \\
        &\leq \sum_{j=0}^{\infty} \nm{ \<x\> a \<x\>^{-1/2} \p_t v_T}_{L^2_t L^2_x([0,T] \times A_j)}\\
        &\leq \sum_{j=0}^{\infty} \nm{\<x\> a}_{L^{\infty}_{t,x}([0,T] \times A_j)} \nm{\<x\>^{-1/2} \p_t v_T}_{L^2_t L^2_x([0,T] \times A_j)}\\
        &\leq \nm{\<x\> a}_{l^1_j L^{\infty}_{t,x}([0,T] \times A_j)} \sup_{j \geq 0}\nm{\<x\>^{-1/2} \p_t v_T}_{L^2_t L^2_x([0,T] \times A_j)} \\
        &\leq C \LET{\p_t v_T} \leq C \LEoT{v_T}. \label{eq:case3AFdamp2}
    \end{align}
    To estimate the third term, we use that $|\p \psi| \leq C T^{-1}$ and apply the definition of $LE^*$ and the asymptotic flatness of the damping $a$, computing as above to obtain
    \begin{align}
        \LEST{a \chi_{>2R_0}(x) (\p_t \psi(t))v_T} &\leq CT^{-1} \LEST{\<x\>^{-1} \<x\> a v_T}\\
        &\leq CT^{-1} \LET{\<x\>^{-1/2} v_T}.
    \end{align}
    Now since $v_T$ is supported in $\{|x| \leq CT\}$ we have $T^{-1} \leq C\<x\>^{-1}$. Applying this and the definition of $LE^1$ we have
    \begin{align}
        \LEST{a \chi_{>2R_0}(x) (\p_t \psi(t))v_T} &\leq C \LET{\<x\>^{-1} \<x\>^{-1/2} v_T} \\
        &\leq C\LEoT{v_T}\label{eq:case3AFdamp3}
    \end{align}
    So now applying \eqref{eq:case3vLEopenultimate}, \eqref{eq:case3AFdamp2}, and \eqref{eq:case3AFdamp3} to estimate the terms in \eqref{eq:case3AFdamp1}, and then applying \eqref{eq:case3vLEofinal} we have 
    \begin{align}
        \LEST{P(\chi_{>2R_0}(x) \psi(t) v_T} &\leq C \left( \LEsltxT{Pu_1} + \LEoT{v_T} \right) \\
        &\leq C \left( \ltwo{\p u_1(0)} + \e^{-1} \LEsltxT{Pu_1} + \e \LEoT{u_1} \right). \label{eq:case3AFdampfinal}
    \end{align}
    4) Now we estimate the $LE^1[0,T]$ norm of $w$. To do so we will apply our assumed estimate, which requires  that $w$ has zero Cauchy data at $t=0$ and $t=T$, and that $Pw$ is compactly supported and in $LE^*$. 
    To see the Cauchy data is trivial, note that by the construction of $\chi_{[k,k+1]}(t)$ and $\psi(t)$ 
    \begin{align} 
        w[0]&=u_1[0]-\sum_{k=0}^N (\chi_{[k,k+1]} w_k)[0]-(\chi_{>2R_0} \psi v_T)[0]=u_1[0]-w_0[0]=0\\
        w[T]&=u_1[T]-\sum_{k=0}^N (\chi_{[k,k+1]} w_k)[T]-(\chi_{>2R_0}\psi v_T)[T] \\
        &=u_1[T]-\chi_{<2R_0}u_1[T]-\chi_{>2R_0} u_1[T]=0.
    \end{align}
    To see that $Pw$ is compactly supported in $\{|x| \leq C(T+1)\}$ for $t \in [0,T]$, note that by the construction of the $w_k$ and $\chi_k$
    \begin{align}
        Pw&= Pu_1- \sum_{k=0}^N \chi_{[k,k+1]}(t)P w_k - \sum_{k=0}^N [P,\chi_{[k,k+1]}(t)]w_k - P( \chi_{>2R_0}(x) \psi(t)  v_T) \\
        &=Pu_1- \sum_{k=0}^N \chi_{[k,k+1]}(t) \mathbbm{1}_{[k,k+1]}(t) Pu_1- \sum_{k=0}^N [P,\chi_{[k,k+1]}(t)]w_k - P( \chi_{>2R_0}(x) \psi(t)  v_T)\\
        &=- \sum_{k=0}^N [P,\chi_{[k,k+1]}(t)]w_k- P (\chi_{>2R_0}(x) \psi(t)  v_T). \label{eq:case2Pwdef}
    \end{align}
    Since the $w_k$ and $v_T$ are all compactly supported in $\{|x| \leq C(T+1)\}$, so is the right hand side. We now apply our assumed estimate to $w$ and obtain 
    \begin{align}
        \LEoT{w}&\leq C \left( \vartheta \LET{\<x\>^{-2} w} + \LEST{Pw} \right). \label{eq:case2assumptionTriangle}
    \end{align}
    To estimate the first term on the right hand side we use the triangle inequality, that $\<x\>^{-1} \leq 1$, and Lemma \ref{l:case3pre} parts 2 and 3, to see that for any $\e>0$ we have
    \begin{align}
        \LET{\<x\>^{-2} w} &\leq \LET{\<x\>^{-2} (u_1-\sum_k \chi_{[k,k+1]} w_k)} \\
        &\leq \LET{\<x\>^{-2}u_1} + \sum_k \LET{\<x\>^{-2} \chi_{[k,k+1]}  w_k} \\
        &\leq \LET{\<x\>^{-2}u_1}+ \sum_k \nm{w_k}_{LE^1[k-\alpha,k+1+\alpha]}\\
        &\leq \LET{\<x\>^{-2}u_1}+ C \sum_k \nm{\p w_k}_{L^{\infty}_{t}L^2_x[k-\alpha, k+1+\alpha]}\\
        & \leq \LET{\<x\>^{-2}u_1} + C \left( \ltwo{\p u_1(0)} + \e^{-1} \LEsltxT{Pu_1} + \e \LEoT{u_1} \right).
        \label{eq:case2x2final}
    \end{align}
    To estimate the $LE^*$ norm of $Pw$ we apply the triangle inequality and \eqref{eq:case2Pwdef}
    \begin{align}
        \LEST{Pw} &\leq \sum_{k=0}^N \LEST{[P,\chi_{[k,k+1]}(t)]w_k} + \LEST{P (\chi_{>2R_0}(x) \psi(t)  v_T)}.
    \end{align}
    We control the sum using Lemma \ref{l:case3pre} part 4 and we control the $v_T$ term using \eqref{eq:case3AFdampfinal}.
    Then for any $\e>0$ we have 
    \begin{equation}\label{eq:case2Pwfinal}
        \LEST{Pw}\leq C \left( \ltwo{\p u_1(0)} + \e^{-1} \LEsltxT{Pu_1} +\e \LEoT{u_1} \right).
    \end{equation}
    Combining \eqref{eq:case2x2final} and \eqref{eq:case2Pwfinal} we have for any $\e>0$
    \begin{equation}\label{eq:case2assumptionfinal}
        \LEoT{w} \leq C\left( \ltwo{\p u_1(0)} + \vartheta \LET{\<x\>^{-2} u_1} + \e^{-1} \LEsltxT{Pu_1} + \e \LEoT{u_1} \right).
    \end{equation}
    5) Now applying the triangle inequality, \eqref{eq:case3vLEofinal}, \eqref{eq:case2assumptionfinal}, and Lemma \ref{l:case3pre} parts 2 and 3, we have for all $\e>0$
    \begin{align}\label{eq:case2u1final}
        \LEoT{u_1} &\leq \LEoT{w} + \LEoT{\sum_k \chi_{[k,k+1]}(t) w_k} + \LEoT{\chi_{>2R_0}(x) \psi(t) v_T}\\
        &\leq C\left( \ltwo{\p u_1(0)} + \vartheta \LET{\<x\>^{-2} u_1} + \e^{-1} \LEsltxT{Pu_1} + \e \LEoT{u_1} \right).
    \end{align}
    Choosing $\e>0$ small enough, we can absorb the final term back into the left hand side and we are left with the desired inequality. 
\end{proof}

\subsection{Reduction to uniformly compactly supported solutions}
Before proceeding with the final case reduction, we quote an exterior estimate \cite[Proposition 3.2]{MST20}.
\begin{proposition}\label{prop:exteriorCase}
     If $P$ is asymptotically flat and $R \geq R_0$, then there exists $C>0$, such that for all $T>0$
     \begin{align}
         \nm{u}_{LE^1([0,T] \times \{R<|x|\})} \leq C \bigg( &\nm{\p u(0)}_{L^2) [0,T] \times \{R<|x|\})} + \nm{\p u(T)}_{L^2 ([0,T] \times \{R<|x|\})} \\
         &+ R^{-1} \nm{u}_{LE([0,T] \times \{\frac{R}{2}<|x|<2R\})} + \nm{Pu}_{LE^*([0,T] \times \{R<|x|\})}\bigg).
     \end{align}
\end{proposition}

We make use of this proposition to prove the final case reduction. Namely, that it suffices to consider solutions which are supported within $\{|x|\leq 2R_0\}$ for all $t \in [0,T]$. We only apply this final lemma in the proof of Proposition \ref{prop:CaseReduction} and so we take $\vartheta=1$.
\begin{lemma}\label{l:case4}
     Assume that there exists $C>0$, such that for all $T>0$, and $u_2(t)$ supported in $\{|x| \leq 2R_0\}$ for $t \in [0,T]$, with $u_2[0]=u_2[T]=0$, $Pu_2 \in LE^*_c$,  we have 
    \begin{equation}
        \LEoT{u_2} \leq C \left( \LET{\<x\>^{-2} u_2 } + \LEST{Pu_2} \right).
    \end{equation}
    Then there exists $C>0$, such that for all $T>0$, all $w$ with $w[0]=w[T]=0$, and $Pw \in LE^*_c$ we have 
    \begin{equation}
        \LEoT{w} \leq C \left( \LET{\<x\>^{-2} w} + \LEST{Pw} \right).
    \end{equation}
\end{lemma}
Before proceeding with the proof we point out that Proposition \ref{prop:CaseReduction} is a direct consequence of successively applying Lemmas \ref{l:case1}, \ref{l:case2}, \ref{l:case3}, and \ref{l:case4} with $\vartheta=1$.
\begin{proof}
    Consider $w$ with $w[0]=w[T]=0$, and $Pw \in LE_c^*$, then write $w = \chilro w + \chigro w$. We will separately estimate these terms using the assumption and Proposition \ref{prop:exteriorCase}. 

    First, note that $\chilro w$ is supported in $\{|x| \leq 2R_0\}$ and satisfies $\chilro w[0] = \chilro w[T]=0$. Furthermore 
    \begin{equation}
        P \chi_{<R_0} w = \chi_{<R_0} Pw + [P, \chi_{<R_0}] w
    \end{equation}
    is compactly supported in $\{|x| \leq 2R_0\}$.  
    To estimate $P \chilro w$ in $LE^*$ we first compute directly via the triangle inequality
    \begin{align}
        \LEST{P \chilro w} \leq \LEST{Pw} + \LEST{[P,\chilro]w}. \label{eq:case4Pchilro1}
    \end{align}
    Now note $[P, \chilro]$ is a first order space-time differential operator with smooth coefficients, compactly supported in $\{R_0 \leq |x| \leq 2R_0\}$. 
    Then using the definition of $LE^*$ and the compact spatial support of $\chi_{R_0<x<2R_0}$, we have  
    \begin{align}
        \LEST{[P,\chilro] w} &\leq C \left( \LEST{\chi_{R_0<x<2R_0} \p w} + \LEST{\chi_{R_0<x<2R_0} w} \right) \\
        &\leq C\left(\LET{\chi_{R_0<x<2R_0} \p w} + \LET{\chi_{R_0<x<2R_0} \<x\>^{-1} w} \right) \\
        &\leq C \nm{w}_{LE^1[0,T]\times \{R_0 \leq |x| \leq 2R_0\}}.
    \end{align}
    Now we apply Proposition \ref{prop:exteriorCase}, that $w[0]=w[T]=0,$ and use $\<x\>^{-2} \geq C$ on $|x| \leq 2R_0$, to estimate this $LE^1$ norm and obtain
    \begin{align}
        \LEST{[P,\chilro] w} 
        &\leq C \left( R_0^{-1} \nm{w}_{LE[0,T] \times \{\frac{R_0}{2} < |x| < 2R_0\}} + \nm{Pw}_{LE^*[0,T]\times \{R_0<|x|\}} \right) \\
        &\leq C \left( \LET{\<x\>^{-2} w}+ \LEST{Pw}\right).
    \end{align}
    Plugging this back into \eqref{eq:case4Pchilro1}, we have 
    \begin{equation}\label{eq:case4Pchilro2}
        \LEST{P \chilro w} \leq C \left( \LET{\<x\>^{-2} w} + \LEST{Pw}\right).
    \end{equation}
    Thus $P \chilro w \in LE^*$. Therefore we can apply our assumption and \eqref{eq:case4Pchilro2} to estimate 
    \begin{align}
        \LEoT{\chilro w} &\leq C \left( \LET{\<x\>^{-2} \chilro w} + \LEST{P \chilro w} \right) \\
        &\leq C\left(\LET{\<x\>^{-2} w} + \LEST{Pw}\right). \label{eq:case4smallR0}
    \end{align}
    
    We now estimate the $LE^1$ norm of $\chigro w$. Using Proposition \ref{prop:exteriorCase}, again noting $w[0]= w[T]=0$ and that $\<x\>^{-2} \geq C$ on $|x| \leq 2R_0$, we have
    \begin{align}
        \LEoT{\chigro w} \leq C \bigg(& R_0^{-1} \nm{\chigro w}_{LE[0,T] \times \{\frac{R_0}{2}<|x|<2R_0\}} \\
        &+ \nm{\chigro Pw}_{LE^*[0,T]\times\{R_0<|x|\}} \\&+\nm{[P,\chigro] w}_{LE^*[0,T]\times\{R_0<|x|\}}\bigg) \\
        \leq C \bigg( &\LET{\<x\>^{-2} w} + \LEST{Pw} +\nm{[P,\chigro] w}_{LE^*[0,T]} \bigg).
    \end{align}
    We can estimate $[P,\chigro]$ exactly as we estimated $[P,\chilro]$ to see 
    \begin{equation}
        \LEST{[P,\chigro] w} \leq C \left( \LET{\<x\>^{-2}w} + \LEST{Pw}\right).
    \end{equation}
    Therefore 
    \begin{equation}
        \LEoT{\chigro w} \leq C \left( \LET{\<x\>^{-2} w} + \LEST{Pw} \right).
    \end{equation}
    Combining this with \eqref{eq:case4smallR0} we obtain
    \begin{align}
        \LEoT{w} &\leq \LEoT{\chilro w} + \LEoT{\chigro} \\
        &\leq C \left( \LET{\<x\>^{-2} w} + \LEST{Pw} \right)
    \end{align}
    which is exactly the desired conclusion.
\end{proof}

\section{Propagation Argument}\label{s:propagation}
After applying the case reduction of Proposition \ref{prop:CaseReduction}, we arrive at.
\begin{proposition}\label{prop:main}
    To prove Theorem \ref{thm:highfreq}, it is enough to prove that there exists $C>0$, such that for all $T>0$ and $v(t)$ supported in $\{|x|\leq 2 R_0\}$ for all $t \in [0,T]$, with $v[0]=v[T]=0$, and $Pv \in LE^*_c$, we have
    \begin{equation}
        \nm{v}_{LE^1[0,T]} \leq C \left( \nm{v}_{L^2_t L^2_x[0,T]} + \nm{Pv}_{LE^*[0,T]} \right).
    \end{equation}
\end{proposition}
    We now prove this estimate using the escape function constructed in Section \ref{s:escapeFunction}. To begin, consider $v(t)$ supported in $\{|x|\leq 2R_0\}$ for all $t \in [0,T]$, and with $v[0]=v[T]=0$, and $Pv \in LE^*_c$.
    
    We extend $v$ by $0$ outside of $[0,T]$. Because of this, we have 
    \begin{align}
        \LEoT{v} = \LEo{v}, \quad \nm{v}_{L^2_tL^2_x[0,T]}=\Ltx{v}, \quad \LEST{Pv} = \LEs{Pv}.
    \end{align}
    Recall the cutoff notation from Section \ref{s:cutoffdef} and the definition of Weyl Quantization in Definition \ref{def:Pseudo}. For $\lambda \geq 1$, we define 
    \begin{equation}
        \vll = \chi_{|\xi|< \lambda}^w v, \quad \vgl = \chi_{|\xi|> \lambda}^w v, \quad \text{ so } v=\vll+\vgl.
    \end{equation}
    We estimate $v$ by estimating these low and high frequency pieces separately. This general approach is standard, see \cite[Section 2.6]{Kofroth23} and \cite[Section 4]{MST20}. However, working on $T^* \Rb^4$, rather than $T^* \Rb^3$ as in \cite{Kofroth23}, requires an additional step in the high frequency argument (Lemma \ref{l:highlow}). We also have some additional technicalities to handle $g^{00} \not \equiv -1$, and some additional error terms in our positive commutator argument due to the time-dependence of $a$.

    First, the low frequency estimate. 
    \begin{lemma}\label{l:lowfreq}
        There exists $C>0$, such that for all $\sigma \geq 1$, $T>0$, $\lambda \geq 1$ and $v$ supported in $\{|x| \leq 2R_0\}$ with $v[0]=v[T]=0$,
        \begin{align}
            \LEo{\vll} &\leq C \left( \sigma \lambda \Ltx{v}+ \frac{1}{ \sigma \lambda} \LEs{Pv} + \frac{1}{\sigma} \LEo{v} \right).
        \end{align}
    \end{lemma}
    Second, the high frequency estimate.
    \begin{lemma}\label{l:highfreq}
        There exist $C>0,$ and  $C(\lambda)>0$, depending on $\lambda$, such that for all $T>0$, $\e>0$, $\lambda \geq 1$ and $v$ supported in $\{|x| \leq 2R_0\}$ with $v[0]=v[T]=0$
        \begin{align}
        \nm{v_{>\lambda}}_{LE^1_{\leq 2R_0}} \leq C \left( C(\lambda) \Ltx{v} + \left(\frac{1}{\lambda} + \frac{1}{\e}\right) \LEs{Pv}+\left(\frac{1}{\lambda^{1/4}}  + \e  \right) \LEo{v} \right).
        \end{align}
    \end{lemma}
    Before proving these Lemmas, we show how they combine to prove Proposition \ref{prop:main} and thus Theorem \ref{thm:highfreq}. 
    \begin{proof}[Proof of Proposition \ref{prop:main}]
    Since $v$ is supported in $\{|x| \leq 2R_0\}$ and by the triangle inequality
    \begin{equation}
        \LEo{v}  = \nm{v}_{LE^1_{\leq 2R_0}} \leq \nm{\vgl}_{LE^1_{\leq 2R_0}}+\LEo{\vll}.
    \end{equation}
    Then by Lemmas \ref{l:lowfreq} and \ref{l:highfreq}, for all $\e>0$
    \begin{align}
        \LEo{v} \leq C \bigg( &(\sigma\lambda+C(\lambda)) \Ltx{v} + \left(\frac{1}{\sigma\lambda} + \frac{1}{\lambda} + \frac{1}{\e}\right) \LEs{Pv} + \left(\frac{1}{\sigma}+\frac{1}{\lambda^{\frac{1}{4}}}+\e  \right) \LEo{v}\bigg).
    \end{align}
    Now we absorb the $\LEo{v}$ terms back into the left hand side, by taking $\lambda$ and $\sigma$ large enough, and $\e$ small enough, obtaining the desired inequality.
    \end{proof}
    Before proceeding with the proofs, we state some basic facts relating norms and two useful consequences of Plancherel's theorem relating norms under various frequency cutoffs.
    \begin{lemma} \label{l:plancherel}
    \begin{enumerate}
        \item For all $u \in L^2_t L^2_x(\Rb \times \Rb^3)$
    \begin{equation}
        \LE{u} \leq \Ltx{u},
    \end{equation}
        and for all $f \in LE^*$
    \begin{equation}
        \Ltx{f}\leq \LEs{f}.
    \end{equation}
    Furthermore there exists $C>0$, such that for all $u$ supported in $\{|x| \leq 2R_0\}$, then 
    \begin{equation}
        \Ltx{u} \leq C \LE{u}.
    \end{equation}
    
    \item There exists $C>0$, such that for all $v$ supported in $\{|x|\leq 2 R_0\}$
        \begin{align}
            &\LE{\<x\>^{-1} v} + \LE{\<x\>^{-1} \vgl}+\LE{\<x\>^{-1} \vll} \leq C \Ltx{v}.
        \end{align}
    \item There exists $C>0$, such that for all $v$ supported in $\{|x| \leq 2R_0\}$, $\vggl=\chi_{|\xi|+|\tau| >\frac{\lambda}{2}}^w v$, and any $\alpha \in \Rb$ we have
        \begin{align}
            &\nm{\vggl}_{H^{1-\alpha}_{t,x}} \leq C\lambda^{-\alpha} \LEo{v}.
        \end{align}
    \end{enumerate}
    \end{lemma}
    \begin{proof}
    1) Since $\<x\>^{-1/2} \leq 1$, by definition of $LE$ we have
    \begin{equation}
        \LE{u} = \sup_{j \geq 0} \nm{\<x\>^{-1/2} u }_{L^2_t L^2_x(\Rb_+\times A_j)} \leq \sup_{j \geq 0} \nm{u }_{L^2_t L^2_x(\Rb_+\times A_j)} \leq \Ltx{u}.
    \end{equation}
    Similarly, we have $\<x\>^{1/2} > 1$, and so
    \begin{align}
        \Ltx{f}\leq \sum_{j=0}^{\infty} \nm{f}_{L^2_t L^2_x(\Rb_+ \times A_j)} \leq \sum_{j=0}^{\infty} \nm{\<x\>^{1/2}f}_{L^2_t L^2_x(\Rb_+ \times A_j)} =\LEs{f}.
    \end{align}
    When $u$ is supported in $\{|x|\leq 2R_0\}$, there exists $c>0$ such that $\<x\>^{-1/2} \geq c$ in $\{|x| \leq 2R_0\}$ and so 
    \begin{align}
        \Ltx{u} &\leq \sum_{j=0}^{N} \nm{u}_{L^2_t L^2_x(\Rb_+ \times A_j)} \leq \frac{1}c \sum_{j=0}^N \nm{\<x\>^{-1/2}u}_{L^2_t L^2_x(\Rb_+ \times A_j)} \\
        &\leq C \sup_{j \geq 0} \nm{\<x\>^{-1/2} u}_{L^2_t L^2_x(\Rb_+ \times A_j)}= C\LE{u}.
    \end{align}
    
        2) By definition of $LE$, since $\<x\>^{-1} \leq 1$
        \begin{equation}
            \LE{\<x\>^{-1} v} =\sup_{j \geq 0} \nm{\<x\>^{-3/2}v}_{L^2_t L^2_x(\Rb \times A_j)} \leq \Ltx{v}.
        \end{equation}
        To control the second term, apply the first, and then Plancherel's theorem in $(t,x)$ 
        \begin{equation}
            \LE{\<x\>^{-1} \vgl} \leq \Ltx{\vgl}= \Ltauxi{\chi_{>\lambda}(\xi) \hat{v}(\tau, \xi)} \leq \Ltauxi{\hat{v}}=\Ltx{v}.
        \end{equation}
        An analogous proof controls the third term.

        3) Again applying Plancherel's theorem 
        \begin{align}
            \nm{\vggl}_{H^{1-\alpha}_{t,x}} &\leq \Ltauxi{\<(\tau,\xi)\>^{1-\alpha} \chiggl \hat{v}} \leq C \lambda^{-\alpha}\Ltauxi{\<(\tau,\xi)\> \chiggl \hat{v}} \\
            &\leq C \lambda^{-\alpha}\Ltauxi{\<(\tau,\xi)\> \hat{v}} \leq C \lambda^{-\alpha} \Ltx{\p v} \leq C \lambda^{-\alpha} \LE{\p v}\leq C \lambda^{-\alpha} \LEo{v},
        \end{align}
        where the second to last inequality follows by part 1) by the compact support of $v$.
        
    \end{proof}

    \subsection{Proof of low frequency propagation estimate: Lemma \ref{l:lowfreq}}
    We now prove the low frequency propagation estimate. We do so by considering high and low $\tau$ frequencies separately. It is straightforward to estimate the low $\tau$ term using Plancherel's theorem. Estimating the high $\tau$ term uses microlocal analysis and takes up most of the proof. 
    \begin{proof}
        [Proof of Lemma \ref{l:lowfreq}]
        If we assume
        \begin{align}\label{eq:lowFreqintermed}
            \LE{\p \vll} &\leq C\left( \sigma \lambda \Ltx{v} + \frac{1}{ \sigma\lambda} \LEs{Pv} + \frac{1}{\sigma} \LEo{v} \right),\\
        \end{align}
        the desired inequality is an immediate consequence of the definition
    \begin{equation}
        \LEo{\vll} = \LE{\p \vll} + \LE{\<x\>^{-1} \vll}
    \end{equation}
        and Lemma \ref{l:plancherel} part 2 applied to $\LE{\<x\>^{-1} \vll}$. 
        
        So, it remains to prove \eqref{eq:lowFreqintermed}. To do so, we write for $\sigma \geq 1$
        \begin{align}
            \vlsl = \chi^w_{|\xi|<\lambda}\chi^w_{|\tau|<\sigma \lambda}v,
            \qquad\vgll = \chi^w_{|\xi|<\lambda}\chi^w_{|\tau|>\sigma\lambda}v, \qquad \vll = \vlsl  +\vgll .
        \end{align}
 By Lemma \ref{l:plancherel} part 1, and Plancherel's theorem, there exists $C>0$ such that 
        \begin{align}
        \begin{split}
            \LE{\p \vlsl}&\leq C\Ltx{\p \vlsl} 
            \\
            &\leq C \Ltauxi{(|\tau|+|\xi|) \chi_{|\xi|<\lambda}\chi_{|\tau|<\sigma\lambda} \hat{v}} \leq C \sigma \lambda \Ltx{v}.\label{eq:lowlow}
        \end{split}
        \end{align}
        On the other hand, again by Lemma \ref{l:plancherel} part 1, and Plancherel's theorem, there exists $C>0$ such that 
        \begin{align}
            \LE{\p \vgll} &\leq C\Ltx{\p \vgll} \leq C \Ltauxi{(|\tau|+|\xi|) \chi_{|\xi|<\lambda} \chi_{|\tau|>\sigma\lambda} \hat{v}}\\
            &\leq C \Ltauxi{|\xi| \chi_{|\xi|< \lambda} \hat{v}}+C\Ltauxi{\frac{\tau^2}{\sigma \lambda} \chi_{|\xi|<\lambda}\chi_{|\tau|>\sigma \lambda} \hat{v}}\\
            &\leq C \lambda \Ltx{v}+\frac{C}{\sigma\lambda} \Ltx{\chi_{|\xi|<\lambda}^w (\p_t^2v)}, \label{eq:lowhigheq}
        \end{align}
        noting that $\tau, \chi_{|\xi|< \lambda},$ and $\chi_{|\tau|>\sigma \lambda}$ are all Fourier multipliers, so they commute with each other. It remains to estimate the second term on the right hand side. To do so we write $\p_t^2 v$ in terms of $Pv$, and then estimate error terms. 

        To proceed we first write
        \begin{align}
            \Ltx{\chi_{|\xi| < \lambda}^w (\p_t^2 v)}&=\Ltx{\chi_{|\xi|<\lambda}^w \left(\frac{-g^{00}}{-g^{00}} \p_t^2 v\right)} \\
            &\leq C \Ltx{(g^{00})^{-1}\chi_{|\xi|<\lambda}^w \left(g^{00}D_t^2 v\right)} + \Ltx{[\chi_{|\xi| <\lambda}^w, (g^{00})^{-1}] (g^{00}D_t^2 v)}\\
            &\leq C \Ltx{\chi_{|\xi|<\lambda}^w \left(g^{00}D_t^2 v\right)} + \Ltx{[\chi_{|\xi| <\lambda}^w, (g^{00})^{-1}] \chi_{|\xi| <2\lambda}^w (g^{00}D_t^2 v)} \\
            & \quad + \Ltx{R_{-\infty} (g^{00} D_t^2 v)}
        \end{align}
        where to obtain the second inequality we used $g^{00}$ is bounded, and wrote $1= 1-\chi^w_{|\xi|<2 \lambda}+\chi^w_{|\xi|<2 \lambda}$ and $R_{-\infty} = [\chi_{|\xi| <\lambda}^w, (g^{00})^{-1}] (1-\chi^w_{|\xi| < 2\lambda})$. Note that $R_{-\infty} \in \Psi^{-\infty}$ by Proposition \ref{prop:pseudoCalc}, because the principal symbols of the commutator and $(1-\chi^w_{|\xi|<2\lambda})$ have non-overlapping support.
        To estimate this error term, we can commute $g^{00}$ with $D_t^2$ because $g$ does not depend on $t$, and then use that $R_{-\infty} D_t^2 \in \Psi^{-\infty}$ is bounded on $L^2$ by Proposition \ref{prop:pseudoBounded} to obtain 
        \begin{equation}
            \Ltx{R_{-\infty} (g^{00} D_t^2 v) } = \Ltx{R_{-\infty} D_t^2 g^{00}v } \leq C\Ltx{g^{00} v} \leq C \Ltx{v}. 
        \end{equation}
        Note that since $g^{00} \geq -C$ and by asymptotic flatness, $(g^{00})^{-1} \in S^0(\Trf)$, so by Proposition \ref{prop:pseudoCalc} we have $[\chi^w_{|\xi|<\lambda}, (g^{00})^{-1}] \in \Psi^{-1}(\Rb^4)$. Furthermore, by Proposition \ref{prop:pseudoBounded} the commutator is bounded on $L^2$, therefore we have 
        \begin{equation}\label{eq:lowdteq0}
            \Ltx{\chi_{|\xi| < \lambda}^w (\p_t^2 v)} \leq C \left( \Ltx{\chi_{|\xi|<2 \lambda}^w \left(g^{00}D_t^2 v\right)} + \Ltx{v} \right).
        \end{equation}        
        Now we use that $P=D_{\alpha} g^{\alpha \beta} D_{\beta} + i a D_t$, the triangle inequality, and that $g$ commutes with $D_t$ since it does not depend on $t$, to write
        \begin{align}
            \Ltx{\chi_{|\xi|<2\lambda}^w (g^{00} D_t^2v)} &\leq \Ltx{\chi_{|\xi|<2\lambda}^w(Pv)}+ \Ltx{\chi_{|\xi|<2\lambda}^w((g^{0j}D_j+D_j g^{0j}) D_t v)} \\
            &\quad + \Ltx{\chi_{|\xi|<2\lambda}^w(D_i g^{ij}D_j v)} + \Ltx{\chi_{|\xi|<2\lambda}^w(aD_t v)} \label{eq:lowdteq}.
        \end{align}
        Using that the frequency cutoff is $L^2$ bounded, and applying Lemma \ref{l:plancherel} part 1 we have
        \begin{equation}
            \Ltx{\chi_{|\xi|<2\lambda}^w (Pv)} \leq \Ltx{Pv} \leq \LEs{Pv} \label{eq:lowP}.
        \end{equation}
        Arguing in the same way and using that $v$ is supported in $\{|x| \leq 2R_0\}$ to apply Lemma \ref{l:plancherel} part 1, we have
        \begin{equation}
            \Ltx{\chi_{|\xi|<2\lambda}^w (aD_tv)} \leq \Ltx{aD_t v} \leq \lp{a}{\infty} \Ltx{\p v} \leq C \LE{\p v} \leq C \LEo{v} .\label{eq:lowdamp}
        \end{equation}
        In order to estimate the metric terms, we would like to use the frequency cutoff to $|\xi|<2\lambda$ to control the spatial derivatives $D_j$. To do so we must commute $\chi^w_{|\xi|<2\lambda}$ and the metric. First, note that $g^{\alpha j}, (D_j g^{\alpha j}) \in S^0(\Trf)$ for all $\alpha \in \{0,1,2,3\}$ and $j \in {1,2,3}$. Therefore by Proposition \ref{prop:pseudoCalc}
        \begin{equation}
            [\chi^w_{|\xi|<2 \lambda}, g^{\alpha j}] \in \Psi^{-1}(\Rb^4), \quad [\chi^w_{|\xi|<2\lambda}, (D_j g^{\alpha j})] \in \Psi^{-1}(\Rb^4).
        \end{equation}
        Furthermore, by Proposition \ref{prop:pseudoBounded}, these commutators are bounded on $L^2_t L^2_x$ with constants independent of $\lambda$
        
        So using Proposition \ref{prop:pseudoCalc} to commute $\chi^w_{|\xi|<\lambda}$ with the metric in $D_i g^{ij} D_j$,  we have that there exists $R_{-2} \in \Psi^{-2}$ such that 
        \begin{align}
            &\Ltx{\chi^w_{|\xi|<2\lambda} ( (  ( D_i g^{ij}) D_j + g^{ij} D_i D_j)v) }\\
                        \leq& \Ltx{(D_ig^{ij}) \chi^w_{|\xi|< 2\lambda}(D_j v)} 
                        + \Ltx{g^{ij}\chi^w_{|\xi|< 2\lambda}(D_i D_j v)}\\
            &+ \Ltx{[\chi^w_{|\xi|<2\lambda}, D_i g^{i j}] \chi^w_{|\xi|< 4\lambda} (D_j v)}+\Ltx{[\chi^w_{|\xi|<2\lambda}, g^{i j}]\chi^w_{|\xi|< 4\lambda} (D_i D_j v)}\\
            &+\Ltx{R_{-2} D_j v} + \Ltx{R_{-2} D_i D_j v},
        \end{align}
        where the additional $\chi^w_{|\xi|< 4\lambda}$ can be multiplied to the right of the commutators, because the principal symbol of the commutator is supported on $|\xi|<4\lambda$. Note also by Proposition \ref{prop:pseudoBounded} $R_{-2} D_j$ and $R_{-2} D_i D_j$ are bounded on $L^2_t L^2_x$. Therefore 
        \begin{align}
            \Ltx{\chi_{|\xi|<2\lambda}^w (D_i g^{ij} D_j v)} &\leq C \left( \Ltx{\chi^w_{|\xi|< 4\lambda}(D_j v)}+ \Ltx{\chi^w_{|\xi|< 4\lambda}(D_i D_jv)}+\Ltx{v}\right) \\
            &\leq C \left(\Ltauxi{\xi \chi_{|\xi|<4\lambda} \hat{v}}+ \Ltauxi{|\xi|^2 \chi_{|\xi|<4\lambda} \hat{v}} + \Ltx{v} \right)\\
            &\leq C\left(1 + \lambda +  \lambda^2 \right) \Ltx{v} \leq  C \lambda^2 \Ltx{v}. \label{eq:lowcom1}
        \end{align}
        Similarly, using Proposition \ref{prop:pseudoCalc} to commute $\chi_{|\xi|<\lambda}^w$ with the metric in $g^{0j}D_j +D_j g^{0j}$, there exists $R_{-2} \in \Psi^{-2}$ such that 
        \begin{align}
            \Ltx{\chi_{|\xi|<2\lambda}^w((g^{0j}D_j +D_j g^{0j}) D_tv)} &\leq \Ltx{D_j g^{0j}\chi^w_{|\xi|< 2\lambda}(D_t v)} + \Ltx{g^{0j} \chi^w_{|\xi|< 2\lambda}(D_j D_t v)}\\
            &+ \Ltx{[\chi_{|\xi|<2\lambda}^w, D_j g^{0j}]D_t v}+\Ltx{[\chi^w_{|\xi|<2\lambda}, g^{0j}]\chi^w_{|\xi|< 4\lambda}(D_j D_tv)}\\
            &+\Ltx{R_{-2} D_t v}+\Ltx{R_{-2} D_j D_t v}.
        \end{align}
        Where again, the additional $\chi^w_{|\xi|<4\lambda} $ can be multiplied to the right of the commutator, because the principal symbol is supported on $|\xi| <4\lambda$. Note also, by Proposition \ref{prop:pseudoBounded} $R_{-2} D_t$ and $R_{-2} D_j D_t$ are bounded on $L^2_t L^2_x$.

        Using this and the $L^2$ boundedness of the commutators and $D_j g^{0j}, g^{0j}$ we have
        \begin{align}
            \Ltx{\chi_{|\xi|<2\lambda}^w ((g^{0j}D_j +D_j g^{0j}) D_tv} &\leq C \bigg( \Ltx{D_t v}+ \Ltx{\chi^w_{|\xi|< 4\lambda}(D_j D_t v)} +\Ltx{v} \bigg)\\
            &\leq C \bigg( \Ltx{\p v} + \Ltauxi{ \xi \chi_{|\xi| < 4 \lambda}\widehat{(D_t v)}}+\Ltx{v}\bigg) \\
            &\leq C \bigg(\lambda\Ltx{\p v} + \Ltx{v} \bigg) \leq C\lambda \LE{\p v}+C \Ltx{v} \label{eq:lowcom2},
        \end{align}
        where in the final inequality we used Lemma \ref{l:plancherel} part 1 and that $v$ is supported in $\{|x| \leq 2R_0\}$.
        
        Now applying \eqref{eq:lowP}, \eqref{eq:lowdamp}, \eqref{eq:lowcom1}, and \eqref{eq:lowcom2} to \eqref{eq:lowdteq0} and \eqref{eq:lowdteq}, we have 
        \begin{equation}
            \Ltx{\chi_{|\xi|<\lambda}^w (\p_t^2v)} \leq C \lambda^2 \Ltx{v}+ \LEs{Pv}+ C \lambda \LE{\p v} .
        \end{equation}
        Plugging this into \eqref{eq:lowhigheq} gives
        \begin{equation}
            \LE{\p \vgll} \leq  C\lambda \Ltx{v} + \frac{C}{\sigma \lambda} \LEs{Pv} + \frac{C}{\sigma}\LE{\p v}.
        \end{equation}
        Combining this with \eqref{eq:lowlow}, the full low frequency contribution is 
        \begin{align}
            \LE{\p v_{<\lambda}} &\leq \LE{\p \vgll} + \LE{\p \vlsl} \\
            &\leq C\left( \sigma \lambda \Ltx{v} + (\sigma\lambda)^{-1} \LEs{Pv} + \sigma^{-1}\LEo{v}\right),
        \end{align}
        as desired. 
        
    \end{proof}
    \subsection{Proof of high frequency estimate: Lemma \ref{l:highfreq}}
    We now prove the high frequency estimate. We begin by explaining the positive commutator approach we use. 
    Let $q \in S^1, m \in S^0$ and $\kappa \geq 1$ be as in Lemma \ref{p:EscapeFunction}.
     Then, letting $Q=q^w-\frac{i}{2}m^w \in \Psi^1(\Rb^4)$ we compute $P^*Q-Q^*P$ in two different ways. First, using adjoints and complex conjugates 
    \begin{equation}
        \<i(P^*Q-Q^*P)v,v\>_{L^2_t L^2_x}=2\Im\<Pv,Qv\>_{L^2_t L^2_x} = 2 \Im \int_{\Rb^4} Pv \overline{Qv} dt dx.
    \end{equation}
    Now, note that by Proposition \ref{prop:pseudoAdjoint} 
    \begin{align}
          Q^*=q^{w}+\frac{i}{2} m^w,\quad P^* = \Box_g-i \kappa a D_t - i \kappa (D_t a).
    \end{align}
    Using these to compute $P^* Q-Q^*P$ directly and then rearranging we obtain
    \begin{align}
        2 \Im&\<Pv,Qv\> + \frac{i \kappa}{2}\<[aD_t, m^w]v,v\> + \frac{i \kappa}{2} \<(D_t a) m^w v,v\>-\kappa \<(D_t a) q^w v,v\> \\
        &= \<i[\Box_g, q^w]v,v\> + \kappa \<(q^w aD_t + aD_t q^w)v,v\> 
        + \frac{1}{2}\<(\Box_g m^w + m^w \Box_g)v,v\>.\label{eq:poscommhigh}
    \end{align}
    The idea for the proof is to bound the left hand side from above by $\Ltx{v}^2+\LEs{Pv}\LEo{v}$ and to bound the right hand side from below by $\nm{v_{>\lambda}}_{LE^1_{\leq 2R_0}}^2$ minus errors. We make this idea precise in the following three lemmas.

    First we bound the terms on the left hand side from above.
    \begin{lemma}\label{l:high1}
    There exists $C(\lambda)>0,$ such that for all $T>0$
    \begin{align}
        &|\Im\<Pv,Qv\>| \leq C \LEs{Pv} \LEo{v} + C(\lambda) \Ltx{v}^2, \\
        &\left|\kappa \<(D_t a) q^w v,v \>\right| \leq C \lambda^{-\frac{1}{2}} \LEo{v}^2 + C(\lambda) \Ltx{v}^2, \\
        &\left|\frac{\kappa}{2}  \<[aD_t,m^w]v,v\>\right| +  \left|\frac{\kappa}{2}  \<(D_t a) m^w v,v \> \right|\leq C \Ltx{v}^2.
    \end{align}
    
    \end{lemma}
    To obtain our desired lower bound, we must split $\vgl$ further. Define 
    \begin{equation}
    \vglgl = \chi_{|\xi|>\lambda}^w \chi_{|\tau|>1}^w v, \qquad \vlgl=\chi_{|\xi|>\lambda}^w \chi_{|\tau|<1}^w v, \qquad \vgl = \vglgl + \vlgl.
    \end{equation}


    On one hand, we bound the right hand side of \eqref{eq:poscommhigh} from below by $\vglgl$ minus errors.
    \begin{lemma}
        \label{l:highprop}
        There exists $C>0, \rho_0>0, C(\lambda)>0$, such that for all $\rho \geq \rho_0$, and $T>0$
        \begin{align}
            \<i[\Box_g, q^w]v,v\> + \kappa \<(q^w aD_t + aD_t q^w)v,v\> + \frac{1}{2}\<(\Box_g m^w + m^w \Box_g)v,v\> \\
            \geq C \nm{\p \vglgl}_{LE_{\leq 2R_0}}^2 - C(\lambda) \Ltx{v}^2-C\lambda^{-1}\LEo{v}^2.
        \end{align}
    \end{lemma}
    On the other hand we have control over $\vlgl$ in $LE$.
    \begin{lemma}\label{l:highlow}
    There exists $C>0$ such that for all $T>0$,
    \begin{align}
        \LE{\p \vlgl}^2 \leq C\left(\frac{1}{\lambda^2}\LEs{Pv}^2 + \Ltx{v}^2+ \frac{1}{\lambda^2} \LEo{v}^2\right).
    \end{align}

    \end{lemma}
    Assuming the preceding three lemmas, we can conclude our high frequency estimate.
    \begin{proof}[Proof of Lemma \ref{l:highfreq}]
    Plugging Lemmas \ref{l:high1} and \ref{l:highprop} into \eqref{eq:poscommhigh} we have 
    \begin{equation}\label{eq:propConclusion}
        \nm{\p \vglgl}_{LE_{\leq 2R_0}}^2 \leq C \LEs{Pv}\LEo{v}+C(\lambda)\Ltx{v}^2  + C \lambda^{-\frac{1}{2}} \LEo{v}^2.
    \end{equation}
    Next by the triangle inequality 
    \begin{equation}
        \nm{\p \vgl}_{LE_{\leq 2R_0}}^2 \leq \nm{\p \vglgl}_{LE_{ \leq 2R_0}}^2 + \LE{\p \vlgl}^2.
    \end{equation}
    So applying Lemma \ref{l:highlow} and \eqref{eq:propConclusion}, we have 
    \begin{align}
    \nm{\p \vgl}_{LE_{\leq 2R_0}}^2 
    \leq &C \LEs{Pv}\LEo{v}+ C(\lambda) \Ltx{v}^2 + \frac{C}{\lambda^{2}} \LEs{Pv}^2+C\lambda^{-\frac{1}{2}} \LEo{v}^2.
    \end{align}
    From Lemma \ref{l:plancherel} part 2 we have $\LE{\<x\>^{-1} v_{>\lambda}}<\Ltx{v}$. Adding this to both sides completes the $LE^1$ norm on the left hand side. Then taking square roots gives
    \begin{align}
        \nm{v_{>\lambda}}_{LE^1_{\leq 2R_0}} \leq& C\LEs{Pv}^{\frac{1}{2}}\LEo{v}^{\frac{1}{2}} + C(\lambda) \Ltx{v}+ C \lambda^{-1} \LEs{Pv} + C\lambda^{-\frac{1}{4}}  \LEo{v}.
    \end{align}
    Finally, applying Young's inequality for products to the $LE^* LE^1$ term gives the desired high frequency estimate.
    \end{proof}
    It remains to prove Lemmas \ref{l:high1}, \ref{l:highprop}, and \ref{l:highlow}.

    \begin{proof}[Proof of Lemma \ref{l:high1}]
    1) Recall in Lemma \ref{l:plancherel} part 3 we defined $\vggl= \chiggl^w v$. Note $(1-\chiggl)^w \in \Psi^{-\infty}(\Rb^4)$, so by Proposition \ref{prop:pseudoCalc} there exists $R_{-\infty, \lambda} \in \Psi^{-\infty}(\Rb^4)$, such that 
    \begin{align}
         \Im \<Pv, Qv\> &=  \Im\<Pv, Q \vggl\> + \<R_{-\infty, \lambda} v,v\>.
    \end{align}
    We include the $\lambda$ in $R_{-\infty, \lambda}$ to emphasize that it depends on $\lambda$. In fact, despite being in $\Psi^{-\infty}$ the $L^{\infty}$ size of the symbol of $R_{-\infty, \lambda}$ grows like $\lambda^3$. Because of this,  $R_{-\infty, \lambda}$ is bounded on $L^2_t L^2_x$ by Proposition \ref{prop:pseudoBounded} but the size of its operator norm depends on $\lambda$. In particular we have 
    \begin{align}
        \left|\Im \<Pv, Q v\> \right| &\leq  \Ltx{Pv}\Ltx{Q \vggl} + C(\lambda) \Ltx{v}.
    \end{align}     
    Since $Q=q^w-\frac{i}{2}m^w \in \Psi^1(\Rb^4)$, by Proposition \ref{prop:pseudoBounded} it is a bounded map from $H^1_{t,x}$ to $L^2_{t}L^2_x$
    \begin{equation}
        \left|\Im \<Pv, Q v\> \right|\leq C \Ltx{Pv} \nm{\vggl}_{H^1_{t,x}}+C(\lambda) \Ltx{v}.
    \end{equation}
    Finally, applying Lemma \ref{l:plancherel} parts 1 and 3 to estimate terms on the right hand side we have the desired inequality
    \begin{equation}
        \left|\Im \<Pv, Q v\> \right| \leq C \LEs{Pv} \LEo{v}+C(\lambda) \Ltx{v}.
    \end{equation}
    2) For the second inequality we proceed similarly with $\vggl$ and a different $R_{-\infty, \lambda} \in \Psi^{-\infty}$
    \begin{align}
        \<(D_t a) q^w v,v \> = \<(D_t a) q^w \vggl, \vggl\> + \<(R_{-\infty, \lambda} v,v\>. 
    \end{align}
    We have $R_{-\infty, \lambda} \in \Psi^{-\infty}$, and following an argument analogous to that in step 1 above, it is bounded on $L^2_t L^2_x$ with operator norm dependent on $\lambda$, so 
    \begin{equation}
       \left| \<R_{-\infty, \lambda} v,v\> \right| \leq C(\lambda) \Ltx{v}^2.
    \end{equation}
    On the other hand, we have
    \begin{align}
        \left| \<(D_t a) q^w \vggl, \vggl\> \right| &= \left| \< \<\p\>^{-1/2} (D_t a) q^w \vggl, \<\p\>^{1/2} \vggl\> \right| \\
        &\leq \Ltx{ \<\p\>^{-1/2} (D_t a) q^w \vggl} \Ltx{\<\p\>^{1/2} \vggl}.
    \end{align}
    Now because $(D_t a) q^w \in \Psi^1$, we have $\<\p \>^{-1/2} (D_t a) q^w \in \Psi^{1/2}$. Thus it is bounded from $H^{1/2}$ to $L^2$ by Proposition \ref{prop:pseudoBounded}. Using this and Lemma \ref{l:plancherel} part 3 we have
    \begin{align}
        \left| \<(D_t a) q^w \vggl, \vggl\> \right| 
        &\leq C \nm{\vggl}_{H^{1/2}_{t,x}} \leq C \lambda^{-1/2} \LEo{v}.
    \end{align}
    3) Now to see the third inequality, we note that $[aD_t, m^w], (D_t a)m^w \in \Psi^0$. Then since $a(t,x)$ is  uniformly continuous in $t$, by Proposition \ref{prop:pseudoBounded} there exists $C>0$ such that, for all $T>0$
    \begin{align}
        \left| \frac{\kappa}{2} \<[aD_t, m^w]v,v\>\right| \leq C \Ltx{[aD_t, m^w]v}\Ltx{v} \leq C \Ltx{v}^2 \\
        \left| \frac{\kappa}{2} \<(D_t a) m^w v,v \> \right| \leq C \Ltx{(D_t a)m^w v}\Ltx{v} \leq C\Ltx{v}^2.
    \end{align}
    \end{proof}

    The main idea of the proof of Lemma \ref{l:highprop} is to use the lower bound on $H_p q + 2 \kappa \tau a  q +m p $ from Lemma \ref{p:EscapeFunction}, along with the sharp G\r{a}rding inequality Proposition \ref{prop:Garding} to obtain the desired lower bound. It is because $H_p q+ 2 \kappa \tau a q + mp$ is only bounded from below on $\{|\xi| \geq 1, |\tau| \geq 1\}$ that we only estimate $\vglgl$ in this Lemma, and must separately estimate $\vlgl$. To simplify  estimates for error terms additional frequency cutoffs are inserted and manipulated. These additional frequency cutoffs do not ultimately change where in $\xi$ and $\tau$ we obtain our lower bound.
    \begin{proof}[Proof of Lemma \ref{l:highprop}]
        1) Using Proposition \ref{prop:pseudoCalc}, there exists $R_0 \in \Psi^0(\Rb^4)$ such that
        \begin{align}\label{eq:highfreqfirst}
            \<i[\Box_g, q^w]v,v\> + \kappa \<(q^w aD_t + aD_t q^w)v,v\> + \frac{1}{2} \<(\Box_g m^w + m^w \Box_g)v,v\> \\
            =\<(H_p q+ 2\kappa \tau a q+mp)^w v,v\> + \<R_0 v,v\>.
        \end{align}
        From Lemma \ref{p:EscapeFunction}, we have a lower bound on the symbol of 
        \begin{equation}
            E=(H_pq+2\kappa \tau a q + mp)^w \in \Psi^2(\Rb^4).
        \end{equation} 
        Now we split up $v$, in order to simplify estimates for future error terms.
        We write 
        \begin{equation}
            v= \vggl + \chill v,
        \end{equation}
        and note again that $\chill \in \Psi^{-\infty}(\Rb^4)$. Therefore by Proposition \ref{prop:pseudoCalc}, for some $R_{-\infty,\lambda} \in \Psi^{-\infty}(\Rb^4)$ we have
        \begin{align}
            \<Ev,v\>& =\<E(\vggl+\chill^w v), (\vggl+\chill^w v)\>\\
            &=\<E \vggl, \vggl\> + \<R_{-\infty,\lambda}v,v\>. \label{eq:highfreqintermed0}
        \end{align}
        We have written $R_{-\infty, \lambda}$ to emphasize the $\lambda$ dependence in this operator. In particular, although it is in $\Psi^{-\infty}$, the $L^{\infty}$ norm of its symbol has size $\lambda^2$. The operator is still bounded on $L^2_t L^2_x$ by Proposition \ref{prop:pseudoBounded}, but the size of its operator norm depends on $\lambda$.
        
        Now by Lemma \ref{p:EscapeFunction}, there exists $C>0$, which does not depend on $T$, such that for all $\omega =(t,x,\tau,\xi) \in T^*\Rb^4$
        \begin{equation}
            H_p q(\omega)+2\kappa \tau a(\omega) q(\omega)+m(\omega) p(\omega) - C \mathbbm{1}_{|\tau| \geq 1} \mathbbm{1}_{|\xi| \geq \lambda}\<x\>^{-4} (|\xi|^2 + \tau^2) \geq 0,
        \end{equation}
        where we have replaced $|\xi| \geq 1$ by $|\xi| \geq \lambda$ in the indicator function. 
        Note also that we crudely bounded $\d<1$ from Lemma \ref{p:EscapeFunction} to get $\<x\>^{-2-2\d} \geq \<x\>^{-4}$. As we will shortly see, the exact power on $\<x\>$ is irrelevant as we eventually work on $\{|x| < 2R_0\}$. 

        Now by the Sharp G\r{a}rding inequality, Proposition \ref{prop:Garding}, there exists $C>0$, still independent of $T$, such that for all $\lambda \geq 1$
        \begin{align}\label{eq:highfreqintermed1}
            \<E \vggl, \vggl\> \geq &C \<( \chi_{|\tau| \geq 1}\chi_{|\xi|\geq \lambda} \<x\>^{-4}(|\xi|^2+\tau^2))^w \vggl, \vggl\> - C \nm{\vggl}_{H^{\frac{1}{2}}_{t,x}}^2.
        \end{align}
        We now must bound this right hand side from below by $\nm{\p \vglgl}_{LE<2R_0}^2$ minus errors. 

        2) To do so, note that by Proposition \ref{prop:pseudoCalc}, for some $R_{1} \in \Psi^1(\Rb^4)$
        \begin{equation}\label{eq:bigComposition}
            (\chitl \chi_{|\xi|>\lambda} \<x\>^{-4}(|\xi|^2+\tau^2))^w = (\chitl^{1/2} \chi_{|\xi|>\lambda}^{1/2})^w D_{\alpha} \<x\>^{-4} D_{\alpha} (\chitl^{1/2} \chi_{|\xi|>\lambda}^{1/2})^w + R_{1}.
        \end{equation}
        Note that although $R_1$ depends on $\lambda$, this dependence comes from derivatives of $\chi_{|\xi|>\lambda}$. These derivatives, and further derivatives, produce negative powers of $\lambda$. Since $\lambda \geq 1$, when applying Proposition \ref{prop:pseudoBounded} to estimate the operator norm of $R_1$ we can ignore any $\lambda$ dependence.
        
        Now since $\chi^{1/2}$ and $\chi$ only differ on a compact set, there exists $r_{-\infty, \lambda} \in S^{-\infty}(\Trf)$ such that
        \begin{align}
            \chitl^{1/2}\chi_{|\xi|>\lambda}^{1/2} \chiggl = \chitl \chi_{|\xi|>\lambda}\chiggl +r_{-\infty, \lambda}.
        \end{align}
        Next, note $\chi_{|\xi|+|\tau|>\frac{\lambda}{2}} \equiv1$ on $\supp \chi_{|\tau|>1} \chi_{|\xi|>\lambda}$, so 
        \begin{align}\label{eq:chiChange}
            \chitl^{1/2}\chi_{|\xi|>\lambda}^{1/2} \chiggl = \chitl \chi_{|\xi|>\lambda} + r_{-\infty, \lambda}.
        \end{align}
        Therefore applying \eqref{eq:bigComposition}, \eqref{eq:chiChange}, and Proposition \ref{prop:pseudoCalc} we have
        \begin{align}
            &\< (\chitl \chi_{|\xi|>\lambda} \<x\>^{-4} (|\xi|^2+\tau^2))^w \vggl, \vggl \> \\
            &\geq \Ltx{\<x\>^{-2}\p(\chitl^{1/2}\chi_{|\xi|>\lambda}^{1/2})^w \vggl}^2 - |\<R_{1} \vggl, \vggl\>| \\
            &\geq \Ltx{\<x\>^{-2}\p\vglgl}^2 - |\<R_{1} \vggl, \vggl\>| - |\<R_{-\infty,\lambda} v,v \>|, \label{eq:highfreqintermed2}
        \end{align}
        where $R_{-\infty,\lambda}=\text{Op}^w(r_{-\infty, \lambda}) \in \Psi^{-\infty}(\Rb^4)$.
        Next, by restricting to $|x| \leq 2R_0$ we have $\<x\>^{-2} \geq C$ and so applying Lemma \ref{l:plancherel} part 1 we have
        \begin{equation}
            \Ltx{\<x\>^{-2}\p \vglgl}^2 \geq C \nm{\p\vglgl}_{LE_{ \leq2R_0}}^2.
        \end{equation}
        It is at this point that we see the exact power on $\<x\>$ is irrelevant. Combining this with \eqref{eq:highfreqintermed0}, \eqref{eq:highfreqintermed1}, and \eqref{eq:highfreqintermed2} we have 
        \begin{align}\label{eq:highfreqwitherror}
            \<E v,v\> + \<R_0 v,v\> \geq &C\nm{\p\vglgl}_{LE_{ \leq 2R_0}}^2-|\<R_{-\infty, \lambda}v,v\>|- C\nm{\vggl}_{H^{1/2}_{t,x}}\\
            &- |\<R_1 \vggl, \vggl\>| - |\<R_0v,v\>|.
        \end{align}
        So it remains to estimate the error terms on the right hand side. 

        3) Since $R_1 \in \Psi^1(\Rb^4)$, by Proposition \ref{prop:pseudoBounded} it is bounded from $H^1_{t,x}$ to $L^2_t L^2_x$, and by our above discussion its operator norm does not depend on $\lambda$. Combining this with Cauchy-Schwarz, then applying Lemma \ref{l:plancherel} part 3 we have 
        \begin{align}
            |\<R_1 \vggl, \vggl\>| &\leq C \nm{\vggl}_{H^1_{t,x}}\Ltx{\vggl} \leq \frac{C}{\lambda} \LEo{v}^2. \label{eq:highfreqr1}
        \end{align}
        For the $H^{\frac{1}{2}}_{t,x}$ term apply Lemma \ref{l:plancherel} part 3
        \begin{equation}\label{eq:highfreqhhalf}
            \nm{\vggl}_{H^{\frac{1}{2}}_{t,x}}^2 \leq \frac{C}{\lambda}\LEo{v}^2.
        \end{equation}
        Since $R_{-\infty, \lambda} \in \Psi^{-\infty}(\Rb^4)$ and $R_0 \in \Psi^0(\Rb^4)$, by Proposition \ref{prop:pseudoBounded} there exists a constant $C(\lambda)>0$ such that 
        \begin{align}
        \begin{split}
            |\<R_0 v,v\>| + |\<R_{-\infty, \lambda} v,v\>| &\leq \Ltx{R_0 v} \Ltx{v} + \Ltx{R_{-\infty, \lambda}v}\Ltx{v} 
            \\
            &\leq C(\lambda) \Ltx{v}^2.\label{eq:highfreql2bound}
        \end{split}
        \end{align}
        Combining \eqref{eq:highfreqr1}, \eqref{eq:highfreqhhalf}, \eqref{eq:highfreql2bound} with \eqref{eq:highfreqwitherror} gives 
        \begin{align}
            \<(H_p q + 2\kappa \tau a q+mP)^w v,v\> + \<R_0 v,v\>  \geq& C\nm{\p\vglgl}_{LE_{ \leq 2R_0}}^2- C(\lambda) \Ltx{v}^2 - C\lambda^{-1} \LEo{v}^2.
        \end{align}
        This along with \eqref{eq:highfreqfirst} gives 
         \begin{align}
            \<i[\Box_g, q^w]v,v\> + \kappa \<(q^w aD_t + aD_t q^w)v,v\> + \frac{1}{2} \<(\Box_g m^w + m^w \Box_g)v,v\> \\
            \geq C\nm{\p\vglgl}_{LE_{ \leq 2R_0}}^2- C(\lambda) \Ltx{v}^2 - C\lambda^{-1} \LEo{v}^2,
        \end{align}
        which is exactly the desired conclusion. 
        
    \end{proof}
    To estimate $\p \vlgl$ we follow the same approach used to estimate $\p \vgll$ in the proof of Lemma \ref{l:lowfreq}
    \begin{proof}[Proof of Lemma \ref{l:highlow}]
    To begin, by Lemma \ref{l:plancherel} part 1, and Plancherel's theorem, there exists $C>0$ such that  
    \begin{align}
        \LE{\p \vlgl} &\leq \Ltx{\p \vlgl} \\
        &\leq \Ltauxi{(|\tau|+|\xi|) \chi_{|\xi|>\lambda}\chi_{|\tau|<1} \hat{v}} \\
        &\leq \Ltauxi{\hat{v}}+ C \Ltauxi{\frac{|\xi|^2}{\lambda} \chi_{|\xi|>\lambda} \chi_{|\tau|<1} \hat{v}} \\
        &\leq  \Ltx{v}+ \frac{C}{\lambda}\Ltx{D_j^2 (\chi_{|\tau|<1}^wv)},
    \end{align}
    where in the final line we used that  $\chi_{|\xi|>\lambda}$ is a bounded Fourier multiplier, and that $\chi_{|\tau|<1}^w$ and $D_j^2$ commute as Fourier multipliers. Now using that $g^{ij}$ is uniformly elliptic, and then commuting $g^{ij}$ with $\chi^w_{|\tau|<1}$, noting that the metric does not depend on $t$, we have
    \begin{align}
    \begin{split}
        \LE{\p \vlgl} &\leq   \Ltx{v}+ \frac{C}{\lambda}\Ltx{g^{ij}D_i D_j \chi_{|\tau|<1}^w v}
        \\
        &= \Ltx{v}+\frac{C}{\lambda}\Ltx{\chi_{|\tau|<1}^w(g^{ij}D_iD_j v)}.\label{eq:highlow1}
    \end{split}
    \end{align}
    To estimate the final term on the right hand side, we follow our approach used to estimate $\chi_{|\xi|<\lambda}^w (\p_t^2v)$ in the Proof of Lemma \ref{l:lowfreq}. That is we use that $P=D_{\alpha} g^{\alpha \beta} D_{\beta} + i aD_t$, and apply the triangle inequality to write
    \begin{align}
        \Ltx{\chi_{|\tau|<1}^w (g_{ij} D_i D_j v)} &\leq \Ltx{\chi_{|\tau|<1}^w (Pv)} + \Ltx{\chi_{|\tau|<1}^w(g^{00} D_t^2 v)} \\
        &\quad + \Ltx{\chi_{|\tau|<1}^w(D_t(g^{0j} D_j+D_j g^{0j})v)} \\
        & \quad + \Ltx{\chi_{|\tau|<1}^w( (D_i g^{ij})D_j v)} + \Ltx{\chi_{|\tau|<1}^w(aD_t v)}.\label{eq:highlow2}
    \end{align}
    We will now estimate the individual terms on the right hand side in turn. By the same argument used to show \eqref{eq:lowP} and \eqref{eq:lowdamp} in the proof of Lemma \ref{l:lowfreq}, we have 
    \begin{align}
        &\Ltx{\chi_{|\tau|<1}^w (Pv)}  \leq \Ltx{Pv} \leq \LEs{Pv} \label{eq:highlowP}\,,\\
        &\Ltx{\chi_{|\tau|<1}^w(aD_t v)} \leq \Ltx{aD_t v} \leq C \LE{\p v}\, .\label{eq:highlowa}
    \end{align}
    Now to estimate the $D_t^2 v$ term we use that $g^{00}$ is independent of $t$, so it commutes with $\chi^w_{|\tau|<\lambda}$, then apply Plancherel 
    \begin{align}
        \Ltx{\chi_{|\tau|<1}^w(g^{00} D_t^2 v)} \leq C\Ltauxi{\chi_{|\tau|<1} \tau^2 \hat{v}} \leq C \Ltauxi{\hat{v}} = C \Ltx{v}.\label{eq:highlowdt}
    \end{align}
    Similarly to estimate the $D_t$ terms we use Plancherel, as well as Lemma \ref{l:plancherel} part 1, and that $v$ is supported in $\{|x| \leq 2R_0\}$
    \begin{align}
        \Ltx{\chi_{|\tau|<1}^w(D_t(g^{0j} D_j+D_j g^{0j})v)} &\leq  \Ltx{(g^{0j} D_j+D_j g^{0j})v} \leq C \Ltx{D_j v}+C\Ltx{v} \\
        &\leq C  \Ltx{\p v} + C  \Ltx{v} \leq C  \LE{\p v} + C  \Ltx{v}.\label{eq:highlowdtg}
    \end{align}
    To estimate the $(D_i g^{ij}) D_j$ term, we use that $\chi_{|\tau|>1}$ is a bounded Fourier multiplier, and $D_i g^{ij}$ is a bounded function, then again use Lemma \ref{l:plancherel} part 1 and that $v$ is supported in $\{|x| \leq 2R_0\}$
    \begin{align}
        \Ltx{\chi_{|\tau|<1}^w ((D_i g^{ij}) D_j v)} \leq \Ltx{D_j v} \leq C \Ltx{\p v} \leq C \LE{\p v}.\label{eq:highlowg}
    \end{align}
    Now if we combine  \eqref{eq:highlow2}, \eqref{eq:highlowP}, \eqref{eq:highlowa}, \eqref{eq:highlowdt}, \eqref{eq:highlowdtg}, and \eqref{eq:highlowg} we have
    \begin{align}
        \Ltx{\chi^w_{|\tau|<1}(g^{ij} D_i D_j v)} \leq C \left(\LEs{Pv} + \Ltx{v}+ \LE{\p v} \right).
    \end{align}
    Plugging this back into \eqref{eq:highlow1}, we obtain
    \begin{align}
        \LE{\p \vlgl} &\leq C \Ltx{v} + \frac{C}{\lambda} \left( \LEs{Pv} + \Ltx{v}+ \LE{\p v}   \right).
    \end{align}
    Squaring both sides gives us the desired inequality.
    
    \end{proof}

\section{Local Energy Decay}\label{s:localEnergyDecay}
To begin we note that, via Proposition \ref{prop:iledCaseReduction}, we have
\begin{proposition}\label{p:timeFreqCase}
    Suppose $P$ is a stationary, asymptotically flat damped wave operator, with time-dependent damping satisfying the time-dependent geometric control condition, and suppose $\p_t$ is uniformly time-like with constant time slices uniformly space-like.
    If there exists $C>0$ such that for all $u$ with $u[0]=u[T]=0$, $Pu \in LE^*$ and $Pu$ compactly supported, we have
    \begin{equation}\label{eq:iledReduced}
        \LEoT{u} \leq C \LEST{Pu},
    \end{equation}
    then integrated local energy decay holds. That is there exists $C>0$, such that for all $T>0$ and $w$ with $w[0] \in \dot{H}^1 \times L^2$
    \begin{equation}
        \LEoT{w} + \nm{\p w}_{L^{\infty}_t L^2_x[0,T]} \leq C \left( \ltwo{\p w(0)} + \nm{Pw}_{LE^* + L^1_t L^2_x[0,T]} \right).
    \end{equation}
\end{proposition}
Therefore to establish local energy decay, it is enough to prove \eqref{eq:iledReduced}. To do so, for a given $u$ we split $u$ into three time frequency regimes. That is for some $\tau_0, \tau_1$, let $\chimed=1-\chi_{\leq \tau_0}-\chi_{\geq \tau_1}$ and define 
\begin{align}
    \chilow^w := \Op^w(\chi_{\leq \tau_0}(\tau)), \qquad \chimed^w:=\Op^w(\chimed(\tau)), \qquad \chihigh^w:= \Op^w(\chi_{\geq \tau_1}(\tau)).
\end{align}
Then we will consider
\begin{equation}
    u = \chilow^w u + \chimed^w u + \chihigh^w u.
\end{equation}
We call these regimes: low frequency, medium frequency, and high frequency. We have already shown the high frequency estimate, Theorem \ref{thm:highfreq}. We will cite low and medium frequency estimates for $\Box_g$, and then explain how these estimates can be combined to obtain Theorem \ref{thm:iled}. A key step is to use $\Box_g=P-a\p_t$ and control $a\p_t$ using standard energy identities. Put another way, in the low and medium frequency regimes the damping can be treated as a perturbative term.

Our general approach is similar to that of \cite[Section 3]{Kofroth23} and \cite[Section 7.2]{MST20}. However, our treatment of the damping as a perturbation is new and simplifies the argument, in part by avoiding a commutator estimate of $[a, \chilow^w]$ which would require an almost-stationary hypothesis on $a$. 

\subsection{Low Frequency}
In this section, we cite a low frequency estimate from \cite{MST20}. To begin, we state a definition.

\begin{definition}
    We say that a wave operator $D_{\alpha} g^{\alpha \beta} D_{\beta}$ satisfies a zero non-resonance condition, if there exists $K_0 >0$ such that for all $u \in \dot{H}^1$
    \begin{equation}\label{eq:zernoresonance}
        \nm{u}_{\dot{H}^1} \leq K_0 \nm{D_i g^{ij} D_j u}_{\dot{H}^{-1}}.
    \end{equation}
\end{definition}
Interpreting $\Box_g$ as a magnetic wave operator with 0 scalar and magnetic potentials,  \eqref{eq:zernoresonance} holds by \cite[Lemma 6.2(iii)]{MST20}. Therefore, we may make use of the following version of \cite[Theorem 6.1]{MST20}.
\begin{proposition}\label{p:MST20low}
    Let $\Box_g$ be an asymptotically flat wave operator with $\p_t$ uniformly time-like, and suppose the zero non-resonance condition \eqref{eq:zernoresonance} condition holds. Then there exists $C>0$, such that for all $T>0$ and all $u$ compactly supported 
    \begin{equation}
        \LEo{u} \leq C \left( \LEo{\p_t u} + \LEs{\Box_g u} \right).
    \end{equation}
\end{proposition}
\begin{remark}
    This estimate is low frequency in the sense that, when we apply it to $\chilow^w u$ using Plancherel we can control 
    \begin{equation}
        \LEo{\p_t \chilow^w u} \leq C \tau_0 \LEo{\chilow^w u}.
    \end{equation}
    Taking $\tau_0>0$ small enough we can absorb this back into the left hand side leaving an estimate of the same form as \eqref{eq:iledReduced}.
\end{remark}
\subsection{Medium Frequency}
In this subsection we cite a medium frequency result, namely \cite[Theorem 5.4]{MST20}. We write $LE^1_0$ to be the closure of $C_0^{\infty}$ in the $LE^1$ norm.
\begin{proposition}\label{p:MST20med}
    Let $\Box_g$ be an asymptotically flat wave operator such that $\p_t$ is uniformly time-like. Then there exists $C>0$, such that for all $\d>0$, there exists a bounded, increasing, radial weight $\phi = \phi(\log(1+r))$ so that for all $u \in LE^1_0$ with $\Box_g u \in LE^*$
    \begin{align}
        &\LE{(1+\phi''_+)^{1/2}e^{\phi}\nabla u} + \LE{\<x\>^{-1}(1+\phi''_+)^{1/2}(1+\phi') e^{\phi} u} +\LE{(1+\phi')^{1/2} e^{\phi} \p_t u} \\
        &\qquad \leq C\bigg( \d \left( \LE{(1+\phi')^{1/2}e^{\phi}u} + \LE{\<x\>^{-1} (1+\phi''_+)^{1/2} (1+\phi') e^{\phi} \p_t u} \right) + \LEs{e^{\phi}\Box_g u} \bigg).  \label{eq:carleman}
    \end{align}
    Here $\phi''_+=\max\{0, \phi''\}$.
\end{proposition}
Note that \cite[Theorem 5.4]{MST20} applies to more general wave operators $P$, but we only require this version. 
\begin{remark}
    This estimate is medium frequency, in the sense that when we apply it to $\chimed^w u=v$, after fixing $\tau_0$ and $\tau_1$, we can apply Plancherel to see 
    \begin{align}
        \d \LE{(1+\phi')^{1/2} e^{\phi} v} &\leq \frac{C \d}{\tau_0} \LE{(1+\phi')^{1/2} e^{\phi} \p_t v}, \\
        \d \LE{\<x\>^{-1} (1+\phi''_+)^{1/2}(1+\phi') e^{\phi}\p_t v} &\leq C \d \tau_1 \LE{\<x\>^{-1}(1+\phi''_+)^{1/2} (1+\phi') e^{\phi} v}.
    \end{align}
    Then choosing $\d>0$ small enough we can absorb both of these terms back into the left hand side leaving an estimate of the same form as \eqref{eq:iledReduced}. Note that $\d$ can be taken arbitrarily small, which allows any $\tau_0>0, \tau_1<\infty$, so long as they are fixed. 
\end{remark}

\subsection{Preliminary Estimates}

Before proceeding with the combination of the estimates, we record two useful estimates.

First, we state a standard energy identity of $a \p_t u$, that allows us to treat the term as a perturbation. 
\begin{lemma}\label{l:aptuControl}
    There exists $C>0$, such that for all $T>0$, and $u$ with $u[0] \in \dot{H}^1 \times L^2$
    \begin{equation}
        \LEST{a \p_t u} \leq C\LtxT{a^{1/2} \p_t u} \leq C\left( \ltwo{\p u(0)} + \LEST{Pu}^{1/2} \LEoT{u}^{1/2}\right).
    \end{equation}
\end{lemma}
\begin{proof}
    First, by the asymptotic flatness of $a$
    \begin{align}
        \LEs{a \p_t u} &= \sum_{j=0}^{\infty} \nm{\<x\>^{1/2} a \p_t u}_{L^2_t L^2_x ([0,T] \times A_j)} \leq C\nm{ \<x\>^{1/2} a^{1/2}}_{L^{\infty}_{t,x}([0,T] \times \Rb^3) }\LtxT{a^{1/2} \p_t u} \\
        & \leq C \LtxT{a^{1/2} \p_t u}.
    \end{align}
    Now to control 
    \begin{equation}
        \LtxT{a^{1/2}\p_t u} = \left( \int_0^{T} \int_{\Rb^3} a |\p_t u|^2 dx dt \right)^{1/2},
    \end{equation}
    recall from \eqref{eq:energyDeriv} we have 
    \begin{equation}
        \frac{d}{dt} E(u,t) = 2 \Re \int_{\Rb^3} Pu \p_t u dx - 2 \int_{\Rb^3} a |\p_t u|^2 dx. 
    \end{equation}
    Integrating both sides from $t=0$ to $t=T$, we have 
    \begin{equation}
        E(u,T) - E(u,0) = 2 \Re \int_0^T \int_{\Rb^3} Pu \p_t u dx dt - 2 \int_0^T \int_{\Rb^3} a|\p_t u|^2 dx dt.
    \end{equation}
    Rearranging, and recalling that $E(u,t) \simeq \ltwo{\p u(t)}^2$ we have 
    \begin{equation}
        \int_0^T \int_{\Rb^3} a|\p_t u|^2 dx dt \leq C \left( \ltwo{\p u(0)}^2 - \ltwo{\p u(T)}^2  + \int_0^T \int_{\Rb^3} |Pu \p_t u| dx dt\right). 
    \end{equation}
    We can drop $\ltwo{\p u(T)}^2$ from the right hand side because it is negative, to obtain
    \begin{equation}
        \int_0^{T} \int_{\Rb^3} a|\p_t u|^2 dx dt \leq C \left( \ltwo{\p u(0)}^2  + \int_0^{T} \int_{\Rb^3} |Pu \p_t u| dx dt \right).
    \end{equation}
    Now arguing as in the proof of Lemma \ref{l:Linftyeasy} we control 
    \begin{equation}
        \int_0^{T} \int_{\Rb^3} |Pu \p_t u| dx dt \leq \LEs{Pu} \LEo{u},
    \end{equation}
    and taking square roots of both sides, we obtain the desired inequality. 
\end{proof}

Next we state a commutator estimate between $P$ and the high frequency cutoff. Note that we do not compute commutators of $a$ with the other cutoffs, because we do not insert $a$ until after we have freely commuted the stationary $\Box_g$ with the time-frequency cutoffs.
\begin{lemma}\label{l:freqcombinecommutator}
There exists $C>0$, such that for all $u$ with $u[0] \in \dot{H}^1 \times L^2$ 
and all $\tau_1 >1$ we have
    \begin{align}
        \LEs{[P,\chihigh^w]u}=\LEs{[a,\chihigh^w] \p_t u} \leq C \tau_1^{-1}\LEo{u}.
    \end{align}
\end{lemma}
\begin{proof}
    First, we note that since $\Box_g$ is stationary, we have
    \begin{equation}
        [P,\chihigh^w] = [a\p_t, \chihigh^w]=[a, \chihigh^w]\p_t.
    \end{equation}
    Let $\chi_{A_j}(x)$ be a smooth cutoff, identically 1 on $A_j$ and supported on $A_{j-1} \cup A_j \cup A_{j+1}$. Then since $\chihigh$ has no spatial component
    \begin{equation}
        [\chihigh,\<x\>^k]=0, \, \forall k \in \Rb \quad \text{ and } \quad [\chihigh, \chi_{A_j}]=0.
    \end{equation}
    Now writing $1=\<x\>^{-1/2} \<x\> \<x\>^{-1/2}$, then using the definition of $LE^*$, computing directly and using the definition of $\chi_{A_j}$ we have 
    \begin{align}
        \LEs{[a, \chihigh^w ]\p_t u} &\leq \sum_{j=0}^{\infty} \nm{[\<x\> a \chi_{A_j}, \chihigh^w] \chi_{A_j} \<x\>^{-1/2} \p_t u}_{L^2_t L^2_x (\Rb \times A_j)}\\
        &\leq C\sup_{j \geq 0} \Ltxaj{\<x\>^{-1/2} \p_t u} \sum_{j=0}^{\infty} \nm{[\<x\>a \chi_{A_j}, \chihigh^w]}_{L^2 \ra L^2} \\
        &\leq \LEo{u} \sum_{j=0}^{\infty} \nm{[\<x\>a \chi_{A_j}, \chihigh^w]}_{L^2 \ra L^2}.\label{eq:freqcombineCommutator}
    \end{align}
    We now estimate the $L^2 \ra L^2$ norm of the commutator. First, $\chihigh(\tau) = \chi_{>1}(\frac{\tau}{\tau_1})$, so taking a change of variables $\eta=\tau/\tau_1= h \tau$
    \begin{align}
        \chi^w_{\tau_1 \leq |\tau|} u(t)&= \frac{1}{2\pi} \int_{\Rb} \int_{\Rb} \chi_{>1}\left(\frac{\tau}{\tau_1}\right) e^{i(t-s) \tau} u(s) ds d \tau \\
        &= \frac{1}{2\pi h} \int_{\Rb} \int_{\Rb}\chi_{>1}(\eta) e^{i(t-s) \frac{\eta}{h}} u(s) ds d \eta =\Ophw(\chi_{>1}(\tau)),
    \end{align}
    we recognize this as a semiclassical quantization, see Definition \ref{def:semiPseudo}.
    
    Therefore, by Proposition \ref{prop:semipseudoCalc}, there exists $R_{j,-3} \in \Psi_h^{-3}(T^* \Rb_t)$ such that 
    \begin{equation}\label{eq:freqcombinecommSemiclassical}
        [\<x\> a \chi_{A_j}, \chi^w_{\tau_1 \leq |\tau|}] = [\<x\> a \chi_{A_j}, \Ophw(\chi_{>1}(\tau))]= i h\Ophw(\<x\> \chi_{A_j}  \p_t a \chi'_{>1}(\tau)) + h^3 R_{j,-3}.
    \end{equation}
    By Propositions \ref{prop:semipseudoBounded} and \ref{prop:semipseudoCalc}, there exists an $N>0$ such that 
    \begin{align}
        \nm{\Ophw(\<x\> \chi_{A_j}  \p_t a \chi'_{>1}(\tau))}_{L^2 \ra L^2}  + \nm{R_{j,-3}}_{L^2 \ra L^2}&\leq C \max_{k+l \leq N} \nm{\p_x^l\left( \<x\> \chi_{A_j} \p_t^{k} a \right) \p_{\tau}^{k} \chi_{>1}(\tau)}_{L^{\infty}_{t,x}(\Rb \times \Rb^3)}\\
        &\leq C \max_{|\alpha| \leq N} \nm{\<x\>^{|\alpha|+1} \p^{\alpha} a}_{L^{\infty}_{t,x}(\Rb \times A_j)}.
    \end{align}
    Combining this with \eqref{eq:freqcombinecommSemiclassical} and recalling that $h= \frac{1}{\tau_1}$ we have 
    \begin{equation}\label{eq:freqCombineClaim}
        \sum_{j=0}^{\infty} \nm{[\<x\>a \chi_{A_j}, \chihigh^w]}_{L^2 \ra L^2} \leq \frac{C}{\tau_1} \max_{|\alpha| \leq N} \|\<x\>^{|\alpha|+1} \p^{\alpha} a\|_{l^1_j L^{\infty}(\Rb\times A_j)} \leq \frac{C}{\tau_1},
    \end{equation}
    where the second inequality follows from the asymptotic flatness of $a$ in Definition \ref{d:asymptoticFlat1}. Plugging this back into \eqref{eq:freqcombineCommutator} gives the desired conclusion.
\end{proof}

\subsection{Combination of Estimates}\label{sec:combineILED}
In this subsection, we prove Theorem \ref{thm:iled} by proving the hypothesis of Proposition \ref{p:timeFreqCase}.
\begin{proposition}\label{p:timeFreqCombine}
    Suppose $P$ is a stationary, asymptotically flat damped wave operator, with time-dependent damping satisfying the time-dependent geometric control condition, and suppose $\p_t$ is uniformly time-like with constant time-slices uniformly space-like. 
    Then there exists $C>0$ such that for all $u$ with $u[0]=u[T]=0$ and $Pu \in LE_c^*$
    \begin{equation}
        \LEoT{u} \leq C \LEST{Pu},
    \end{equation}
\end{proposition}
\begin{proof}
    First, note that since $u[0]=u[T]=0$ and $Pu$ is compactly supported, then $u$ is compactly supported by finite speed of propagation. However we do not have uniform in $T$ control over the size of the compact support of $u$.
    We extend $u$ by $0$ outside of $[0,T]$. Because of this, we have 
    \begin{align}
        \LEoT{u} = \LEo{u}, \quad \Ltx{u}= \nm{u}_{L^2_tL^2_x[0,T]}, \quad \LEST{Pu} = \LEs{Pu}.
    \end{align}
    Note also, since $\Box_g = P-a \p_t$, by Lemma \ref{l:aptuControl} we have 
    \begin{align}
        \LEs{\Box_g u} &\leq \LEs{Pu} + \LEs{a \p_t u} \leq \LEs{Pu} + C \LEs{Pu}^{1/2} \LEo{u}^{1/2},\label{eq:boxgPerturb}
    \end{align}
    where the $\p u(0)$ term was dropped because $u[0]=0$. Therefore $\Box_g u \in LE^*$, where we note that $u \in LE^1$ because $u$ is compactly supported in $x$.
    
    Now, we write $u = \chilow^w u + \chimed^w u + \chihigh^w u$, with $\tau_0, \tau_1>0$ to be chosen. We will estimate each of these terms. 

    We first estimate the $\chilow^w u$ term. By Proposition \ref{p:MST20low}
    \begin{align}
        \LEo{\chilow^w u} &\leq C \left( \LEo{\chilow^w \p_t u} + \LEs{\Box_g \chilow^w u} \right).
    \end{align}
    Then applying Plancherel and using that $\Box_g$ is stationary, so $[\chilow, \Box_g]=0$, we have
    \begin{align}
         \LEo{\chilow^w u} &\leq C\left(\tau_0 \LEo{\chilow^w u} + \LEs{\Box_g u} \right).
    \end{align}
    Now choosing $\tau_0$ small enough, we can absorb the first term on the right hand side back into the left hand side. Note that at this point we have fixed $\tau_0$. Then applying \eqref{eq:boxgPerturb} we have
    \begin{align}
        \LEo{\chilow^w u} &\leq C \left( \LEs{Pu} + \LEs{Pu}^{1/2} \LEo{u}^{1/2}\right).\label{eq:timeFreqlow}
    \end{align}

    We now estimate the $\chihigh^w u$ term. By Theorem \ref{thm:highfreq}, noting that $u[0]=0$ so the $\p u(0)$ term drops out, we have
    \begin{align}
        \LEo{\chihigh^w u} \leq C \left( \LE{\<x\>^{-2} \chihigh^w u} + \nm{P \chihigh^w u}_{LE^* +L^1_t L^2_x} \right).
    \end{align}
    Note that by Plancherel,
    \begin{align}
        \LE{\<x\>^{-2} \chihigh^w u} &\leq C \nm{\<x\>^{-2} \chihigh \hat{u}(\tau,x)}_{LE_{\tau,x}} \\
        &\leq \frac{C}{\tau_1} \nm{\<x\>^{-2} \tau \chihigh \hat{u}(\tau,x)}_{LE_{\tau,x}} \leq \frac{C}{\tau_1} \LEo{\chihigh^w u}.
    \end{align}
    So, choosing $\tau_1$ large enough, we can absorb this term back into the left hand side, and, commuting $P$ and $\chihigh$, we have 
    \begin{align}
        \LEo{\chihigh^w u} &\leq C \LEs{P \chihigh^w u}\leq C \left( \LEs{\chihigh^w Pu} + \LEs{[P,\chihigh^w] u} \right).
    \end{align}
    Now we apply Lemma \ref{l:freqcombinecommutator} to estimate the commutator and obtain 
    \begin{align}
        \LEo{\chihigh^w u} \leq C \left( \LEs{Pu}+ \tau_1^{-1}\LEo{u} \right).\label{eq:timeFreqhigh}
    \end{align}
    Now we choose $\tau_1$ large enough so that $C \tau_1^{-1} <\frac{1}{2}$, so that the $\tau_1^{-1} \LEo{u}$ term can eventually be absorbed back into $\LEo{u}$ on the left hand side. Note that at this point we have fixed $\tau_1$. 
 
    We now estimate the $\chimed^w u=:v$ term. By Proposition \ref{p:MST20med}
    \begin{align}
        &\LE{(1+\phi''_+)^{1/2}e^{\phi}\nabla v} + \LE{\<x\>^{-1}(1+\phi''_+)^{1/2}(1+\phi') e^{\phi} v} +\LE{(1+\phi')^{1/2} e^{\phi} \p_t v} \\
        &\qquad \leq C\bigg( \d \left( \LE{(1+\phi')^{1/2}e^{\phi} v} + \LE{\<x\>^{-1} (1+\phi''_+)^{1/2} (1+\phi') e^{\phi} \p_t v} \right) + \LEs{e^{\phi}\Box_g v} \bigg).\label{eq:timeFreqMed1}
    \end{align}
    Now by Plancherel's theorem
    \begin{equation}
        \d \LE{(1+\phi')^{1/2}e^{\phi} \chimed^w u} \leq \frac{C\d}{\tau_0} \LE{(1+\phi')^{1/2} e^{\phi} \chimed^w \p_t u}.
    \end{equation}
    Similarly by Plancherel's theorem
    \begin{align}
        &\d \LE{\<x\>^{-1} (1+\phi''_+)^{1/2} (1+\phi') e^{\phi} \p_t \chimed^w u} \leq C \d \tau_1 \LE{\<x\>^{-1} (1+ \phi''_+)^{1/2} (1+\phi') e^{\phi} \chimed^w u}.
    \end{align}
    Therefore choosing $\d>0$ small enough, since $\tau_0, \tau_1$ have been fixed, we can absorb these terms back into the left hand side of \eqref{eq:timeFreqMed1}. Then we are left with 
    \begin{align}
        \LE{(1+\phi''_+)^{1/2}e^{\phi}\nabla (\chimed^w u)} &+ \LE{\<x\>^{-1}(1+\phi''_+)^{1/2}(1+\phi') e^{\phi} \chimed^w u}\\ & \quad+\LE{(1+\phi')^{1/2} e^{\phi} \p_t (\chimed^w u)} 
        \leq C \LEs{e^{\phi}\Box_g \chimed^w u}.
    \end{align}
    Since $\phi$ is increasing and bounded this implies
    \begin{equation}
        \LEo{\chimed^w u} \leq C \LEs{\Box_g \chimed^w u}.
    \end{equation}
    Because $\Box_g$ is stationary it commutes with $\chimed^w$ and we have 
    \begin{align}
        \LEo{\chimed^w u} &\leq C \LEs{ \chimed^w \Box_g u} \leq C \LEs{\Box_g u}.
    \end{align}
    Now we apply \eqref{eq:boxgPerturb} to estimate the right hand side 
    \begin{align}
        \LEo{\chimed^w u} \leq C \left( \LEs{Pu} + \LEs{Pu}^{1/2} \LEo{u}^{1/2}\right). \label{eq:timeFreqMed}
    \end{align}

    We now combine the three time-frequency regimes, applying \eqref{eq:timeFreqlow}, \eqref{eq:timeFreqhigh}, and \eqref{eq:timeFreqMed} to obtain 
    \begin{align}
        \LEo{u} &\leq \LEo{\chilow^w u} + \LEo{\chimed^w u} + \LEo{\chihigh^w u}\\
         &\leq C \left( \LEs{Pu} + \LEs{Pu}^{1/2} \LEo{u}^{1/2} \right) + \frac{1}{2} \LEo{u}.
    \end{align}
    Recall the $\frac{1}{2}\LEo{u}$ term came from the high frequency estimate by choosing $\tau_1$ large enough. We can absorb this term back to be left with 
    \begin{equation}
        \LEo{u} \leq C \left( \LEs{Pu} + \LEs{Pu}^{1/2} \LEo{u}^{1/2} \right).
    \end{equation}
    Applying Young's inequality for products to the second term, for any $\e>0$, we have 
    \begin{equation}
        \LEo{u} \leq C \left( \LEs{Pu} + \e^{-1} \LEs{Pu}+ \e \LEo{u} \right).
    \end{equation}
    Choosing $\e>0$ small enough, we can absorb the final term back into the left hand side and arrive at the desired inequality.
    \end{proof}

\appendix
\section{Appendix: General Estimates}
    In this appendix we record the proofs of some estimates which we use elsewhere. The proofs are straightforward, but sometimes quite detailed. 
        \begin{lemma}\label{l:magnitudeDeriv}
            There exists $c>0$ such that for all $w \in \characteristicsetofPplusminus$, if $|x_s^{\pm}(\omega)|> R_0 $ then   
            \begin{equation}
                \frac{\p^2}{\p s^2} |x_s^+(\omega)|^2 \geq c.
            \end{equation}
        \end{lemma}
        \begin{proof}
        Without loss of generality we work with $x^+$. 
        For ease of notation and without loss of  generality we may assume $g^{00}=-1$ by replacing $g^{\alpha \beta}$ with $g^{\alpha \beta}/g^{00}$. We note that $g^{\alpha \beta}/g^{00}$ satisfies the same asymptotic flatness assumptions as $g^{\alpha \beta}$ because $g^{00} \geq -C$ and is asymptotically flat. 
        By abuse of notation, in this proof we will write 
            \begin{align}
                g^{\alpha \beta}=g^{\alpha \beta}(x_s^+(\omega)), \qquad 
                \p_j g^{\alpha \beta} = \frac{\p}{\p {x_j}} g^{\alpha \beta}(x_s^+(\omega)), \\
                x_j=(x_s^+)_j, \qquad \xi_j=(\xi_s^+)_j,  \qquad b^{\pm} = b^{\pm}(x_s(\omega), \xi_s(\omega)).
            \end{align}
            A key fact that we make repeated use of is 
            \begin{equation}
                \tau_s=\tau_0 = b^+(x_0(\omega), \xi_0(\omega))=b^+(x_s(\omega),\xi_s(\omega)).
            \end{equation}
            For any $\omega=\left(t_0,x_0,\tau_0,\xi_0\right)\in \characteristicsetofPplus$, we have
            \begin{equation}
                \frac{1}{2}\frac{\partial^2}{\partial s^2}\left|x^+_s(\omega)\right|^2=\left|\frac{\partial}{\partial s}x^+_s(\omega)\right|^2+x^+_s(\omega)\cdot\frac{\partial^2}{\partial s^2}x^+_s(\omega).
            \end{equation}
            If $g=m$, then a direct computation gives  
            \begin{equation}
                \bigg| \frac{\p}{\p s} x_s^+(\omega)\bigg|^2 + x_s(\omega)^+ \cdot \frac{\p^2}{\p s^2} x_s^+(\omega) =1.\label{eq:magnitudeDerivMinkowski}
            \end{equation}
            By the definition of the half wave flow \eqref{eq:flowdef} and the definition of $b$ \eqref{eq:bdef}, then computing directly we have 
            \begin{align}
                \left(\frac{\p}{\p s} x_s^+\right)_k = -(g^{0k} + 2(b^+-b^-)^{-1}(g^{0k} \xi_k + g^{kj} \xi_j)) =-2 (b^+- b^-)^{-1}(g^{0k} \tau + g^{kj} \xi_j) .\label{eq:magnitudeDerivX}
            \end{align}
            Therefore 
            \begin{align}
                \left|\frac{\p}{\p s} x_s^+ \right|^2 = 4(b^+-b^-)^{-2} \left( \tau^2 \sum_{k=1}^3 (g^{0k})^2 + 2 \tau \sum_{k=1}^3 g^{0k} g^{kj} \xi_j + \sum_{k=1}^3 g^{kj} \xi_j g^{ki} \xi_i \right). \label{eq:magnitudeDerivXSquared}
            \end{align}
            Before computing $\frac{\p^2}{\p s^2} x_s^+$, we make some preliminary computations. Again by \eqref{eq:flowdef} and \eqref{eq:bdef} and computing directly 
            \begin{align}
                \left(\frac{\p}{\p s} \xi_s^+(\omega) \right)_k = \p_{x_k} b^+(x_s,\xi_s)&= \p_k g^{0j} \xi_j + (b^+-b^-)^{-1}(2 g^{0j} (\p_{k}g^{0j}) \xi_j^2 + (\p_{k} g^{ij}) \xi_i \xi_j)\\
                &= (b^+-b^-)^{-1} (2 \tau (\p_k g^{0j}) \xi_j  + (\p_k g^{ij}) \xi_i \xi_j). \label{eq:magnitudeDerivXi}
            \end{align}
            A direct computation and \eqref{eq:magnitudeDerivX} give
            \begin{align}\label{eq:magnitudeDerivG}
                \frac{\p}{\p s}g^{\alpha \beta}(x_s) = \p_i g^{\alpha \beta} \left( \frac{\p}{\p s} x_s^+ \right)_i = -2 (b^+-b^-)^{-1} \p_i g^{\alpha \beta} (g^{0i} \tau + g^{ij} \xi_j).
            \end{align}
            As a final preliminary, by \eqref{eq:bdef}, using that $\tau$ is constant, and applying \eqref{eq:magnitudeDerivXi} and \eqref{eq:magnitudeDerivG} we have
            \begin{align}
                \frac{\p}{\p s} (b^+- b^-) &= \frac{\p}{\p s} 2 (\tau - g^{0l} \xi_l) = -2 \left( \xi_l \frac{\p}{\p s} g^{0l} + g^{0l} \frac{\p}{\p s} \xi_l\right) \\
                &= - 2(b^+ - b^-)^{-1} \bigg(-2 \xi_l \p_i g^{0l}(g^{0i}\tau + g^{ij} \xi_j)   + 
                g^{0l}( (\p_l g^{0j})2\tau \xi_j + (\p_l g^{ij})\xi_i \xi_j
                \bigg).\label{eq:magnitudeDerivB}
            \end{align}
            So now computing directly and applying \eqref{eq:magnitudeDerivX}, \eqref{eq:magnitudeDerivXi}, \eqref{eq:magnitudeDerivG}, and \eqref{eq:magnitudeDerivB} we obtain
            \begin{align}
            x_s^+(\omega) &\cdot\frac{\p^2}{\p s^2} x_s^+(\omega) = -2 x_k \frac{\p}{\p s} \bigg( (b^+ -b^-)^{-1} (g^{0k} \tau + g^{kj} \xi_j) \bigg) \\
            = -2 x_k  \bigg( &-(b^+-b^-)^{-2} (g^{0k} \tau + g^{kj} \xi_j) \frac{\p}{\p s}(b^+ -b^-) + (b^+-b^-)^{-1} \left(\tau \frac{\p}{\p s} g^{0k}  + \xi_j \frac{\p}{\p s} g^{kj}  + g^{kj} \frac{\p}{\p s} \xi_j \right) \bigg)\\
            =-2x_k\bigg( &2(b^+-b^-)^{-3} (g^{0k} \tau + g^{kj} \xi_j)(-2\xi_l (g^{0i}\tau + g^{ij} \xi_j)\p_i g^{0l} + g^{0l}( 2 \tau\xi_j  \p_l g^{0j} +  \xi_i \xi_j \p_l g^{ij})) \\
            &+(b^+-b^-)^{-2}(-2 \tau (g^{0i} \tau + g^{il} \xi_l)  \p_i g^{0k} -2\xi_j  (g^{0i} \tau + g^{il} \xi_l)\p_i g^{kj} +g^{kj}(2\tau \xi_l \p_j g^{0l} + \xi_i \xi_l \p_j g^{il}) ) \bigg).
            \label{eq:magnitudeDerivSecondx}
            \end{align}
            Now note that 
            \begin{equation}
                \tau = b^{+}(x_0,\xi_0) = b^{+}(x_s,\xi_s) \simeq |\xi_s|, \quad (b^+-b^-) \simeq |\xi_s|. \label{eq:magnitudeDerivPrelim}
            \end{equation}
            Thus every term in \eqref{eq:magnitudeDerivXSquared} and \eqref{eq:magnitudeDerivSecondx} has an equal number of powers of $\tau$ (or $|\xi|$) in its numerator and denominator. 
            So now combining \eqref{eq:magnitudeDerivXSquared}, \eqref{eq:magnitudeDerivSecondx}, and \eqref{eq:magnitudeDerivPrelim} to estimate terms in \eqref{eq:magnitudeDerivMinkowski}, since $|x_s^{\pm}(\omega)| > R_0$, we have 
            \begin{equation}
                \frac{\p^2}{\p s^2} |x_s^+(\omega)|^2 \gtrsim (1- \|g-m\|_{AF_{>R_0}}).
            \end{equation}
             By asymptotic flatness, we have that $1-\|g-m\|_{AF_{>R_0}} \gtrsim c>0$. Thus,
            \begin{equation}
                \frac{\partial^2}{\partial s^2}\left|x^+_s(\omega)\right|^2 \geq c.
            \end{equation}
            whenever $|x_s^+(\omega)|>R_0$. 
        \end{proof}

    \begin{lemma}\label{l:easyGronwall}(Reverse Gr\"onwall inequality) 
    Assume $\eta(\cdot)$ is a nonnegative absolutely continuous function on $[0,T]$ and for almost every $t \in [0,T]$
    \begin{equation}
        \eta'(t) \geq  - \psi(t) - C \eta(t),
    \end{equation}
    where $C \geq 0$ and $\psi(t)$ is a nonnegative function with $\psi \in L^1[0,T]$. Then for all $t \in [0,T]$
    \begin{equation}
        \eta(t) \leq e^{C(T-t)} \left(\int_t^T \psi(s) ds+\eta(T)\right).
    \end{equation}
    \end{lemma}
    \begin{proof}
        For almost every $s \in [0,T]$
        \begin{equation}
            \frac{d}{ds} ( \eta(s) e^{Cs}) = (\eta'(s)+C \eta(s)) e^{Cs} \geq -e^{Cs}\psi(s).
        \end{equation}
        Therefore integrating both sides from $s=t$ to $s=T$ we obtain
        \begin{equation}
            \eta(T) e^{C T} - \eta(t) e^{C t} \geq \int_t^T -e^{C s} \psi(s) ds.
        \end{equation}
        Rearranging we have 
        \begin{equation}
            \eta(t) \leq e^{-C t} \left( \int_{t}^T e^{C s}\psi(s) ds + e^{C T} \eta(T) \right).
        \end{equation}
        Controlling $e^{C  s} \leq e^{C T}$ gives the desired inequality.
    \end{proof}

\section{Appendix: Pseudodifferential operators}\label{s:pseudo}
In this appendix we compile some fundamental definitions and results on pseudodifferential operators that we use in the proof. For details on the homogeneous psuedodifferential calculus see \cite[Chapter 18]{Hormander3}  or \cite{TaylorPseudos}. For details on the semiclassical pseudodifferential calculus see \cite{Zworski2012} or \cite[Appendix E]{DyatlovZworski2020}.

\begin{definition}\label{def:Symbol}
    Let $m \in \Rb$. We define the Kohn-Nirenberg symbol class $S^m(\Rb^n)=S^m$ to consist of the set of $p \in C^{\infty}(T^* \Rb^n)$ such that for any multi-indices $\alpha, \beta$, there exists a constant $C_{\alpha,\beta}$ such that 
    \begin{equation}
        |D_z^{\beta} D_{\zeta}^{\alpha} p(z,\zeta)| \leq C_{\alpha, \beta} (1+|\zeta|)^{m-|\alpha|} \quad \forall (z,\zeta) \in T^* \Rb^n.
    \end{equation}
\end{definition}

\subsection{Homogeneous Calculus}
We make use of the Homogeneous Calculus on $(\Rb^4,g)$. First we define the Weyl quantization and pseudodifferential operators. 
\begin{definition}\label{def:Pseudo}
For $a \in S^m(\Rb^4)$ we define the operator
    \begin{equation}
        a^w(z,D) u(z) = \Op^w(a)u(z)= (2\pi)^{-4} \int_{\Rb^4} \int_{\Rb^4} a\left( \frac{z+w}{2}, \zeta \right) e^{i\<z-w, \zeta\>} u(w) dw d \zeta.
    \end{equation}
    We define $\Psi^k(\Rb^4)$ to be the image of $S^k(\Rb^4)$ under $\Op^w$.
\end{definition}
Note that if $a(z,\zeta) =a(z)$, then $a^w(z,D)=a(z)$. On the other hand, if $a(z,\zeta)=a(\zeta)$, then $a^w(z,D)$ is just the Fourier multiplier with kernel $a(\zeta)$.

Next, we have that based on their order, pseudodifferential operators are bounded on Sobolev spaces.
\begin{proposition}(Calderon-Vaillancourt Theorem)\label{prop:pseudoBounded}
    If $a\in S^k(\Rb^4)$, then $\Op^w(a)$ is bounded as an operator from $H^k_{t,x}$ to $L^2_{t,x}$. Furthermore, there exists $C>0,N=N(k) \in \Nb$ such that the operator norm is bounded 
    \begin{equation}
        \nm{ \Op^w(a)}_{H^k_{t,x} \ra L^2_{t,x}} \leq C \sup_{|\alpha|, |\beta| \leq N} C_{\alpha \beta},
    \end{equation}
    where the supremum is taken over multi-indices $\alpha,\beta$ and $C_{\alpha \beta}$ is the constant from the symbol estimates in Definition \ref{def:Symbol}.
\end{proposition}
An important property of the Weyl quantization is that real symbols are quantized into self adjoint operators. We state the more general version of this property here.
\begin{proposition}\label{prop:pseudoAdjoint}
If $a \in S^k(\Rb^4)$, then $\Op^w(a)^* = \Op^w(\overline{a})$.
\end{proposition}
Next, we have expansion formulas for evaluating compositions and commutators of pseudodifferential operators. 
\begin{proposition}\label{prop:pseudoCalc}
\begin{enumerate}
    \item 
    Let $a_j \in S^{m_j}, j=1,2$, then there exists $b \in S^{m_1+m_2}$ such that 
    \begin{equation}
        a_1^w(z,D) a_2^w(z,D) = b^w(z,D).
    \end{equation}
    Furthermore, for any $N \in \Nb$ there exists $r_{N} \in S^{m_1+m_2-N}$ such that 
    \begin{equation}
        b(z,\zeta) = \sum_{k=0}^{N-1} \frac{i^k}{k!} (\p_w \p_{\zeta} - \p_z \p_{\rho})^k a_1(z,\zeta) a_2(w, \rho) \bigg|_{\rho = \zeta, w=z} + r_N.
    \end{equation}
    \item Let $\{f,g\}= \p_{\zeta} f \p_z g - \p_z f \p_{\zeta}g $ be the Poisson bracket. There exists $r_2 \in S^{m_1+m_2-2}$ such that 
    \begin{equation}
        b = a_1 a_2 - \frac{i}{2}\{a_1, a_2\} + r_{2}.
    \end{equation}
    Furthermore, there exists $r_3 \in S^{m_1+m_2-3}$ such that 
    \begin{align}
        [a_1^w(z,D), a_2^w(z,D)] :&= a_1^w(z,D) a_2^w(z,D) - a_2^w(z,D)a_1^w(z,D)\\
        &= -i \{a_1, a_2\}^w(z,D) + r_{3}^w(z,D).
    \end{align}
    Similarly, there exists $r_2 \in S^{m_1+m_2-2}$ such that
    \begin{align}
        a_1^w(z,D) a_2^w(z,D)+a_2^w(z,D) a_1^w(z,D) = 2(a_1 a_2)^w(z,D) + r_2^w(z,D).
    \end{align}
    \item If $\supp(a_1) \cap \supp(a_2)=\emptyset$, then $b \in S^{-\infty}$.
\end{enumerate}
\end{proposition}
Finally, we state a way to convert a lower bound on a symbol to a lower bound for its quantization.
\begin{proposition}(Sharp G\r{a}rding inequality) \label{prop:Garding}
    Consider $m \in \Rb$, if $a \in S^{2m+1}(\Rb^4)$ and $a \geq 0$, then there exists a constant $C>0$, depending on the constants $C_{\alpha \beta}$ from Definition \ref{def:Symbol}, such that 
    \begin{equation}
        \<\Op^w(a) u, u \>_{L^2_{t,x}} \geq - C \nm{u}_{H^m_{t,x}}.
    \end{equation}
\end{proposition}

\subsection{Semiclassical Calculus}
We first define the semiclassical Weyl quantization and semiclassical pseudodifferential operators. 
\begin{definition}\label{def:semiPseudo}
For $a \in S^m(\Rb^n)$ we define the operator
    \begin{equation}
        a^w(z,hD)=\Ophw(a)u(z)= (2\pi h)^{-n} \int_{\Rb^n} \int_{\Rb^n} a\left( \frac{z+w}{2}, \zeta \right) e^{\frac{i}{h}\<z-w, \zeta\>} u(w) dw d \zeta.
    \end{equation}
    We define $\Psi_h^k(\Rb^n)$ to the image of $S^k(\Rb^n)$ under $\Ophw$.
\end{definition}
Note that if $a(z,\zeta) =a(z)$, then $\Ophw(a)=a(z)$. On the other hand, if $a(z,\zeta)=a(\zeta)$, then $\Ophw(a)$ is just the \emph{semiclassical} Fourier multiplier with kernel $a(\zeta)$.

Next we have an $L^2$ boundedness result for $0$th order semiclassical pseudodifferential operators. 
\begin{proposition}(Semiclassical Calderon-Vaillancourt Theorem)\label{prop:semipseudoBounded}
    If $a\in S^0(\Rb^n)$, then $\Ophw(a)$ is bounded as an operator from $L^2_{t,x}$ to $L^2_{t,x}$. Furthermore, there exists $C>0,N=N(k) \in \Nb$ such that the operator norm is bounded 
    \begin{equation}
        \nm{ \Op^w(a)}_{L^2_{t,x} \ra L^2_{t,x}} \leq C \sup_{|\alpha|, |\beta| \leq N} C_{\alpha \beta},
    \end{equation}
    where the supremum is taken over multi-indices $\alpha,\beta$ and $C_{\alpha \beta}$ is the constant from the symbol estimates in Definition \ref{def:Symbol}.
\end{proposition}
Finally we have formulas for compositions and commutators of semiclassical pseudodifferential operators. 
\begin{proposition}\label{prop:semipseudoCalc}
\begin{enumerate}
    \item 
    Let $a_j \in S^{0}(\Rb^n), j=1,2$, then there exists $b \in S^{0}(\Rb^n)$ such that 
    \begin{equation}
        a_1^w(z,D) a_2^w(z,D) = b^w(z,D).
    \end{equation}
    Furthermore, for any $N \in \Nb$ there exists $r_{N} \in S^{-N}$ such that 
    \begin{equation}
        b(z,\zeta) = \sum_{k=0}^{N-1} \frac{i^k h^k }{k!} (\p_w \p_{\zeta} - \p_z \p_{\rho})^k a_1(z,\zeta) a_2(w, \rho) \bigg|_{\rho = \zeta, w=z} + h^{-N} r_N.
    \end{equation}
    \item Let $\{f,g\}= \p_{\zeta} f \p_z g - \p_z f \p_{\zeta}g $ be the Poisson bracket. Then 
    \begin{equation}
        b = a_1 a_2 - \frac{ih}{2}\{a_1, a_2\} + h^2 r_{2}, \quad r_2 \in S^{-2},
    \end{equation}
    and 
    \begin{align}
        [\Ophw(a_1), \Ophw(a_2)] :&= \Ophw(a_1) \Ophw(a_2) - \Ophw(a_2)\Ophw(a_1)\\
        &= -ih \Ophw(\{a_1, a_2\}) + h^3\Ophw(r_{3}), \quad r_3 \in S^{-3}.
    \end{align}
    Furthermore, for multi-indices $\alpha$
    \begin{equation}
        \nm{\Ophw(r_{3})}_{L^2 \ra L^2} \leq \sup_{|\alpha| \leq 4n+1} \nm{\nabla_{z,w,\zeta,\rho}^{\alpha}(\p_w \p_{\zeta}-\p_z \p_{\rho})^N(a(z,\zeta) b(w,\rho))}_{L^{\infty}(\Rb^{4n})}.
    \end{equation}
\end{enumerate}
\end{proposition}
A proof of the final statement is in \cite[Lemma A.5]{Kleinhenz2023Torus}.

\bibliographystyle{alpha}
\bibliography{mybib}

\begin{thebibliography}{LRLTT17}

\bibitem[AK02]{AlouiKhenissi2002}
L.~Aloui and M.~Khenissi.
\newblock Stabilisation pour l’equation des ondes dans un domaine
  ext{\'e}rieur.
\newblock {\em Rev. Mat. Iberoamericana}, 18:1--16, 2002.

\bibitem[Ali06]{Alinhac2006}
S.~Alinhac.
\newblock On the morawetz--keel--smith--sogge inequality for the wave equation
  on a curved background.
\newblock {\em Publications of the Research Institute for Mathematical
  Sciences}, 42(3):705--720, 2006.

\bibitem[BCMP19]{BCMP17}
R.~Booth, H.~Christianson, J.~Metcalfe, and J.~Perry.
\newblock Localized energy for wave equations with degenerate trapping.
\newblock {\em Mathematical Research Letters}, 26(4):991--1025, 2019.

\bibitem[BH09]{BH2009}
J.F. Bony and D.~H{\"a}fner.
\newblock The semilinear wave equation on asymptotically euclidean manifolds.
\newblock {\em Communications in Partial Differential Equations}, 35(1):23--67,
  2009.

\bibitem[BJ16]{BurqJoly2016}
N.~Burq and R.~Joly.
\newblock Exponential decay for the damped wave equation in unbounded domains.
\newblock {\em Communications in Contemporary Mathematics}, 18(06):1650012,
  2016.

\bibitem[BR14]{BoucletRoyer2014}
J.M. Bouclet and J.~Royer.
\newblock Local energy decay for the damped wave equation.
\newblock {\em Journal of Functional Analysis}, 266(7):4538--4615, 2014.

\bibitem[DZ19]{DyatlovZworski2020}
S.~Dyatlov and M.~Zworski.
\newblock {\em Mathematical theory of scattering resonances}, volume 200.
\newblock American Mathematical Soc., 2019.

\bibitem[H{\"o}r07]{Hormander3}
L.~H{\"o}rmander.
\newblock {\em The analysis of linear partial differential operators III:
  Pseudo-differential operators}.
\newblock Springer Science \& Business Media, 2007.

\bibitem[Kle22]{Kleinhenz2022a}
P.~Kleinhenz.
\newblock Energy decay for the time dependent damped wave equation.
\newblock {\em arXiv preprint arXiv:2207.06260}, 2022.

\bibitem[Kle23]{Kleinhenz2023Torus}
P.~Kleinhenz.
\newblock Decay rates for the damped wave equation with finite regularity
  damping.
\newblock {\em Mathematical Research Letters}, 29(4):1087--1140, 2023.

\bibitem[Kle25]{kleinhenz2025sharp}
P.~Kleinhenz.
\newblock Sharp conditions for exponential and non-exponential uniform
  stabilization of the time-dependent damped wave equation.
\newblock {\em Transactions of the American Mathematical Society}, 2025.

\bibitem[Kof23a]{KofrothMagnetic}
C.~Kofroth.
\newblock Integrated local energy decay for damped magnetic wave equations on
  stationary space-times.
\newblock {\em arXiv preprint arXiv:2311.08628}, 2023.

\bibitem[Kof23b]{Kofroth23}
C.~Kofroth.
\newblock Integrated local energy decay for the damped wave equation on
  stationary space-times.
\newblock {\em SIAM Journal on Mathematical Analysis}, 55(5):5086--5126, 2023.

\bibitem[KPV95]{KPV95}
C.~Kenig, G.~Ponce, and L.~Vega.
\newblock On the zakharov and zakharov-schulman systems.
\newblock {\em Journal of Functional Analysis}, 127(1):204--234, 1995.

\bibitem[KSS02]{KSS02}
M.~Keel, H.~Smith, and C.~Sogge.
\newblock Almost global existence for some semilinear wave equations.
  {D}edicated to the memory of {T}homas {H}. {W}olff.
\newblock {\em Journal d'Analyse Math{\'e}matique}, 87(1):265--279, 2002.

\bibitem[Leb96]{Lebeau1996}
G.~Lebeau.
\newblock Equation des ondes amorties.
\newblock In {\em Algebraic and Geometric Methods in Mathematical Physics:
  Proceedings of the Kaciveli Summer School, Crimea, Ukraine, 1993}, pages
  73--109. Springer Netherlands, Dordrecht, 1996.

\bibitem[Loo22a]{Looi22b}
S.~Looi.
\newblock Decay rates for cubic and higher order nonlinear wave equations on
  asymptotically flat spacetimes.
\newblock {\em arXiv preprint arXiv:2207.10280}, 2022.

\bibitem[Loo22b]{Looi22c}
S.~Looi.
\newblock Improved decay for quasilinear wave equations close to asymptotically
  flat spacetimes including black hole spacetimes.
\newblock {\em arXiv preprint arXiv:2208.05439}, 2022.

\bibitem[Loo22c]{Looi22a}
S.~Looi.
\newblock Pointwise decay for the energy-critical nonlinear wave equation.
\newblock {\em arXiv preprint arXiv:2205.13197}, 2022.

\bibitem[Loo23]{Looi23}
S.~Looi.
\newblock Pointwise decay for the wave equation on nonstationary spacetimes.
\newblock {\em Journal of Mathematical Analysis and Applications},
  527(1):126939, 2023.

\bibitem[LRLTT17]{LRLTT17}
J.~Le~Rousseau, G.~Lebeau, P.~Terpolilli, and E.~Tr{\'e}lat.
\newblock Geometric control condition for the wave equation with a
  time-dependent observation domain.
\newblock {\em Analysis \& PDE}, 10(4):983--1015, 2017.

\bibitem[LT20]{LT20}
H.~Lindblad and M.~Tohaneanu.
\newblock A local energy estimate for wave equations on metrics asymptotically
  close to {K}err.
\newblock {\em Annales Henri Poincar{\'e}}, 21(11):3659--3726, 2020.

\bibitem[LT25]{LooiTohaneanu2025}
S.~Looi and M.~Tohaneanu.
\newblock Global existence and pointwise decay for nonlinear waves under the
  null condition.
\newblock {\em arXiv preprint arXiv:2204.03626}, 2025.

\bibitem[MMT08]{MMT08}
J.~Marzuola, J.~Metcalfe, and D.~Tataru.
\newblock Strichartz estimates and local smoothing estimates for asymptotically
  flat schr{\"o}dinger equations.
\newblock {\em Journal of Functional Analysis}, 255(6):1497--1553, 2008.

\bibitem[MMTT10]{MMTT10}
J.~Marzuola, J.~Metcalfe, D.~Tataru, and M.~Tohaneanu.
\newblock Strichartz estimates on schwarzschild black hole backgrounds.
\newblock {\em Communications in Mathematical Physics}, 293(1):37--83, 2010.

\bibitem[Mor66]{Mor66}
C.S. Morawetz.
\newblock Exponential decay of solutions of the wave equation.
\newblock {\em Communications on Pure and Applied Mathematics}, 19(4):439--444,
  1966.

\bibitem[Mor68]{Mor68}
C.S. Morawetz.
\newblock Time decay for the nonlinear klein-gordon equation.
\newblock {\em Proceedings of the Royal Society of London. Series A.
  Mathematical and physical sciences}, 306(1486):291--296, 1968.

\bibitem[Mor75]{Mor75}
C.S. Morawetz.
\newblock Decay for solutions of the exterior problem for the wave equation.
\newblock {\em Communications on Pure and Applied Mathematics}, 28(2):229--264,
  1975.

\bibitem[Mor24]{Morgan24}
K.~Morgan.
\newblock The effect of metric behavior at spatial infinity on pointwise wave
  decay in the asymptotically flat stationary setting.
\newblock {\em American Journal of Mathematics}, 146(1):47--105, 2024.

\bibitem[MRS77]{MRS77}
C.S. Morawetz, J.V. Ralston, and W.A. Strauss.
\newblock Decay of solutions of the wave equation outside nontrapping
  obstacles.
\newblock {\em Communications on Pure and Applied Mathematics}, 30(4):447--508,
  1977.

\bibitem[MS06]{MS2006}
S.~Metcalfe and C.~Sogge.
\newblock Long-time existence of quasilinear wave equations exterior to
  star-shaped obstacles via energy methods.
\newblock {\em SIAM journal on mathematical analysis}, 38(1):188--209, 2006.

\bibitem[MS07]{MS2007}
Jason Metcalfe and Christopher~D Sogge.
\newblock Global existence of null-form wave equations in exterior domains.
\newblock {\em Mathematische Zeitschrift}, 256(3):521--549, 2007.

\bibitem[MST20]{MST20}
J.~Metcalfe, J.~Sterbenz, and D.~Tataru.
\newblock Local energy decay for scalar fields on time dependent non-trapping
  backgrounds.
\newblock {\em American Journal of Mathematics}, 142(3):821--883, 2020.

\bibitem[MT09]{MT2009}
J.~Metcalfe and D.~Tataru.
\newblock {\em Decay Estimates for Variable Coefficient Wave Equations in
  Exterior Domains}, pages 201--216.
\newblock Birkh{\"a}user Boston, Boston, 2009.

\bibitem[MT12]{MT12}
J.~Metcalfe and D.~Tataru.
\newblock Global parametrices and dispersive estimates for variable coefficient
  wave equations.
\newblock {\em Mathematische Annalen}, 353(4):1183--1237, 2012.

\bibitem[MTT12]{MTT12}
J.~Metcalfe, D.~Tataru, and M.~Tohaneanu.
\newblock Price’s law on nonstationary space--times.
\newblock {\em Advances in Mathematics}, 230(3):995--1028, 2012.

\bibitem[MTT17]{MTT17}
J.~Metcalfe, D.~Tataru, and M.~Tohaneanu.
\newblock Pointwise decay for the maxwell field on black hole space--times.
\newblock {\em Advances in Mathematics}, 316:53--93, 2017.

\bibitem[MW21]{MW21}
K.~Morgan and J.~Wunsch.
\newblock Generalized price's law on fractional-order asymptotically flat
  stationary spacetimes.
\newblock {\em arXiv preprint arXiv:2105.02305}, 2021.

\bibitem[NZ09]{NZ2009}
S.~Nonnenmacher and M.~Zworski.
\newblock Semiclassical resolvent estimates in chaotic scattering.
\newblock {\em Applied Mathematics Research eXpress}, 2009(1):74--86, 2009.

\bibitem[Ral69]{Ralston69}
J.V. Ralston.
\newblock Solutions of the wave equation with localized energy.
\newblock {\em Communications on Pure and Applied Mathematics}, 22(6):807--823,
  1969.

\bibitem[RT74]{RauchTaylor1974}
J.~Rauch and M.~Taylor.
\newblock Exponential decay of solutions to hyperbolic equations in bounded
  domains.
\newblock {\em Indiana university Mathematics journal}, 24(1):79--86, 1974.

\bibitem[Sbi15]{Sbierski15}
J.~Sbierski.
\newblock Characterisation of the energy of gaussian beams on lorentzian
  manifolds: with applications to black hole spacetimes.
\newblock {\em Analysis \& PDE}, 8(6):1379--1420, 2015.

\bibitem[SS00]{SmithSogge2000}
H.~Smith and C.~Sogge.
\newblock Global strichartz estimates for nonthapping perturbations of the
  laplacian: Estimates for nonthapping perturbations.
\newblock {\em Communications in Partial Differential Equations},
  25(11-12):2171--2183, 2000.

\bibitem[Ste05]{Sterbenz2005}
J.~Sterbenz.
\newblock Angular regularity and strichartz estimates for the wave equation.
\newblock {\em International Mathematics Research Notices}, 2005(4):187--231,
  2005.

\bibitem[Str75]{Strauss75}
W.~Strauss.
\newblock Dispersal of waves vanishing on the boundary of an exterior domain.
\newblock {\em Communications on Pure and Applied Mathematics}, 28(2):265--278,
  1975.

\bibitem[Tat13]{Tataru13}
D.~Tataru.
\newblock Local decay of waves on asymptotically flat stationary space-times.
\newblock {\em American Journal of Mathematics}, 135(2):361--401, 2013.

\bibitem[Tay06]{TaylorPseudos}
M.~Taylor.
\newblock Pseudo differential operators.
\newblock {\em Pseudo Differential Operators}, 2006.

\bibitem[Toh12]{Tohaneanu12}
M.~Tohaneanu.
\newblock Strichartz estimates on kerr black hole backgrounds.
\newblock {\em Transactions of the American Mathematical Society},
  364(2):689--702, 2012.

\bibitem[WZ11]{WZ2011}
J.~Wunsch and M.~Zworski.
\newblock Resolvent estimates for normally hyperbolic trapped sets.
\newblock {\em Annales Henri Poincar{\'e}}, 12(7):1349--1385, 2011.

\bibitem[Zwo12]{Zworski2012}
M.~Zworski.
\newblock {\em Semiclassical analysis}, volume 138.
\newblock American Mathematical Soc., 2012.

\end{thebibliography}
\end{document}